\documentclass[a4paper,11pt]{amsart}
\usepackage[utf8]{inputenc}
\usepackage[T1]{fontenc}
\usepackage[english]{babel}
\usepackage{textcomp}
\usepackage{lmodern}
\usepackage{amsthm}
\usepackage{amsmath}
\usepackage{amssymb}
\usepackage{mathrsfs}
\usepackage{enumerate}
\usepackage{dsfont}
\usepackage{color}
\usepackage{graphicx}
\usepackage{accents}
\usepackage[hidelinks]{hyperref}
\usepackage{float}
\usepackage[titletoc]{appendix}
\usepackage[scale = .75]{geometry}
\newtheorem{theo}{Theorem}[section]
\newtheorem{prop}[theo]{Proposition}

\newtheorem{lem}[theo]{Lemma}
\newtheorem{rema}[theo]{Remark}
\newtheorem{exam}[theo]{Example}

\theoremstyle{definition}
\newtheorem{defi}[theo]{Definition}

\makeatletter
\newcommand\ds{ \displaystyle }

\newcommand\NN{\mathbb{N}}
\newcommand\RR{\mathbb{R}}

\newcommand\bA{{\mathbf A}}
\newcommand\bB{{\mathbf B}}

\newcommand\bE{{\mathbf E}}
\newcommand\bF{{\mathbf F}}

\newcommand\bU{{\mathbf U}}

\newcommand\bk{{\mathbf k}}

\newcommand\bn{{\mathbf n}}

\newcommand\bu{{\mathbf u}}
\newcommand\bv{{\mathbf v}}

\newcommand\bx{{\mathbf x}}
\newcommand\by{{\mathbf y}}

\newcommand{\T}{\mathcal{T}}
\newcommand{\E}[1][]{\mathcal{E}_{#1}}

\newcommand{\Ei}[1][]{\mathcal{E}_{\text{int}#1}}
\newcommand{\Ee}[1][]{\mathcal{E}_{\text{ext}#1}}
\newcommand{\X}[1][]{X_{#1}}
\newcommand{\R}{\mathbb{R}}
\newcommand{\F}[1][]{F_{#1,\sigma}}
\newcommand{\C}{\mathcal{C}}

\newcommand{\Pnt}{\mathcal{P}}
\newcommand{\D}{\mathcal{D}}

\newcommand{\dx}{\,\mathrm{d}\mathbf{x}}

\newcommand{\dd}{\,\mathrm{d}}

\newcommand{\finf}{f^\infty}
\makeatother

\title{A finite volume scheme for boundary-driven convection-diffusion equations with
  relative entropy structure}

\author{Francis FILBET}
\address{\flushleft
  Francis FILBET\\\vspace{.2cm}
  Institut de Mathématiques de Toulouse,\\
  Université Toulouse III \& Institut Universitaire de France,\\
  Bâtiment 1R3, 118, route de Narbonne\\
  F-31062, Toulouse cedex 9 France\\\vspace{.2cm}
  \texttt{E-mail: \href{mailto:francis.filbet@math.univ-toulouse.fr}{francis.filbet@math.univ-toulouse.fr}}}

\author{Maxime HERDA}
\address{\flushright
  Maxime HERDA\\\vspace{.2cm}
  Institut Camille Jordan,\\
  Université Claude Bernard Lyon 1, CNRS UMR 5208,\\
  43 blvd. du 11 novembre 1918\\\flushright
  F-69622, Villeurbanne cedex France\\\vspace{.2cm}
  \texttt{E-mail: \href{mailto:herda@math.univ-lyon1.fr}{herda@math.univ-lyon1.fr}}\\}

\begin{document}
    
    \begin{abstract}
      We propose a finite volume scheme for a class of nonlinear parabolic
      equations endowed with non-homogeneous Dirichlet boundary conditions and
      which admit relative entropy functionals. For this kind of models including 
      porous media equations, Fokker-Planck equations for plasma physics or
      dumbbell models for polymer flows, it has been proved that the transient solution converges
      to a steady-state when time goes to infinity. The present scheme is built from a discretization of the steady equation
      and preserves steady-states and natural
      Lyapunov functionals which provide a satisfying long-time behavior.
      After proving well-posedness, stability, exponential return to equilibrium and convergence,
      we present several numerical results which 
      confirm the accuracy  and underline the efficiency to preserve large-time asymptotic.

      
      \medskip
      \noindent \textsc{MSC2010 subject classifications.} 65M08, 65M12, 76S05, 76X05, 82D60
    \end{abstract}
    \maketitle
    \tableofcontents
    \section{Introduction}
    
    \subsection{General Setting}
    
    Let  $\Omega$ be a polyhedral open bounded connected subset of $\R^d$ with boundary $\Gamma = \partial\Omega$. 
    Let us introduce the steady advection field ${\bE}:\Omega\rightarrow\R^d$ and $\eta:\R\rightarrow\R$ a strictly increasing smooth function onto $\R$ satisfying $\eta(0)=0$.
    We consider the following nonlinear convection-diffusion equation with non-homogeneous Dirichlet boundary conditions
    \begin{equation}
      \left\lbrace
      \begin{array}{ll}
	\ds \frac{\partial f}{\partial t} \,+\, {\nabla}\cdot\left({\bE}(\bx)\,\eta(f)-{\nabla} \eta(f)\right) \,=\, 0&\text{ for }\ \bx\in\Omega, \  t\geq0,
	\\
	\,
	\\
	\ds f(t,\bx) = f^b(\bx) &\text{ for }\ \bx\in\Gamma, \  t\geq0,
	\\ 
	\,
	\\
	\ds f(0,\bx) = f^\text{in}(\bx) &\text{ for }\ \bx\in\Omega.
      \end{array}
      \right.
      \label{contprob}
    \end{equation}
    In \cite{bodineau_2014_lyapunov}, T. Bodineau, C. Villani, C. Mouhot and J. Lebowitz showed that this equation admits a large class of Lyapunov functionals,
    that we denote, using their denomination, relative $\phi$-entropies. 
    Each functional is generated by a convex function $\phi$ satisfying the following properties.
    \begin{defi}[Entropy generating functions]
      For any non-empty interval $J$ of $\R$ containing $1$, we say that $\phi\in\mathcal{C}^2(J,\R_+)$ is an \emph{entropy generating function} or simply
      \emph{entropy function} if it is strictly convex and satisfies $\phi(1)=0$ and $\phi'(1)=0$.
    \end{defi}
    The entropies are defined relatively to a steady state of \eqref{contprob}.
    Therefore, we assume that there exists $f^\infty$ which satisfies
    \begin{equation}
      \left\lbrace
      \begin{array}{ll}
	{\nabla}\cdot\left({\bE}\,\eta(f^\infty)-{\nabla} \eta(f^\infty)\right) = 0 &\text{in }\Omega,
	\\
	\,
	\\
	f^\infty = f^b &\text{on }\Gamma.
      \end{array}
      \right.
      \label{stateq}
    \end{equation}
    Now we can define the relative $\phi$-entropies and associated dissipations.
    \begin{defi}[Relative $\phi$-entropy and dissipation]
      For any entropy generating function $\phi$,
      we denote by $\mathcal{H}_\phi$ the \emph{relative $\phi$-entropy} defined by
      \[
      \mathcal{H}_\phi(t) = \int_\Omega\int_{f^\infty({\bx})}^{f(t,{\bx})}\phi'\left(\frac{\eta(s)}{\eta(f^\infty({\bx}))}\right)  \dd s \dx,
      \]
      and by $\mathcal{D}_\phi$ the relative $\phi$-entropy dissipation defined by
      \[
      \mathcal{D}_\phi(t) = \int_\Omega \left|{\nabla} h\right|^2\,\phi''(h)\,\eta(f^\infty) \dx,
      \]
      where $h$ is the ratio between the transient and stationary nonlinearities
      \begin{equation}
	h\ =\ 
      \left\lbrace
      \begin{array}{ll}
	\ds\frac{\eta(f)}{\eta(f^\infty)}&\text{in}\quad\Omega\,,\\
	\,\\
	\ds1 &\text{on}\quad\Gamma\,.
      \end{array}
      \right.
	\label{ratio}
      \end{equation}
      \label{relentcont}
    \end{defi}
    Let us note that for the linear problem, namely when $\eta$ is the identity function, the relative $\phi$-entropy rewrites 
    \[
    \mathcal{H}_\phi(t) = \int_\Omega\phi\left(\frac{f}{f^\infty}\right) f^\infty \dx.
    \]
    Typical examples of relative $\phi$-entropies are the \emph{physical relative entropy} and \emph{p-entropies}
    (or \emph{Tsallis relative entropies}) respectively generated, for $p\in(1,2]$, by
    \begin{equation}
      \left\{
      \begin{array}{lcl}
      \ds\phi_1(x)& =& x\ln(x) -(x-1)\,,\\
      \textstyle\,\\
      \ds\phi_p(x) &= &\ds\frac{x^p - px}{p-1} +1\,.
    \end{array}
    \right.
      \label{examples}
    \end{equation}
    
    One readily sees that, since $\eta$ and $\phi'$ are increasing functions satisfying $\eta(0) = 0$ and $\phi'(1) = 0$,
    the relative $\phi$-entropy is a non-negative quantity which cancels if and only if $f$ and $f^\infty$ coincide almost everywhere.
    The $\phi$-entropies are not, in general, distances between the solution and the steady state. However 
    Csiszar-Kullback type inequalities \cite{csiszar_1967_information, information_1959_kullback, generalized_2000_unterreiter} 
    yield a control of the $L^1$ distance between the solution and the equilibrium.
    Therefore if a relative $\phi$-entropy goes to zero when time goes to infinity,
    the solution converges to equilibrium in a strong sense.
    
    The following proposition was proved in \cite[Theorem 1.4]{bodineau_2014_lyapunov} and yields an entropy-entropy dissipation
    principle for Equation \eqref{contprob}. It starts from a reformulation of \eqref{contprob} using the new unknown \eqref{ratio}. It is easily derived using Leibniz product rule in Equation~\eqref{contprob} together with \eqref{stateq} and \eqref{ratio} and reads
    \begin{equation}
      \left\{
      \begin{array}{l}
	\ds \frac{\partial f}{\partial t} \,+\, {\nabla}\cdot\left(\bU^\infty\,h \,-\,  \eta(f^\infty)\nabla h\right) \,=\,0\,,\\
	\,\\
	\ds \bU^\infty\, =\, {\bE}\,\eta(f^\infty)- {\nabla} \eta(f^\infty)\,,\\
	\,\\
	\ds
	\nabla\cdot\bU^\infty\, =\, 0\,.
      \end{array}\right.
      \label{contprob2}
    \end{equation}
    \begin{prop}Any $L^\infty$ solution of \eqref{contprob} satisfies in the sense of distributions
      \begin{equation}
	\frac{\dd \mathcal{H}_\phi}{\dd t}= -\mathcal{D}_\phi\leq0,
	\label{inegcont}
      \end{equation}
      for any entropy generating function $\phi$.
      \label{theobodineau}
    \end{prop}
    The formal computations leading to \eqref{inegcont} motivate our choices in the elaboration of the discrete scheme.
    Therefore, we recall the proof yielding the entropy equality.
    \begin{proof}
      First, we integrate \eqref{contprob2} against $\phi'(h)$, integrate by parts and use the boundary conditions and the fact that $\phi'(1)=0$ to get
      \begin{eqnarray*}
	\frac{\dd \mathcal{H}_\phi}{\dd t} &=&
	\int_\Omega\phi'\left(h\right){\nabla}\cdot\left( - \bU^\infty\,h\, +\, \eta(f^\infty)\,{\nabla} h
	\right)\,\dx
	\\
	&=& \int_\Omega\bU^\infty\cdot{\nabla}h\,\phi''\left(h\right)\,h\,\dx\, -\, 
	\int_\Omega  \left|{\nabla} h\right|^2\,\phi''(h)\,\eta(f^\infty)\,\dx.
      \end{eqnarray*}
      Let $\varphi:s\mapsto  s\phi'(s) - \phi(s)$. It satisfies $\varphi'(s) = \phi''\left(s\right)s$ and $\varphi(1) = 0$. Hence, substituting it in the last expression yields
      \[
      \frac{\dd \mathcal{H}_\phi}{\dd t}\ =\ \int_\Omega\bU^\infty\cdot{\nabla}\varphi(h)\,\dx\, -\, \mathcal{D}_\phi\ =\  - \mathcal{D}_\phi,
      \]
      where we integrated the first term by parts, used the stationary equation \eqref{stateq} and the boundary conditions.
    \end{proof}
    
    There are two important facts that justify the use of \eqref{contprob2} instead of \eqref{contprob} to derive the above 
    entropy dissipation inequality.
    The rewriting transforms the advection field $\bE$ on $\eta(f)$
    into the incompressible field $\bU^\infty$ on $h$ so that the contribution of the convection can vanish when
    the time derivative of the relative entropy  is computed.
    The underlying cancellations stems from the transformation of  ${\nabla}h \,\phi''\left(h\right)h$ into a gradient
    thanks to $\varphi$ and on the fact that $f^\infty$ solves \eqref{stateq}.
    The second reason is that considering the equation on $h$ instead of 
    $f$ changes non-homogeneous Dirichlet boundary conditions into homogeneous ones on $h-1$.
    Together with properties of $\phi$, it enables cancellations of boundary terms. 
    In other words, relative $\phi$-entropies are the correct functionals and \eqref{contprob2} the right form of the equation to capture the boundary-driven dynamics.
    
    \smallskip
    
    The purpose of this work  is the design and analysis of a finite volume scheme preserving the whole class of relative $\phi$-entropy dissipation inequalities
    \eqref{inegcont}. This is done by discretizing the reformulated equation \eqref{contprob2} in a way that enables the counterpart of computations of the proof of 
    Proposition \ref{theobodineau} to hold at the discrete level. Because of the reformulation involving the steady state, the scheme is based on a preliminary discretization of the steady state and flux.
    
    \smallskip
    
    Let us emphasize that $\bE$ is a general field and need not to be
    either incompressible nor irrotational as for parabolic equations
    with a gradient flow structure \cite{peletier_2014_variational}.
    Indeed, assuming some regularity on the advection field, one can apply the Hodge
    decomposition to get the existence of a potential
    $V:\Omega\rightarrow\R$ and $\bF:\Omega\rightarrow\R^d$ such that
    \begin{equation}
      \bE\ =\ -\nabla V + \bF, \qquad\nabla\cdot\bF\ =\ 0
      \label{hodge}
    \end{equation}
    When $\bF=\mathbf{0}$, there are many examples  in the literature
    \cite{FFCWS, chainais_2007_asymptotic, bessemoulin_2012_finite,bessemoulin_2012_finite_filbet,
      cances_2015_numerical, burger_2010_mixed, CHJS} of finite volume
    schemes preserving entropy dissipation properties. C. Chainais-Hillairet and
    F. Filbet studied in \cite{chainais_2007_asymptotic} a finite volume
    discretization for nonlinear drift-diffusion system and proved that
    the numerical solution converges to a steady-state when time goes to
    infinity. In \cite{burger_2010_mixed}, M. Burger, J. A. Carrillo and
    M. T. Wolfram proposed a mixed finite element method for nonlinear
    diffusion equations and proved convergence towards the steady-state in
    case of a nonlinear Fokker-Planck equation with uniformly convex
    potential. All these schemes exploit the gradient flow structure of
    the equation, which gives a natural entropy. 
    
    In the non-symmetric case $\bF\neq\mathbf{0}$, the gradient
    structure cannot be exploited anymore, but as Proposition \ref{theobodineau} shows, there
    is still a relative entropy structure, which may be investigated to
    prove convergence to a steady state. The relative entropy properties of Fokker-Planck type equations in the whole space are
    exhaustively studied in the famous paper \cite{convex_2000_arnold} of A. Arnold, P. Markowich, G. Toscani, A. Unterreiter 
    and specific properties of the non-symmetric equations have been investigated in \cite{arnold_2008_large, achleitner_2015_large}.

    In bounded domains, entropy properties are often used in the context of
    no-flux boundary conditions or in the whole space, but few
    results concern Dirichlet boundary conditions. 
    In \cite{bessemoulin_2012_finite}, M. Bessemoulin-Chatard proposed an extension of the 
    Scharfetter-Gummel for finite volume scheme for convection-diffusion equations with nonlinear
    diffusion and non-homogeneous and unsteady Dirichlet boundary conditions.
    While in the latter work the author presents a scheme with a satisfying long-time behavior for a 
    similar class of models than those of the present paper, our strategy and objectives differ. Here we aim at 
    preserving a whole class of relative entropies and build our scheme for the transient problem from a discretization of the stationary 
    equation.
    
    Let us precise that we can generalize our approach to the more general boundary conditions 
    \[
    \left\{
    \begin{aligned}
      &f = f^b \,\text{on }\,\Gamma_D,\\
      &\left[{\bE}\,\eta(f)-{\nabla} \eta(f)\right]\cdot\bn(\bx) = 0 \quad\text{on}\quad\Gamma_N,
    \end{aligned}
    \right.
    \] 
    with $\Gamma = \Gamma_D\cup\Gamma_N$. 
    Our results hold in this setting with minor modifications but to avoid unnecessary technicalities in the notation
    and in the analysis we consider non-homogeneous Dirichlet conditions on the whole boundary. 
    However, numerical results will be shown in both cases.

    \subsection{Physical models}
    
    Before describing our numerical scheme, let us present some physical
    models described by equation \eqref{contprob} for which the large-time asymptotic
    has been studied using entropy-entropy dissipation arguments. 
    Some of these models are the homogeneous part of kinetic Fokker-Planck-type equations and
    this work constitutes a first step towards treating full kinetic models. In future work, we aim at adapting the strategy developed here
    to ensure the property of convergence to local equilibrium for the solutions of these equations.
    
    \subsubsection{\bf The Fokker-Planck equation with magnetic field}
    
    A classical model of plasma physics describing the dynamic of charged particles
    evolving in an external electromagnetic field $(-\nabla_\bx\phi,\, \bB)$
    is given by the Vlasov-Fokker-Planck equation reading
    \begin{equation}
      \dfrac{\partial F}{\partial t} + \bv\cdot\nabla_\bx F - \nabla_\bx \phi\cdot\nabla_\bv F + (\bv\wedge \bB)\cdot\nabla_\bv F =
      \nabla_\bv \cdot(\bv F + \nabla_\bv F).
      \label{VFPf}
    \end{equation}
    For more details on the model we refer to \cite{ bouchut_1995_on, herda_2015_on}.
    In \cite{bouchut_1995_on}, Bouchut and Dolbeault proved that the solution of \eqref{VFPf}
    in the whole phase space and without magnetic field converges to a global equilibrium.
    Their proof mainly relies on the decrease of the free energy functional, 
    which corresponds to the physical relative entropy introduced in \eqref{examples}. 
    The external magnetic field does not alter the relative entropy inequality.
    We refer to \cite{herda_2015_on} for the corresponding computations.
    Here, we consider the phenomena happening in the velocity space which results in a
    Fokker-Planck equation with magnetic field, namely equation \eqref{contprob} with $\eta(s)=s$ and 
    an advection field given by 
    \begin{equation}
      {\bE}(\bv) = -\bv +\bv\wedge \bB\,,
      \label{advecFP}
    \end{equation}
    with constant magnetic field $\bB$. 
    In applications, the velocity variable $\bv$ usually lives in $\R^3$. 
    However when performing numerical simulations, one needs to restrict the velocity domain to a bounded set $\Omega$.
    On the edge of this restricted domain $\Omega$, we shall consider the following non-homogeneous Dirichlet boundary conditions
    \begin{equation}
      f(t,\bv) = f^\infty(\bv)\quad \forall\bv\in\partial\Omega,
      \label{FPfbound}
    \end{equation}
    where $f^\infty$ is the local Maxwellian associated with \eqref{VFPf} which writes 
    \begin{equation}
      f^\infty(\bv) = \frac{1}{(2\pi)^{3/2}}e^{- \frac{|\bv|^2}{2}}\quad \forall\bv\in\Omega,
    \end{equation}
    and is a stationary state of \eqref{contprob} with \eqref{advecFP}-\eqref{FPfbound}. 
    With the boundary conditions \eqref{FPfbound}, one recovers the same stationary state as in the whole space while working in a bounded domain.
    As for the more complicated kinetic model \eqref{VFPf}, free energy (relative $\phi$-entropy) decrease holds.
    
    Our approach is particularly promising for this kind of problem when
    the solution develops some micro-instabilities around a steady
    state. In this situation, it is important that numerical artefacts do
    not generate spurious oscillations.

    \subsubsection{\bf The dumbbell model for the density of polymers in a dilute solution}
    
    The following kinetic equation describes the evolution of the density $F\equiv F(t,\bx,\bk)$ of polymers at time $t$ and position $\bx$ diluted in a fluid flow of velocity $\bu(\bx)$ from a mesoscopic point of view
    \begin{equation}
      \dfrac{\partial F}{\partial t} + \bu\cdot\nabla_\bx F =
      -\nabla_\bk\cdot\left[\left(\nabla_\bx \bu\,\bk -
      \frac{1}{2}\nabla_\bk\Pi(\bk)\right) F -  \frac{1}{2}\nabla_\bk F\right].
      \label{dumbbell}
    \end{equation}
    The polymers are pictured as two beads linked by a spring and the variable $\bk$ stands for the vector indicating the length and orientation of the molecules. The potential $\Pi$ is given by 
    $\Pi(\bk)= |\bk|^2/2$ in the case of Hookean dumbbells or by $\Pi(\bk)= -\ln(1-|\bk|^2)/2$ in the case of Finite Extensible Nonlinear Elastic dumbbells. 
    In the complete model, the velocity of the fluid $\bu$ follows an incompressible Navier-Stokes equation featuring an additional force term modeling for the contribution of the polymers on the dynamic of the fluid which results in a nonlinear kinetic-fluid coupling. 
    Here, we consider the simpler case where $u(\bx)$ is a given incompressible field. For more details on the modeling, we refer to \cite{jourdain_2006_long} and references therein.  We also refer to the paper \cite{masmoudi_2008_well} of Masmoudi
    that treats the well-posedness and provides additional information on
    the model. 
    
    Once again we aim at approximating numerically the ``velocity'' part of the kinetic equation \eqref{dumbbell} which rewrites as \eqref{contprob} 
    with $\eta(s)=s$ and 
    an advection field given by
    \begin{equation}
      {\bE}(\bk) = \bA\,\bk - \frac{1}{2}\nabla_\bk\Pi(\bk),
      \label{advecPoly}
    \end{equation}
    with constant matrix $\bA$ satisfying $\mathrm{tr}(\bA) = 0$. This matrix is the gradient of an incompressible velocity field at some space location.
    Natural boundary conditions for this model are given by null outward flux.
    In \cite{jourdain_2006_long}, the long-time behavior of \eqref{dumbbell}
    and of the latter reduced model
    are investigated using relative $\phi$-entropies 
    with $\phi$ given by \eqref{examples}.
    
    \subsubsection{\bf A nonlinear model, the porous medium equation}
    
    The porous medium equation is a nonlinear PDE writing
    \begin{equation}
      \dfrac{\partial f}{\partial t} = \Delta f^m,
      \label{PME}
    \end{equation}
    with $m>1$. It can model many physical applications and generally describes processes involving fluid flow, heat transfer or diffusion.
    The typical example is the description of the flow of an isentropic gas through a porous
    medium. There is a huge literature on this equation and we refer to the book of V\'asquez
    \cite{vasquez_2007_porous} for the detailed mathematical theory.
    
    Here, equation \eqref{PME} is set in a bounded domain $\Omega$ with non-homogeneous Dirichlet boundary conditions
    \[
    f(t,\bx) = f^b(\bx)>0\quad \forall\bx\in\partial\Omega,
    \]
    such that it might be recast like \eqref{contprob} with a null advection field and with $\eta(s) = s^m$.
    Using their relative $\phi$-entropy method, Bodineau, Mouhot, Villani and Lebowitz show exponential convergence
    to equilibrium for this nonlinear equation.

    \bigskip

    \subsection{Outline and main results}
    The plan of the paper is as follows.
    In Section~\ref{presscheme}, we present the finite volume scheme
    and the discrete version of the relative $\phi$-entropies.
    Then, in Section~\ref{propflux}, we prove the main properties of our discrete flux , namely the preservation of steady states in Lemma~\ref{preserve} and the non-negativity of discrete $\phi$  entropy dissipations in Proposition~\ref{relentprop}. 
    In Section~\ref{fullydiscrete}, we analyze our fully-discrete schemes. In Theorem~\ref{mainimp}, we prove well-posedness, stability and decay of discrete $\phi$-entropies for the time implicit version  of our scheme. For the explicit version, we prove the same results in Theorem~\ref{mainexp}, under a parabolic Courant-Friedrichs-Lewy (CFL) condition. Besides, the long-time behavior of the discrete solution, for both schemes, is investigated in this section and we prove in Theorem~\ref{longtimetheo} that discrete solutions return to equilibrium exponentially fast, with  a rate that does not depend on the size of the discretization.
    In Section~\ref{convergence}, we show convergence of the discrete solution to a solution of \eqref{contprob2}, when the size of the discretization goes to zero. 
    Finally in Section~\ref{numerics}, we end by providing numerical illustrations 
    of the properties of our schemes on the models presented above.

    \section{Presentation of the numerical schemes}
    \label{presscheme}
    In this section, we introduce our finite volume schemes. In the following, $T$ is a positive real number and $\Omega_T$ denotes  the cylinder $[0,T)\times\Omega$.
    We start with some notations associated with the discretization of $\Omega_T$.
    
    \subsection{Mesh and time discretization}

    An admissible mesh of $\Omega$ is defined by the triplet $(\T, \E, \Pnt)$. The set $\T$ is a finite family of nonempty connected open disjoint subsets $K\subset\Omega$ called control volumes or cells. 
    The closure of the union of all control volumes is equal to $\bar{\Omega}$.
    The set $\E$ is a finite family of nonempty subsets $\bar{\Omega}$ called edges.
    Each edge is a subset of an affine hyperplane in $\R^{d-1}$. Moreover, for any control volume $K\in\T$ there exists a subset $\E[K]$ of $\E$
    such that the closure of the union of all the edges in $\E[K]$ is equal to $\partial K = \bar{K}\setminus K$.
    We also define several subsets of $\E$.
    The family of interior edges $\Ei$ is given by $\{\sigma\in\E,\ \sigma\nsubseteq\Gamma\}$ and the family of exterior edges by $\Ee = \E\setminus\Ei$. 
    Similarly, for any control volume $K\in\T$, we define $\Ei[,K] = \Ei\cap\E[K]$ and $\Ee[,K] = \Ee\cap\E[K]$. 
    We assume that for any edge $\sigma$, the number of control volumes sharing the edge $\sigma$ is exactly $2$ for interior edges and $1$ for exterior edges.
    With these assumptions, every interior edge is shared by two control volumes, say $K$ and $L$, so that we may use the notation $\sigma = K|L$ whenever $\sigma\in\Ei$.
    The set $\Pnt = \left\{\bx_K\right\}_{K\in\T}$ is a finite family of points satisfying that for any control volume $K\in\T$, $\bx_K\in K$. 
    We introduce the transmissibility of the edge $\sigma$, given by 
    \[
    \tau_\sigma = \frac{m(\sigma)}{d_\sigma},
    \]
    where
    \[
    d_{\sigma} = \left\{
    \begin{aligned}
      &d(\bx_K, \bx_L),&&\text{if }\sigma\in\Ei,\ \sigma = K|L\,,\\
      &d(\bx_K, \sigma),&&\text{if }\sigma\in\Ee[,K]\,,
    \end{aligned}
    \right.
    \]
    with $d(\cdot,\cdot)$ the euclidean distance in $\R^{d}$.
    The size of the mesh is defined by 
    \[
    \Delta x = \max_{K\in\T}\sup_{\bx,\by\in K}d(\bx,\by)\,.
    \]
    
    \smallskip
    
    The Dirichlet condition on the boundary is given by $f^b\in L^\infty(\Gamma)$. Endowed with these boundary conditions, a discrete  solution of the scheme at some fixed time is an element of the set
    \[\X[f^b] = \left\{ f\in\R^{\T}\times\R^{\Ee}: \,f_\sigma =
    \frac{1}{m(\sigma)}\int_\sigma f^b dm,\,\forall\sigma\in\Ee \right\}.\]
    \begin{rema}
      The particular formula providing $f_\sigma$ on the exterior edges is cosmetic here since by the reformulation of the equation, we will only use the constant boundary values of a reformulated unknown defined in \eqref{defh}.
    \end{rema}
    Let us mention that with a slight abuse we keep the same notation for discrete and continuous unknowns. Concerning the initial condition we assume that 
    \begin{equation}
      f^\text{in}\geq0\quad\text{ and }\quad K_\text{in}\, :=\, \|f^\text{in}\|_{L^\infty(\Omega)}<+\infty\,,
      \label{boundinit}
      \tag{H1}
    \end{equation}
    and the discrete initial condition is given by
    \[
    f^\text{in}_K\, =\, \frac{1}{m(K)}\int_Kf^\text{in}(\bx)\dx\,,
    \]
    for all $K\in\T$.
    For any function $\psi:\R\rightarrow\R$, and $f\in \X[f^b]$ we shall define the component-wise composition with the intuitive notation
    $\psi(f) = ((\psi(f_K))_{K\in\T}, (\psi(f_\sigma))_{\sigma\in\Ee})$.
    
    \smallskip
    
    We denote the time step by $\Delta t$ and set $t^n = n\,\Delta t$.  From this, a time discretization $[0,T)$ is given by the integer $N_T\,=\,\lfloor T/\Delta t\rfloor $ and the sequence $(t^n)_{0\leq n\leq N_T}$. A spacetime discretization $\D$ of $\Omega_T$ is composed of an admissible mesh of $\Omega$ and the time discretization parameters $\Delta t$  and $N_T$.  Finally, the size of the discretization $\D$ is given by
    \[
    \delta\, =\, \max(\Delta x,\,\Delta t).
    \]

    \subsection{Discretization of the steady equation}
    
    In order to build our numerical scheme for the reformulated equation \eqref{contprob2}, we need a discrete version of the steady state  as well as corresponding discrete steady flux at interfaces. Hence we introduce an approximation of $\eta(\finf)$ on cells $(\eta(f^\infty)_K)_{K\in\T}$ and on edges $(\eta(f^\infty)_\sigma)_{\sigma\in\E}$. We suppose that there exists positive constants $m_\infty$ and $M_\infty$ that do not depend on the discretization $\D$ and such that for all $K\in\T$ and $\sigma\in\E$
    \begin{equation}
      0<m_\infty\,\leq\, \eta(f^\infty)_K,\ \eta(f^\infty)_\sigma\,\leq\, M_\infty\,.
      \label{boundstat}
      \tag{H2}
    \end{equation}
    \begin{rema}
      Of course this hypothesis requires some kind of maximum principle holding at the continuous level. It is the case at least for divergence free fields $\bE$, if $f^b$ is bounded and bounded from below by a positive constant then by \cite[Section 6.4, Theorem 1]{evans_2010_partial}.
    \end{rema}

    Moreover for each cell $K\in\T$ and edge $\sigma\in\E[K]$, we introduce a discrete flux $\F[K]^\infty$ approximating $\int_{\sigma}\bU^\infty(\bx)\cdot\bn_{K,\sigma}(\bx)\,\dx$ along with the divergence free hypothesis
    \begin{equation}
      \sum_{\sigma\in\E[K]}\F[K]^\infty = 0\,,
      \label{stat}
      \tag{H3}
    \end{equation}
    and the interior continuity condition
    \begin{equation}
      \forall\sigma = K|L\in\Ei\,,\quad\F[K]^\infty = -\F[L]^\infty.
      \label{intcont}
      \tag{H4}
    \end{equation}
    Let us emphasize that this last condition is required and satisfied for all the discrete flux we define in the following.
    From the flux we also introduce the discrete velocities 
    \[
    U^\infty_{K,\sigma} = \frac{1}{m(\sigma)}\F[K]^\infty\,.
    \]
    
    On one hand, for some models, the global equilibrium $f^\infty$ may be known
    analytically. In this case, we may build a discrete approximation in
    $\X[f^b]$ by a standard projection on the mesh
    and the numerical flux $\F[K]^\infty $ may be computed exactly or approximated with a quadrature formula. In any case, hypotheses \eqref{boundstat}-\eqref{intcont} must be satisfied. In Section~\ref{test0} and Section~\ref{test1}, this strategy is used for computing $\finf_K$ and $\F[K]^\infty $ on our test cases.
    
    On the other hand, when the  steady state is not known we
    apply a finite volume scheme to compute a numerical approximation. 
    Our method does not impose any scheme for solving the stationary equation as long as assumptions \eqref{boundstat}-\eqref{intcont} are satisfied. However to fix ideas let us provide an example here.
    \begin{exam} One can determine the discrete steady state by solving \eqref{stat} on each cell $K\in\T$ with the flux 
      \[
      \F[K]^\infty\ =\  \left\{
      \begin{aligned}
	&m(\sigma)\left[E_{K,\sigma}^{+}\ \eta(\finf)_K - E_{K,\sigma}^{-}\eta(\finf)_L\ -\ \frac{\eta(\finf)_L - \eta(\finf)_K}{d_\sigma}\right],&&\text{if}\ \sigma = K|L\,,\\
	&m(\sigma)\left[E_{K,\sigma}^{+}\ \eta(\finf)_K  - E_{K,\sigma}^{-}\eta(\finf)_\sigma\ -\ \frac{\eta(\finf)_\sigma - \eta(\finf)_K}{d_\sigma}\right],&&\text{otherwise}\,,
      \end{aligned}\right.
      \]
      where for any real number $u$, we denote by $u^+ = \max(u,0)$ and $u^- = \max(-u,0)$ the positive and negative parts of $u$. The quantity $E_{K,\sigma}$ is a consistent approximation of $\bE\cdot\bn_{K,\sigma}$ on the edge $\sigma$ such as $\bE(\bx_\sigma)\cdot\bn_{K,\sigma}$, where $\bx_\sigma$ is the center of mass of the edge $\sigma$.  Then, for $\sigma=K|L$, the approximation of $\eta(\finf)$ on the edge $\sigma$ can be given by $\eta(f^\infty)_\sigma = (\eta(f^\infty)_K+\eta(f^\infty)_L)/2$ and for each cell $K\in\T$, $\finf_K = \eta^{-1}(\eta(\finf)_K)$.
      \label{exampFlux}
    \end{exam}
    In Section~\ref{test2}, we provide another example of resolution of the steady equation by a finite volume discretization.

    \subsection{Discretization of the evolution equation}
    Now we treat the time evolution problem and use the discrete steady state to build a
    numerical approximation of the reformulated equation \eqref{contprob2}. 
    We start by introducing our finite volume scheme 
    in a general implicit and explicit form. Then we define the flux.
    \subsubsection{Fully discrete schemes}
    In the continuous setting, we defined a new unknown to reformulate the convection-diffusion equation. Its discrete equivalent is still denoted $h$ and belongs to the space $\X[1]^{N_T+1}$. It is defined for any $K\in\T$ and $n\in\{0,\dots,N_T\}$ by
    \begin{equation}
      h_K^n\, =\, \frac{\eta(f_K^n)}{\eta(f^\infty)_K}\,.
      \label{defh}
    \end{equation}
    Let us insist on the fact that boundary conditions are contained in the definition of the approximation space so that here, $h_\sigma = 1$ for $\sigma\in\Ee$.
    Solving the fully discrete implicit scheme consists in finding $f = (f_K^n)_{n\in\{0,\dots, N_T\}}\in X_{f^b}^{N_T+1}$ such that for all $K\in\T$
    \begin{equation}
      \left\{
      \begin{aligned}
	&m(K)\frac{f_K^{n+1} - f_K^{n}}{\Delta t} + \sum_{\sigma\in\E[K]}\F[K](h^{n+1})\ =\ 0,\quad\forall n\in \{0,\dots, N_T-1\}\,,\\
	&f_K^0\ =\ f_K^\text{in}\,.
      \end{aligned}
      \right.
      \label{schemeimp}
    \end{equation}
    In explicit form, it amounts to building $f$ sequentially by
    \begin{equation}
      \forall K\in\T,
      \left\{
      \begin{aligned}
	&m(K)\frac{f_K^{n+1} - f_K^{n}}{\Delta t} + \sum_{\sigma\in\E[K]}\F[K](h^{n})\ =\ 0,\quad\forall n\in \{0,\dots, N_T-1\}\\
	&f_K^0 = f_K^\text{in},
      \end{aligned}
      \right.
      \label{schemeexp}
    \end{equation}
    In order to  lighten the notation we sometimes write $\F[K](h^{n}) = \F[K]^n$ in the following.
    For any discrete element $u\in X_{u^b}^{N_T+1}$, its reconstruction is given almost everywhere on $\Omega_T$ by
    \[
    u_\delta(t,x) = 
    \left\{\begin{aligned}
      &u_K^{n+1}&&\text{ if }x\in K\text{ and }t\in[t^n,t^{n+1})\,,&& \text{(Implicit case)}\\
      &u_K^{n}&&\text{ if }x\in K\text{ and }t\in[t^n,t^{n+1})\,.&&\text{(Explicit case)}
    \end{aligned}
    \right.
    \]
    \subsubsection{Definition of the flux}
    The flux is divided in two parts,
    \begin{equation}
      \F[K] = \F[K]^{\text{conv}} + \F[K]^{\text{diss}},
      \label{defflux}
    \end{equation}
    corresponding respectively to the discretization of the convective term $\int_{\sigma} h\,\bU^\infty\cdot\bn_{K,\sigma}(\bx)\,\dd m$ and of the dissipative term 
    $-\int_{\sigma}\eta(\finf)\, \nabla h \cdot\bn_{K,\sigma}(\bx)\,\dd m$. Accordingly,
    we call $F^{\text{conv}}_{K,\sigma}$ the convective flux and
    $F^{\text{diss}}_{K,\sigma}$ the dissipative flux.  The former is discretized thanks to an monotone upstream discretization reading
    \begin{equation}
      \F[K]^\text{conv}(h) = 
      \left\{
      \begin{aligned}
	&m(\sigma)\left[U_{K,\sigma}^{\infty,+}\ g(h_K,h_L) - U_{K,\sigma}^{\infty,-}g(h_L,h_K)\ \right],&&\text{if}\ \sigma = K|L,\\
	&m(\sigma)\left[U_{K,\sigma}^{\infty,+}\ g(h_K,h_\sigma)  - U_{K,\sigma}^{\infty,-} g(h_\sigma,h_K)\ \right],&&\text{otherwise},
      \end{aligned}\right.
      \label{convflux}
    \end{equation}
    where  $U_{K,\sigma}^{\infty,+} = \max(U_{K,\sigma}^{\infty},0)$ and $U_{K,\sigma}^{\infty,-} = \max(- U_{K,\sigma}^{\infty},0)$ denote the positive and negative parts of $ U_{K,\sigma}^{\infty}$ respectively. The classical upwind flux, that we use for our numerical simulations is obtained  by choosing $g(s,t) = s$.
    With this more general version, we aim at highlighting the importance of the monotony property of this flux concerning stability of the scheme and decay of discrete relative $\phi$-entropy. Accordingly,  $g:\RR^2\rightarrow\RR$ ought to satisfy the following assumptions.
    \begin{equation}
      \left\{
      \begin{aligned}
	&g\text{ is locally Lipschitz-continuous,}\\
	&g\text{ is non-decreasing in the first variable and non-increasing in the second variable,}\\
	&g(s,s) = s,\text{ for all $s\in\R$} .
      \end{aligned}\right.
      \label{condg}
      \tag{H5}
    \end{equation}
    The first and last conditions ensure consistency of the approximation.

    The dissipative flux is built on a two-point approximation of the derivative along the outward normal vector of each edge, namely
    \begin{equation}
      \F[K]^\text{diss}(h) = -\tau_\sigma\, \eta(\finf)_\sigma\,D_{K,\sigma}h,
      \label{dissflux}
    \end{equation}
    where the difference operator $D_{K,\sigma}$ is defined for any  $K\in\T$, $\sigma\in \E[K]$ and $u\in\X[u^b]$ by
    \[
    D_{K,\sigma}u = \left\{
    \begin{aligned}
      &u_L-u_K&&\text{if }\sigma\in\Ei,\ \sigma = K|L,\\
      &u_\sigma-u_K&&\text{if }\sigma\in\Ee[,K].
    \end{aligned}
    \right.
    \]
    For consistency of discrete gradients, we require an orthogonality condition for the mesh, namely
    \begin{equation}
      \begin{aligned}
	\forall\, \bx,\by\in\sigma= K|L,\quad (\bx-\by)\cdot(\bx_K-\bx_L) = 0.
      \end{aligned}
      \label{orthog}
      \tag{H6}
    \end{equation}
    \begin{rema}
      Hypothesis \eqref{orthog} is necessary for establishing the convergence of the scheme in Section \ref{convergence} and even if it is standard \cite{eymard_2000_finite}, it remains a restrictive condition on the shape of the mesh. The necessity for this hypothesis stems from the fact that we choose a two-point flux \eqref{dissflux} for the diffusion. This choice is motivated by the monotony properties of this flux which enable the decay of $\phi$-entropies at the discrete level.
    \end{rema}
    \begin{rema}
      Even if the flux \eqref{convflux} and \eqref{dissflux} seem fairly classical, one must not forget that they act on the unknown $h$ while the discrete time derivative in \eqref{schemeimp} or \eqref{schemeexp} is on $f$. As in the continuous setting, this is the most important part of the strategy.
    \end{rema}

    \subsection{Discrete relative $\phi$-entropies and dissipations}
    For $f\in\X[f^b]$, the discrete equivalent of the relative $\phi$-entropy in Definition~\ref{relentcont} is given by
    \begin{equation}
      H_{\phi}(f)  = \sum_{K\in\T} m(K) e_{\phi, K}(f),
      \label{defentro}
    \end{equation}
    where $e_\phi = \left(e_{\phi, K}\right)_{K\in\T}$ is the local discrete relative $\phi$-entropy writing, for $K\in\T$
    \[
    e_{\phi, K}(f) = \int_{f^\infty_K}^{f_K}\phi'\left(\frac{\eta(s)}{\eta(f^\infty_K)}\right)\dd s.
    \]
    Contributions of the convective and diffusive part of the equation to the relative entropy variation are defined, for $h\in\X[1]$ by
    \begin{equation}
      C_\phi(h) = \sum_{K\in\T}\sum_{\sigma\in\E[K]}\phi'(h_K)\F[K]^\text{conv}(h),\qquad D_\phi(h) = \sum_{K\in\T}\sum_{\sigma\in\E[K]}\phi'(h_K)\F[K]^{\text{diss}}(h).
      \label{defdissip}
    \end{equation}
    We write $H_\phi^n$, $C_\phi^n$ and $D_\phi^n$ to denote respectively $H_{\phi}(f^n)$, $C_{\phi}(h^n)$ and $D_{\phi}(h^n)$.
    The precise relation between these quantities is derived in the proof of Theorem~\ref{mainimp} and \ref{mainexp}. For the moment, $C_\phi^n+D_\phi^n$ should be thought as the ``time derivative'' of $H_\phi^n$.
    In Proposition \ref{relentprop} we show that $D_\phi$ is as expected consistent with its continuous analogue $\mathcal{D}_\phi$ and non-negative.
    Moreover, thanks to the monotonicity properties of the convective flux \eqref{convflux}, we also prove that  $C_\phi(f)$ remains non-negative, creating an additional numerical dissipation consistent with $0$ as $\delta\rightarrow 0$.

    \section{Properties of the discrete flux}\label{propflux}
    We present here, independently of time discretizations \eqref{schemeimp} or \eqref{schemeexp}, the properties of the flux introduced in the previous section. In this section, $f$ is an element of $\X[f^b]$ and $h$ the corresponding element of $\X[1]$ using relation \eqref{defh}.
    
    \subsection{Preservation of the steady state}
    The next lemma show that the discrete steady state is a steady state of our schemes.
    \begin{lem}
      Under hypotheses \eqref{boundstat}-\eqref{stat}, if for all cells $K\in\T$, $\eta(f_K) = \eta(\finf)_K$, then
      \[
      \forall K\in\T,\quad\sum_{\sigma\in\E[K]}\F[K](h)\, =\, 0.
      \]
      \label{preserve}
    \end{lem}
    \begin{proof}
      Just observe that every component of $h$ equals $1$. Hence, for any $K\in\T$ and $\sigma\in\E[K]$, $F_{K,\sigma}^\text{diss}(h) = 0$ and
      \[
      \sum_{\sigma\in\E[K]}F_{K,\sigma}^\text{conv}(h) = \sum_{\sigma\in\E[K]}m(\sigma)\left(U_{K,\sigma}^{\infty,+} - U_{K,\sigma}^{\infty,-}\right) = \sum_{\sigma\in\E[K]}\F[K]^\infty = 0,
      \]
      using \eqref{stat}. 
    \end{proof}
    \subsection{Non-negativity of dissipations}
    
    \begin{prop}
      Let $\phi$ be any entropy generating function and $h\in\X[1]$. Under hypotheses \eqref{boundstat}-\eqref{condg},
      the following results hold. 
      \begin{itemize}
	\item[(i)] The numerical dissipation is non-negative, namely,
	\[
	C_\phi(h)\,\geq\,0\,.
	\]
	\item[(ii)] The physical dissipation rewrites
	\begin{equation}
	  \ds
	  \begin{aligned}
	    &D_\phi(h) &=&&& \sum_{\substack{\sigma\in\Ei\\ \sigma= K|L}}\tau_\sigma\; D_{K,\sigma}h\;D_{K,\sigma}\phi'(h)\;\eta(f^\infty)_\sigma\\
	    &&&&+& \sum_{K\in \T}\sum_{\sigma\in\Ee[,K]}\tau_\sigma\; D_{K,\sigma}h\;D_{K,\sigma}\phi'(h)\;\eta(f^\infty)_\sigma\;\geq \;0.
	  \end{aligned}
	  \label{consitentdissip}
	\end{equation}
	
      \end{itemize}
      \label{relentprop}
    \end{prop}
    In order to prove the non-negativity of the numerical dissipation we compare it to $C_\phi^{M_\phi}$ which is the numerical dissipation 
    of a centered convective flux that we define hereafter as well as some complementary notation.
    
    \begin{defi}
      A function $M:\RR_+\times \RR_+\rightarrow \R$ is called a \emph{mean function} if it satisfies for all $s,t\in \RR_+$,
      \begin{enumerate}
	\item $M(s,t) = M(t,s)$,
	\item $M(s,s) = s$,
	\item If $s\leq t$, then $s \leq M(s,t)\leq t$.
      \end{enumerate}
      \label{meanfunc}
    \end{defi}
    We also define $M_\sigma : \X[u^b]\longrightarrow \R$ by $M_{\sigma}(u) = M(u_K, u_L)$ if
    $\sigma = K|L$ and $M_{\sigma}(u) = M(u_K, u_\sigma)$ otherwise.
    This is well defined thanks to the symmetry of $M$. For any such function $M$,
    we define the centered convective flux associated to $M$ by
    \[
    F_{K,\sigma}^M(h) = m(\sigma)\ U_{K,\sigma}^\infty\ M_\sigma(h)
    \]
    and 
    \[
    C_\phi^M = \sum_{K\in\T}\sum_{\sigma\in\E[K]}\phi'(h_K)\F[K]^M(h).
    \]
    Finally for any entropy generating function $\phi$, it is elementary to show that
    \[
    M^\phi(s,t) = \frac{\varphi(s) - \varphi(t)}{\phi'(s)-\phi'(t)},
    \]
    where  $\varphi(s) = s\phi'(s) - \phi(s)$, defines a continuous mean function.
    We call it the $\phi$-mean.
    \begin{rema}
      Let us note that for the 2-entropy generating function $\phi_2(s) = (s-1)^2$ , the corresponding 
      $\phi_2$-mean is the arithmetic average and therefore $F_{K,\sigma}^{M_{\phi_2}}$ is a
      centered approximation for the convective flux, namely for $\sigma\in\Ei$
      \[
      F_{K,\sigma}^{M_{\phi_2}}(h) = m(\sigma)\ U_{K,\sigma}^\infty\ \frac{h_K+h_L}{2}.
      \]
      When choosing the generator of the physical entropy ${\phi_1}(s) = s\log(s) - s + 1$, the corresponding mean function is the logarithmic average 
      reading $M_{\phi_1}(s,t) = (s-t)/(\log(s) - \log(t))$.
    \end{rema}
    
    We are now ready to prove Proposition \ref{relentprop}.

    \begin{proof}[Proof of Proposition \ref{relentprop}]
      In order to prove \eqref{consitentdissip}, we use \eqref{defdissip} and \eqref{dissflux} to get
      \[
      D_\phi(f) = -\sum_{K\in\T}\sum_{\sigma\in\E[K]}\tau_\sigma\ D_{K,\sigma}h\ \phi'(h_K)\ \eta(f_\sigma^\infty)
      \]
      and the result stems from a discrete integration by parts.
      
      Now, let us prove the non-negativity of $C_\phi$. For $M$ a mean function (see Definition~\ref{meanfunc}), we perform a discrete integration by parts of $C_\phi - C_\phi^M$ which yields
      \[
      C_\phi - C_\phi^M = -\sum_{\substack{\sigma\in\Ei\\ \sigma= K|L}}(\F[K]^\text{conv} -  \F[K]^M)D_{K,\sigma}(\phi'(h)) - \sum_{K\in \T}\sum_{\sigma\in\Ee[,K]}(\F[K]^\text{conv} -  \F[K]^M)D_{K,\sigma}(\phi'(h))\,.
      \]
      Now let us just remark that for any $\sigma\in\Ei$ with $\sigma = K|L$, 
      \[
      \begin{aligned}
	&-(\F[K]^\text{conv} -  \F[K]^M)D_{K,\sigma}(\phi'(h)) &=
	&&& \,m(\sigma)U_{K,\sigma}^{\infty,+}\,(g(h_K,h_K) - g(h_K,h_L))\,(\phi'(h_L)-\phi'(h_K))\\
	&&+&&& m(\sigma)\,U_{K,\sigma}^{\infty,+}\,(M_\sigma(h) - h_K)\,(\phi'(h_L)-\phi'(h_K))\\
	&&+&&& m(\sigma)\,U_{K,\sigma}^{\infty,-}\,(h_L-M_\sigma(h))\,(\phi'(h_L)-\phi'(h_K))\\
	&&+&&& m(\sigma)\,U_{K,\sigma}^{\infty,-}\,(g(h_L,h_K) - g(h_L,h_L))\,(\phi'(h_L)-\phi'(h_K))\,,
      \end{aligned}
      \]
      where we used that $g(s,s) = s$. If $\sigma\in\Ee$, the same equation holds replacing $h_L$ with $h_\sigma$.
      Therefore, since $\phi'$ and $g(s,\cdot)$ are  monotonically non-decreasing functions and $M_\sigma(h)$ is always between $h_K$ and $h_L$ (\textit{resp.} $h_\sigma$), the above quantity is non-negative.
      Hence, it proves that $C_\phi \geq C_\phi^M$.
      
      Finally, a simple computation using two integrations by parts yields
      \[
      \begin{aligned}
	&C_\phi^{M_\phi} &=&&&  -\sum_{\substack{\sigma\in\Ei\\ \sigma= K|L}}\F[K]^{M_\phi}D_{K,\sigma}(\phi'(h)) - \sum_{K\in \T}\sum_{\sigma\in\Ee[,K]}\F[K]^{M_\phi}D_{K,\sigma}(\phi'(h))\,,\\
	&&=&&&-\sum_{\substack{\sigma\in\Ei\\ \sigma= K|L}}\F[K]^\infty D_{K,\sigma}(\varphi(h)) - \sum_{K\in \T}\sum_{\sigma\in\Ee[,K]}\F[K]^\infty D_{K,\sigma}(\varphi(h))\,,\\
	&&=&&&\sum_{K\in \T}\sum_{\sigma\in\E[K]}\F[K]^\infty \varphi(h_K)\ =\ 0\,,\\
      \end{aligned}
      \]
      where we used \eqref{stat} in the last equality. Thus, $C_\phi \geq C_\phi^{M_\phi}=0$.
    \end{proof}
    \begin{rema}
      Let us note that if we had used the fluxes $F_{K,\sigma}^{M_\phi}$ instead of $F_{K,\sigma}^\text{conv}$ in our scheme, then the contribution of the convection would be null in the discrete relative $\phi$-entropy variation.
      However, the scheme would have been $\phi$-dependent.
      With the upstream flux, we get the whole class of relative entropy inequalities at the cost of an additional numerical dissipation.
    \end{rema}
    
    \section{Analysis of the schemes}
    \label{fullydiscrete}
    
    \subsection{Implicit Euler}
    
    Before stating our main result for the implicit scheme \eqref{schemeimp}-\eqref{dissflux}. Let us show that the control of a large class of relative $\phi$-entropies yields
    $L^\infty$ stability of the discrete solution. To establish this result with bounds that are independent of the discretization $\D$,  $L^\infty$ bound on the initial data, given by \eqref{boundinit} is mandatory.
    
    \begin{lem}
      Assume that for any entropy function $\phi$, and $n\in\{0,\dots, N_T\}$
      \[
      H_{\phi}(h^n) \leq H_\phi(h^0)\,.
      \]
      Then, under hypotheses \eqref{boundinit} and \eqref{boundstat}, there exists a positive constant $K_\infty$ depending only on $K_\text{in}$, $\eta$, $m_\infty$ and $M_\infty$ such that for all $(t,x)\in[0,T)\times\R^d$,
      \[
      0\leq\,f_\delta(t,x),\ h_\delta(t,x)\, \leq\, K_\infty\,.
      \] 
      \label{stabilityh}
    \end{lem}
    \begin{proof}
      Let us define $\phi_{u_0}(\cdot) = (\cdot - u_0)^+$ if $u_0 > 1$ and $\phi_{u_0}(\cdot) = (\cdot - u_0)^-$ if $u_0 < 1$. It is differentiable on $\R\setminus\{u_0\}$ with derivative $\phi_{u_0}' =\mathds{1}_{[u_0,+\infty)}$ if $u_0>1$ and $\phi_{u_0}' =-\mathds{1}_{(-\infty,u_0]}$ if $u_0<1$, where $\mathds{1}_{A}$ denotes the indicator function of the set $A$. The function $\phi_{u_0}$ is not an entropy function. However, if for any $n\in\NN$ one has $H_{\phi_{u_0}}(h^n) = 0$ for $u_0 = 0$ and $u_0 = \eta(K_\text{in})/m_\infty$, then $e_{\phi_{u_0},K}(h^n) = 0$ and $h_K^n\in[0,\eta(K_\text{in})/m_\infty]$ for all cell $K$. Hence let us show that  $H_{\phi_{u_0}}(h^n) = 0$ by an approximation argument. 
      
      \smallskip
      
      Let $B:x\mapsto x/(\exp(x)-1)$ be the Bernoulli function, which is a strictly convex $\C^2$ function.
      Then, for any $\varepsilon>0$ and $u_0\in\R$, one readily checks that
      \[
      \phi_{\varepsilon, u_0}:u\longmapsto \phi_{\varepsilon, u_0}(u) = \varepsilon\left[B\left(\frac{u-u_0}{\varepsilon}\right) - B\left(\frac{1-u_0}{\varepsilon}\right)\right] + B'\left(\frac{1-u_0}{\varepsilon}\right)(1-u),
      \]
      are entropy generating functions. 
      When $\varepsilon$ tends to $0$, $\phi_{\varepsilon, u_0}'$ converges pointwise to $\phi_{u_0}'$ on $\R\setminus\{u_0\}$.
      Therefore, by dominated convergence, one has for any $n\in\{0,\dots, N_T\}$
      \[
      0\leq H_{\phi_{u_0}}^n\leq H_{\phi_{u_0}}^0.
      \]
      Since, by \eqref{boundstat} and \eqref{boundinit}, $H_{u_0}^0 = 0$ for $u_0 = \eta(K_\text{in})/m_\infty$ and $u_0 = 0$ then $H_{u_0}^n = 0$ for all $n\in\{0,\dots, N_T\}$. Hence we infer the uniform bounds on $h_\delta$, and consequently on $f_\delta$.
    \end{proof}

    \begin{theo}[Implicit Euler]
      Under hypotheses \eqref{boundinit}-\eqref{condg} the scheme
      \eqref{schemeimp} together with \eqref{defflux}-\eqref{dissflux} satisfies the following properties.
      \begin{itemize}
	\item[(i)] There exists a unique discrete solution $f\in\X[f^b]^{N_T+1}$;
	\item[(ii)] there is a positive constant $K_\infty$ depending only on $K_\text{in}$, $\eta$, $m_\infty$ and $M_\infty$ such that for all $(t,x)\in[0,T)\times\R^{d}$,
	\[
	0\leq\,f_\delta(t,x),\ h_\delta(t,x)\, \leq\, K_\infty\,;
	\] 
	\item[(iii)] the scheme preserves the steady state $f^\infty$
	and for any entropy function $\phi$ and $n\in \{0,\dots,N_T-1\}$,
	\begin{equation}
	  \frac{H_{\phi}^{n+1} - H_{\phi}^{n}}{\Delta t}  + D_\phi^{n+1} \leq 0\quad \text{ and }\quad D_\phi^{n+1} \geq 0\,.
	  \label{decrentimplicit}
	\end{equation}
      \end{itemize}
      \label{mainimp}
    \end{theo}

    \begin{proof}
      (i): The existence of a unique solution to the implicit scheme can be shown with a fixed point strategy close to that in 
      \cite[Remark 4.9]{eymard_2000_finite} and we do not detail this part.\\
      (iii): Let us derive the entropy inequality.
      The Taylor-Young theorem provides the existence of $\theta_{K}^{n,n+1}\in(\min(f_K^n, f_K^{n+1}),\ \max(f_K^n, f_K^{n+1}))$ such that
      \[
      \begin{aligned}
	&e_{\phi, K}^{n+1} - e_{\phi, K}^n &=&&& \int_{f^n_K}^{f^{n+1}_K}\phi'\left(\frac{\eta(s)}{\eta(f^\infty_K)}\right)\dd s\\
	&&=&&& \phi'(h_K^{n+1})(f^{n+1}_K - f^n_K) - \frac 12\psi_K(\theta_{K}^{n,n+1})(f^{n+1}_K - f^n_K)^2\\
	&&=&&& -\frac{\Delta t}{m(K)} \phi'(h_K^{n+1})\sum_{\sigma\in\E[K]}F_{K,\sigma}^{n+1} -
	\frac 12\psi_K(\theta_{K}^{n,n+1})(f^{n+1}_K - f^n_K)^2,
      \end{aligned}
      \]
      where $\psi_K$ is given by 
      \begin{equation}
	\psi_K:x\mapsto\frac{\eta'(x)}{\eta(f_K^\infty)}\phi''\left(\frac{\eta(x)}{\eta(f_K^\infty)}\right).
	\label{psiK}
      \end{equation}
      Note that $\psi_K$ is a positive function thanks to the positive monotony of $\eta$ and $\phi'$. 
      From the definition of the dissipation in \eqref{defdissip}, it yields 
      \[
      \frac{H_{\phi}^{n+1} - H_{\phi}^{n}}{\Delta t}  + D_\phi^{n+1} \leq -C_\phi^{n+1} - \frac{1}{2\Delta t}\sum_{K\in\T}\psi_K(\theta_{K}^{n,n+1})(f^{n+1}_K - f^n_K)^2 m(K)\leq 0,
      \]
      where we applied Proposition~\ref{relentprop} to control the first term of the right-hand side.\\
      (ii): By summing \eqref{decrentimplicit} over $n$ and using the positivity of $D^n_\phi$, we obtain the boundedness of $H_\phi^n$ and we can use Lemma~\ref{stabilityh} to conclude. 
    \end{proof}
    
    \subsection{Explicit Euler}
    Before stating the main result on this scheme, let us introduce 
    \begin{equation}
      a_{K,\sigma}^n =\left\{
      \begin{aligned}
	&d_\sigma\left(U_{K,\sigma}^{\infty,+}\ \frac{h_K^n - g(h_K^n,h_L^n)}{D_{K,\sigma}h^n}\, +\, U_{K,\sigma}^{\infty,-}\ \frac{g(h_L^n,h_K^n) - h_K^n}{D_{K,\sigma}h^n} \right),&&\text{if}\ \sigma = K|L,\\
	&d_\sigma\left(U_{K,\sigma}^{\infty,+}\ \frac{h_K^n - g(h_K^n,h_\sigma)}{D_{K,\sigma}h^n}\,+\, U_{K,\sigma}^{\infty,-}\ \frac{g(h_\sigma,h_K^n) -h_K^n}{D_{K,\sigma}h^n} \right),&&\text{otherwise},
      \end{aligned}\right.
      \label{aKsig}
    \end{equation}
    with the convention $a_{K,\sigma}^n = 0$ if  $D_{K,\sigma} h^n = 0$.
    Then observe that we can use this $a_{K,\sigma}$ to reformulate the convective part
    of the scheme as a ``diffusive term'', thanks to the incompressibility of $U_{K,\sigma}^\infty$, which is a consequence of \eqref{stat}. Indeed, for all $K\in\T$, we have
    \[
    -\sum_{\sigma\in\E[K]}\tau_\sigma\ a_{K,\sigma}\ D_{K,\sigma}h = \sum_{\sigma\in\E[K]}\F[K]^\text{conv}  - h_K\sum_{\sigma\in\E[K]}m(\sigma)U_{K,\sigma}^\infty = \sum_{\sigma\in\E[K]}\F[K]^\text{conv},
    \]
    where we used \eqref{stat} in the last equality.
    Now we suppose that there is a positive constant $V_\infty$ that does not depend on $\delta>0$ and such that 
    \begin{equation}
      \max_{K\in\T}\max_{\sigma\in\E[K]} |U_{K,\sigma}^\infty|\,\leq\,V_\infty\,.
      \label{boundU}
      \tag{H7}
    \end{equation}
    Thanks to the monotonicity and regularity properties of $g$ from \eqref{condg}, one has
    \[
    0\leq a_{K,\sigma}\leq C_g\ V_\infty\ \text{diam}(\Omega),
    \]
    where $C_g$ is the Lipschitz constant of $g$ on $[0,\,\eta(K_\text{in})/m_\infty]^2$ and $\text{diam}(\Omega)$ is the diameter of 
    $\Omega$.
    \begin{rema}
      Hypothesis \eqref{boundU} is the most restrictive of our assumptions since it implicitly demand uniform $W^{1,\infty}$ bound on the discrete steady state $\eta(\finf)$. A more natural and similar hypothesis is a uniform $L^2$ control on the velocity (instead of $L^\infty$).
      It corresponds to discrete $H^1$ control on $\eta(\finf)$, uniformly in $\delta>0$, which can be easily obtained for a finite volume discretization of an elliptic equation like \eqref{stateq}. For the convergence analysis in Section~\ref{convergence}, this hypothesis follows from \eqref{cvstat}. Unfortunately, we need the stronger \eqref{boundU} to obtain $L^\infty$ stability for the explicit scheme.  
    \end{rema}
    \begin{theo}[Explicit Euler]
      Let $f \in\X[f^b]^{N_T+1}$ be defined by the scheme
      \eqref{schemeexp}-\eqref{dissflux}. 
      Then, under hypotheses \eqref{boundinit}-\eqref{condg}, \eqref{boundU} the following properties hold.

      \begin{itemize}
	\item[(i)] There exists a positive constant $C_{L^\infty}$ depending only on $K_\text{in}$, $M_\infty$, $V_\infty$, $g$, $\eta$ and $\Omega$ such that under the CFL condition 
	\[
	\max_{K\in\T}\frac{\Delta t}{m(K)}\sum_{\sigma\in\E[K]}\tau_\sigma\leq C_{L^\infty},
	\]
	there is a positive constant $K_\infty$ depending only on $K_\text{in}$, $\eta$, $m_\infty$ and $M_\infty$ such that for all $(t,x)\in[0,T)\times\R^{d}$,
	\[
	0\leq\,f_\delta(t,x),\ h_\delta(t,x)\, \leq\, K_\infty\,.
	\] 
	\item[(ii)] If $\Phi$ is a family of entropy functions with second derivate bounded between $m_\Phi$ and $M_\Phi$, then there exists a positive constant $C_\text{ent}$ depending only on 
	on $K_\text{in}$, $m_\infty$, $M_\infty$, $V_\infty$, $g$, $\eta$ and $\Omega$  such that for every $\varepsilon\in(0,1)$, under the CFL condition
	\begin{equation}
	  \max_{K\in\T}\frac{\Delta t}{m(K)}\sum_{\sigma\in\E[K]}\tau_\sigma\,\leq\, \min\left(C_{L^\infty},\,C_\text{ent}\;\frac{m_\Phi}{M_\Phi}\,\varepsilon\right)  \label{CFL}
	\end{equation}
	for any $\phi\in\Phi$ and  $n\in\{0,\dots,N_T-1\}$,
	\begin{equation}
	  \frac{H_{\phi}^{n+1} - H_{\phi}^{n}}{\Delta t}  + (1-\varepsilon)D_\phi^{n} \leq 0\quad \text{ and }\quad D_\phi^{n} \geq 0\,.
	  \label{decrentexplicit}
	\end{equation}
      \end{itemize}
      Moreover the scheme preserves the stationary state $f^\infty$.
      \label{mainexp}
    \end{theo}
    \begin{proof}
      
      (i): For explicit time discretization it is rather classical (see \cite{eymard_2000_finite}) to use the convexity property of the scheme to show $L^\infty $ stability. However due to the fact that flux are expressed in terms of $h$ instead of $f$ adds up some technicalities here. We have to proceed in two steps to show the heredity of the induction hypothesis $h^n\in J = [0,\eta(K_\text{in})/m_\infty]$.
      
      \smallskip
      
      First note that for any $\delta\geq 0$, under the CFL condition
      \[
      \max_{K\in\T}\frac{\Delta t}{m(K)}\sum_{\sigma\in\E[K]}\tau_\sigma\,\leq\, C_1(\delta) := \frac{\delta}{K_\text{in} \;(C_g\ V_\infty\ \text{diam}(\Omega) + M_\infty)}\,,
      \]
      one has $f^{n+1}_K\in I_\delta = [-\delta,\,K_\text{in} + \delta]$ for all control volume $K\in\T$ since the scheme rewrites 
      \[
      f_K^{n+1} = f_K^{n} + \frac{\Delta t}{m(K)}\sum_{\sigma\in\E[K]}\tau_\sigma(a_{K,\sigma} + \eta(f_K^\infty))\;D_{K,\sigma}h^n\,.
      \]
      Let us define $M_\eta^\delta = \sup_{s\in I_\delta} \eta'(s)$. Now we show that under a possibly more restrictive CFL condition, $h_K^n$ belongs to $J$.
      By the mean value theorem, there exists $g_K^n\in I_\delta$ such that
      \[
      h_K^{n+1} - h_K^{n} = \frac{\eta'(g_K^n)}{\eta(f_K^\infty)}(f_K^{n+1} - f_K^{n}).
      \]
      The scheme can then be rewritten as
      \[
      \begin{aligned}
	&h_K^{n+1} &=&&& \left(1 - \frac{\eta'(g_K^n)\;\Delta t}{\eta(f_K^\infty)\;m(K)}\sum_{\sigma\in\E[K]}\tau_\sigma(a_{K,\sigma} + \eta(f_K^\infty)) \right)h_K^n\\
	&&+&&& \frac{\eta'(g_K^n)\;\Delta t}{\eta(f_K^\infty)\;m(K)}\sum_{\substack{\sigma\in\Ei[,K]\\ \sigma = K|L}}\tau_\sigma(a_{K,\sigma} + \eta(f_K^\infty))\; h_L^n\\ 
	&&+&&&\frac{\eta'(g_K^n)\;\Delta t}{\eta(f_K^\infty)\;m(K)}\sum_{\substack{\sigma\in\Ee[,K]}}\tau_\sigma(a_{K,\sigma} + \eta(f_K^\infty))\; h_\sigma^n\,.\\
      \end{aligned}
      \]
      Under the CFL condition
      \[
      \max_{K\in\T}\frac{\Delta t}{m(K)}\sum_{\sigma\in\E[K]}\tau_\sigma\,\leq\, C_2(\delta) :=
      \frac{m_\infty}{M_\eta^\delta \;(C_g\ V_\infty\ \text{diam}(\Omega) + M_\infty)},
      \]
      it provides $h_K^{n+1}$ as a convex combination of elements of $J$ and hence $\{h_K^{n+1}\}_{K\in\T}\subseteq J$. 
      The CFL constant $C_{L^\infty}$ can then be taken as the supremum of $\min(C_1(\delta),C_2(\delta))$ when $\delta>0$.
      
      \bigskip
      
      (ii): We proceed exactly as in the proof of Theorem \ref{mainimp} to get the existence of $\theta_{K}^{n,n+1}$ such that
      \begin{equation}
	\frac{H_{\phi}^{n+1} - H_{\phi}^{n}}{\Delta t}  + D_\phi^{n} \leq -C_\phi^{n} + \frac{1}{2\Delta t}\sum_{K\in\T}\psi_K(\theta_{K}^{n,n+1})(f^{n+1}_K - f^n_K)^2 m(K),
	\label{remainder}
      \end{equation}
      for $\psi_K$ defined by \eqref{psiK}. Note that the sign of the last term has changed compared to the implicit scheme.
      Using the scheme, the last term in \eqref{remainder} can be estimated with the Cauchy-Schwartz inequality as 
      
      \begin{multline*}
	\frac{\Delta t }{2}\sum_{K\in\T}\frac{1}{m(K)}\psi_K(\theta_{K}^{n,n+1} )\left(\sum_{\sigma\in\E[K]}\F[K]^n\right)^2
	\leq\\  \frac{\Delta t\,M_\Phi\, M_\eta^0\,(C_g\,V_\infty\,\text{diam}(\Omega) + M_\infty)^2}{m_\infty}\sum_{K\in\T}\frac{1}{m(K)}\left(\sum_{\sigma\in\E[K]}\tau_\sigma\right)\,\left(\sum_{\sigma\in\E[K]}\tau_\sigma\left(D_{K,\sigma}h^n\right)^2\right),
      \end{multline*}
      Then, using that 
      \[
      D^{n}_\phi\geq\frac{m_\Phi\;m_\infty}{2}\sum_{K\in \T}\sum_{\sigma\in\E[K]}\tau_\sigma\; \left(D_{K,\sigma}h^n\right)^2,
      \]
      it yields \eqref{decrentexplicit} provided that the CFL condition is satisfied with constant
      \[
      C_\text{ent} = \frac{m_\infty^2}{2\;M_\eta^0\;(C_g\ V_\infty\ \text{diam}(\Omega) + M_\infty)^2}.
      \]
    \end{proof}

    \subsection{Long time behavior}
    \label{longtime}
    In this section, we study the long-time behavior of the discrete solution in the linear case $\eta(s) = s$.
    We prove exponential decay to equilibrium with the following strategy. From a classical discrete Poincaré inequality, we establish a discrete Poincaré-Sobolev inequality for controlling $H_{\phi_2}$ by $D_{\phi_2}$, where we recall that  $\phi_2(s) = (s-1)^2$. Then,
    properties \eqref{decrentimplicit} and \eqref{decrentexplicit} provide exponential decay to equilibrium by a discrete Gronwall-type argument for $H_{\phi_2}$ and all relative $\phi$-entropies controlled by the latter. 
    \begin{rema}
      The question of the existence of a general $\phi$-Poincaré-Sobolev functional inequality
      for any entropy generating function $\phi$, even in the continuous setting, goes way beyond the scope of this paper and we refer to  \cite{bodineau_2014_lyapunov, convex_2000_arnold} for discussions on this matter.
    \end{rema}

    The $2$-entropy and its dissipation are closely related to, respectively, the $L^2$ norm
    \[
    \|u\|_{0,2}\ =\ \left( \sum_{K\in\T}m(K)\; \left|u_K\right|^2\right)^{1/2}\,,
    \]
    and the discrete $H^1$ semi-norm 
    \begin{equation}
      |u|_{1,2,\T}\ =\ \left(\sum_{\substack{\sigma\in\Ei\\ \sigma= K|L}}\tau_\sigma\; \left|D_{K,\sigma}u\right|^2+\sum_{K\in\T}\sum_{\sigma\in\Ee[,K]}\tau_\sigma\; \left|D_{K,\sigma}u\right|^2\right)^{1/2}\,,
      \label{H1seminorm}
    \end{equation}
    for which M. Bessemoulin-Chatard, C. Chainais-Hillairet and F. Filbet proved the following discrete Poincaré inequality
    in \cite[Theorem 6]{bessemoulin_2014_discrete}, based on older work that may be found in references therein. Before stating it, we need to introduce the following regularity constraint on the mesh. There is a positive constant $\xi_1$ that does not depend on $\D$ such that
    \begin{equation}
      \forall K\in\T,\ \forall \sigma\in\E[K],\quad d_{K,\sigma}\geq \xi_1\, d_\sigma
      \label{regmesh}
      \tag{H8}
    \end{equation}

    \begin{prop}[\cite{bessemoulin_2014_discrete}]
      Under hypothesis \eqref{regmesh}, there exists a constant $C_{P}$ only depending on $\Omega$ such that for all $u\in \X[0]$, it holds
      \[
      \|u\|_{0,2}\ \leq\ \frac{C_P}{\xi_1^{1/2}}|u|_{1,2,\T}\,.
      \]
      \label{Poincare}
    \end{prop}
    As a consequence we have the following chain of inequalities. 
    For any $h\in\X[1]$
    \begin{equation}
      H_{\phi_2}(h)\,\leq\,M_\infty\,\|h-1\|_{0,2}^2\,\leq\,\frac{C_P^2\,M_\infty}{\xi_1}|h|_{1,2,\T}^2\,\leq\,\frac{C_P^2\,M_\infty}{\xi_1\,m_\infty}D_{\phi_2}(h)\,.
      \label{phiPoincare}
    \end{equation}
    \begin{theo}(Exponential return to equilibrium)~\\
      We suppose that $\eta$ is the identity function, and that the assumptions of Theorem \ref{mainimp} (respectively those of Theorem \ref{mainexp} as well as the CFL condition \eqref{CFL} for $\Phi = \{\phi_2\}$) and \eqref{regmesh} are satisfied. In the implicit case, we also assume
      that $\Delta t$ is smaller that a given constant $k$.
      Then, the solution $f$ of the scheme \eqref{schemeimp},\eqref{defflux}-\eqref{dissflux} (respectively \eqref{defflux}-\eqref{dissflux}) is such that there are two positive constants, $\kappa$ depending only on $\Omega$, $m_\infty$, $M_\infty$, $\xi_1$,  and $C_\text{in}$ depending additionally on $H_{\phi_2}^0$ (and $k$ in the implicit case)  such that for all $n\in\{0,\dots,N_T-1\}$,
      \[
      H_{\phi_2}^n\leq C_\text{in}\ e^{-\kappa t^n}\,.
      \]
      As a consequence, for all $n\in\{0,\dots,N_T-1\}$, it holds
      \[
      \|f_\delta(t^n,\cdot)-f^\infty_\delta(\cdot)\|_{L^1(\Omega)}^2\leq C_\text{in}\ e^{-\kappa t^n}.
      \]
      \label{longtimetheo}
    \end{theo}
    \begin{proof}
      By combining \eqref{phiPoincare} and \eqref{decrentimplicit} (respectively \eqref{decrentexplicit}), we obtain
      $H^{n+1}_{\phi_2} \leq (1+\kappa\Delta t)^{-1}\,H^n_{\phi_2}$ in the implicit case with $\kappa =  \xi_1\,m_\infty/(C_P^2\,M_\infty)$ and $H^{n+1}_{\phi_2} \leq (1-\kappa\Delta t)\,H^n_{\phi_2}$ in the explicit case with $\kappa = (1-\varepsilon)\xi_1\,m_\infty/(C_P^2\,M_\infty)$ where $\varepsilon$ is in the CFL condition \eqref{CFL}. Hence we get exponential decay of  the $2$-entropy, with $C_\text{in} = H_{\phi_2}^0$ in the explicit case and $C_\text{in} = H_{\phi_2}^0\,\exp(\kappa^2k^2/2)$ in the implicit case. Finally,
      the estimate in $L^1$ is obtained from a Cauchy-Schwartz inequality
      \[
      \|f_\delta(t^n)-f^\infty_\delta\|_{L^1(\Omega)}^2\leq \|f^\infty_\delta\|_{L^1(\Omega)}\|f_\delta(t^n)/\sqrt{f^\infty_\delta}-\sqrt{f^\infty_\delta}\|_{L^2(\Omega)}^2 = \|f^\infty_\delta\|_{L^1(\Omega)}\,H_{\phi_2}^n\,,
      \]
      and the decay of $H_{\phi_2}$.
      
    \end{proof}
    \begin{rema}
      The exponential decay holds for many other $\phi$-entropies. For instance, let us restrict our class of entropy generating functions to that introduced by
      Arnold, Markowich, Toscani and Unterreiter in \cite{convex_2000_arnold}, that is those satisfying
      $\phi\in\mathcal{C}^4(\R_+)$ and
      \[
      \left(\phi^{'''}\right)^2\ \leq\ \frac{1}{2}\phi''\ \phi^{IV}.
      \]
      Mark that the physical and $p$-entropies are generated by these entropy functions. 
      As a consequence of \cite[Lemma 2.6]{convex_2000_arnold}, one has that any such $\phi$ is bounded
      from above by a quadratic entropy function, namely
      \[
      \frac{\phi(s)}{\phi''(1)}\ \leq\ \phi_2(s) = (s-1)^2.
      \]
      Therefore, since the same inequality holds for the corresponding relative entropies, namely $H_\phi\leq \phi''(1)H_{\phi_2}$, 
      the exponential decay holds for any relative $\phi$-entropy in this class.
    \end{rema}

    \section{Convergence}\label{convergence}
    In this section, we prove that when $\delta\rightarrow 0$, the sequence $(h_\delta)_{\delta>0}$ converges to a solution of \eqref{contprob2} in the following sense.
    \begin{defi}
      We say that $h$ is a solution of \eqref{contprob2} starting at initial data $f^\text{in}$  if
      \begin{itemize}
	\item $h-1\in L^2(0,T;\,H^1_0(\Omega))$
	\item $f = \eta^{-1}(\eta(f^\infty)h)\in L^1(\Omega_T)$
	\item For any test function $\psi\in\mathcal{C}^\infty_c(\Omega_T)$, it holds
	\[
	-\iint_{\Omega_T}\left[f\,\partial_t\psi +
	\left(\bU^\infty\,h -\eta(f^\infty)\,\nabla h\right)\cdot \nabla \psi\right]
	- \int_\Omega f^\text{in}\,\psi(0,\cdot)\ =\ 0\,.  \]
      \end{itemize}
      \label{defsol}
    \end{defi}
    Thanks to the reformulation of Equation~\eqref{contprob} into \eqref{contprob2}, the only nonlinearity lies in the first term.
    In order to prove convergence of the scheme we need strong compactness to recover pointwise convergence and identify 
    the limit of $(f_\delta)_\delta$ as $\eta^{-1}(\eta(f^\infty)h)$ with $h$ the limit of $(h_\delta)_\delta$. 
    
    Except for this originality, the strategy is fairly standard. We first derive uniform-in-$\delta$ estimates on $h_\delta$ and its (discrete) gradient in $\bx$
    in order to get compactness in the space variable. Then thanks to the structure of the equation compactness in time is obtained from previous estimates.
    Using the consistency we shall then take limits in our scheme and recover a global weak solution in the sense of Definition~\ref{defsol}.
    \subsection{Hypotheses}
    We discuss specific hypotheses needed for the convergence result to hold. While, it was not crucial until now, the orthogonality condition \eqref{orthog} is very important in this part, for consistency of discrete gradients. As in \cite{chainais_2003_finite, chainais_2004_finite, bessemoulin_2012_finite}, the latter are introduced on a dual mesh. For $\sigma = K|L$, we define $T_{\sigma}$ as the cell with vertices $\bx_K$, $\bx_L$ and those of $\sigma$. If $\sigma\in \Ee$ then $T_{\sigma}$ is the cell with vertices $\bx_K$ and those of $\sigma$, where $K$ is the only cell having $\sigma\in\E[K]$. We refer to \cite[Fig. 1]{chainais_2004_finite} for a visualization. On this dual mesh we set
    \[
    \nabla^\delta h_\delta(t,\bx) = 
    \left\{\begin{aligned}
      &\frac{m(\sigma)}{m(T_\sigma)}\,D_{K,\sigma}h^{n+1}\,\bn_{K,\sigma},&&\text{ if }\bx\in T_\sigma\text{ and }t\in[t^n,t^{n+1})&& \text{(Implicit case)}\\
      &\frac{m(\sigma)}{m(T_\sigma)}\,D_{K,\sigma}h^{n}\,\bn_{K,\sigma},&&\text{ if }\bx\in T_\sigma\text{ and }t\in[t^n,t^{n+1})&&\text{(Explicit case)}
    \end{aligned}
    \right.
    \]
    \begin{rema}
      The discrete gradients are well-defined since for $u\in\X[u_b]$ and $\sigma = K|L$, the product $D_{K,\sigma}u\ \bn_{K,\sigma}$ is independent of the cell $K$. For $\sigma\in\Ee$, there is no ambiguity since only one cell is sharing the edge.
    \end{rema}
    In order to obtain convergence of these gradients we need the following regularity hypothesis on the mesh (see \cite{chainais_2004_finite} and references therein for related comments). There exist a constant $\xi_2>1$ that does not depend on the discretization such that 
    \begin{equation}
      \forall \sigma\in\E,\quad m(T_\sigma)\,\leq m(\sigma)d_\sigma\,\leq\, \xi_2\,m(T_\sigma)\,.
      \label{regmesh2}
      \tag{H9}
    \end{equation}
    Since our scheme is based on a discretization of the stationary equation \eqref{stateq}, one needs the convergence of the latter to obtain that of the former. Therefore we assume that both the discrete steady state and the discrete velocity converge to their continuous counterparts. First, let us define them on the dual mesh, by
    \[
    \eta(f^\infty)_\delta(\bx)\,=\,\eta(f^\infty)_\sigma\quad\text{for }\ \bx\in T_\sigma\,,
    \]
    and
    \[
    \bU_\delta^\infty(\bx)\,=\,\frac{m(\sigma)d_\sigma}{m(T_\sigma)}\,U_{K,\sigma}^\infty\,\bn_{K,\sigma}\quad\text{for }\ \bx\in T_\sigma\,.
    \]
    Then, when $\delta\rightarrow 0$, we assume that
    \begin{equation}
      \left\{
      \begin{aligned}
	\eta(\finf)_\delta&&&\longrightarrow&&\eta(\finf)&\text{strongly in } L^2(\Omega)\\
	\bU_\delta^\infty&&&\longrightarrow&& \bU^\infty&\text{weakly in } L^2(\Omega)
      \end{aligned}\right.
      \label{cvstat}
      \tag{H10}
    \end{equation}
    \begin{rema}
      Hypothesis \eqref{cvstat} is very natural. Indeed, since $\eta(f^\infty)$ satisfies a linear non-degenerate second order elliptic equation, one expects to get uniform $H^1$ control on discrete solutions of any reasonable finite volume scheme. 
    \end{rema}
    
    In the rest of Section~\ref{convergence}, all results are stated for the implicit scheme. However, everything holds also for the explicit scheme with minor modifications provided that one adds hypothesis \eqref{boundU} and the CFL condition \eqref{CFL} with $\Phi = \{\phi_2\}$ to each proposition.

    \subsection{Compactness}
    In the following proposition, we provide the uniform $L^\infty$ and $L^2(0,T;H^1)$ estimates we need for compactness.
    The discrete $L^2(0,T;H^1)$ norm is defined by 
    \[
    \|h_\delta\|_{1,2,\D}\, =\, \left(\sum_{n=0}^{N_T-1}\Delta t\,|h^{n+1}|_{1,2,\T}^2\right)^{1/2}
    \]
    for the implicit scheme.
    \begin{prop}(Uniform estimates)
      Under hypotheses \eqref{boundinit}-\eqref{orthog} and \eqref{regmesh2}-\eqref{cvstat}, there are positive constants $K_\infty$ and $K_{1,2}$ depending only on $K_\text{in}$, $\eta$, $m_\infty$ and $M_\infty$ such that  the solution of the scheme \eqref{schemeimp} together with \eqref{defflux}-\eqref{dissflux} satisfies 
      \[
      0\,\leq\,f_\delta(t,\bx),\,h_\delta(t,\bx)\, \leq\, K_\infty\,,
      \]
      and 
      \[
      \|h_\delta\|_{1,2,\D}\,\leq\, K_{1,2}.
      \]
      \label{estim}
    \end{prop}
    \begin{proof}
      The first estimate is proved in Theorem \ref{mainimp} and \ref{mainexp}.
      The second estimate is a consequence of the boundedness of the dissipation of $2$-entropy and of the lower bound on the steady
      state since 
      \[
      m_\infty\,|h^n|_{1,2,\T}^2\,\leq\, D_{\phi_2}^n\,.
      \]
      By summing \eqref{decrentimplicit} over $n$ and inserting the above inequality one obtains the result.
    \end{proof}

    From the first result of Proposition~\ref{estim} one obtains weak-$\star$ compactness in $L^\infty$. 
    Now, we need to gain weak compactness on the sequence of discrete gradients as well as strong compactness
    on the sequence of discrete solutions. The strong compactness is obtained by the Riesz-Fréchet-Kolmogorov criterion on spacetime translates. In order to recover right boundary conditions at the limit the criterion is stated on the whole space $\R^{d+1}$ for $\tilde{h}_\delta$ defined by
    \[
    \tilde{h}_\delta(t,x) = 
    \left\{\begin{aligned}
      &h_\delta(t,\bx)-1&&\text{ if }(t,x)\in\Omega_T\,,\\
      &0&&\text{ if }(t,x)\in \R^{d+1}\setminus\overline{\Omega_T}\,.
    \end{aligned}
    \right.
    \]
    The proof of the following lemma can be readily adapted from \cite[Lemma 4.3 and 4.7]{eymard_2000_finite}.
    \begin{lem}
      Under hypotheses \eqref{boundinit}-\eqref{orthog} and \eqref{regmesh2}-\eqref{cvstat}, there is a positive constant $K_2$ depending only on  $K_\infty$, $K_{1,2}$, $\Omega$, $T$ and $\eta$ such that the solution of the scheme \eqref{schemeimp} together with \eqref{defflux}-\eqref{dissflux} satisfies for all $\by\in\R^d$
      \[
      \|\tilde{h}_\delta(\cdot,\cdot+\by)-\tilde{h}_\delta(\cdot,\cdot)\|_{L^2(\R^{d+1})}^2\ \leq\ K_2\,|\by|\,(|\by|+\delta)\,,
      \]
      and for all $\tau\in(0,T)$
      \[
      \|\tilde{h}_\delta(\cdot+\tau,\cdot)-\tilde{h}_\delta(\cdot,\cdot)\|_{L^2((0,T-\tau)\times\R^{d})}^2\ \leq\ K_2\,|\tau|\,.
      \]
    \end{lem}
    From the previous results we obtain the following convergences.
    \begin{prop}
      There is a  function $h\in L^\infty(0,T;H^1(\Omega))$ such that as $\delta\rightarrow 0$ and up to  the extraction of a subsequence,
      \[
      \begin{aligned}
	h_\delta\ \longrightarrow&&& h&&\text{strongly in }L^2(\Omega_T),\\
	\nabla^\delta h_\delta\ \longrightarrow&&& \nabla h&&\text{weakly in } \left(L^2(\Omega_T)\right)^d,\\
	f_\delta\ \longrightarrow&&&  f &&\text{strongly in } L^1(\Omega_T),
      \end{aligned}
      \]
      with $f(t,\bx) = \eta^{-1}(\eta(f^\infty(t,\bx))\,h(t,\bx))$, $(t,\bx)-$almost everywhere in $\Omega_T$. Moreover $h-1\in L^\infty(0,T;H^1_0(\Omega))$.
      \label{p:compactness}
    \end{prop}
    \begin{proof}
      The first result is a consequence of the Riesz-Fréchet-Kolmogorov $L^p$ compactness criterion. It yields the convergence of $(\tilde{h}_\delta)_{\delta>0}$ to a function $\tilde{h}$ strongly in $L^2([0,T)\times\R^{d})$. By taking limits in the first space translates estimate, we know that $\tilde{h}\in L^2(0,T;H^1(\R^{d}))$. Since by definition $\tilde{h}_\delta(t,\cdot)\equiv 0$ on $\R^{d}\setminus\overline{\Omega}$, we get that $\tilde{h}\in L^2(0,T;H^1_0(\Omega))$.  Hence $(h_\delta)_\delta$ converges strongly  in $L^2(\Omega_T)$ to $h = 1+ \tilde{h}$  and up to the extraction of a sparser subsequence, it also converges almost everywhere. Since $h_\delta$ and $\eta(\finf)_\delta$ are uniformly bounded, $\eta^{-1}$ is continuous and the cylinder $\Omega_T$ is bounded, one can apply the dominated convergence theorem to obtain the convergence of $f_\delta$ towards $f$. For the weak convergence of discrete gradients we refer to \cite[Lemma 4.4]{chainais_2003_finite} with minor modifications.
    \end{proof}

    \subsection{Convergence of the scheme}
    
    \begin{theo}
      Under hypotheses \eqref{boundinit}-\eqref{orthog} and \eqref{regmesh2}-\eqref{cvstat}, the function $h$ defined in Proposition \ref{p:compactness} is a solution of Equation~\eqref{contprob2} in the sense of Definition~\ref{defsol}.
      \label{t:convergence}
    \end{theo}
    Let us mention that we follow hereafter the methods of proof from \cite{chainais_2003_finite, chainais_2004_finite}.
    \begin{proof}
      Let us consider a test function $\psi\in\mathcal{C}^\infty_c([0,T)\times\Omega)$ and set $\psi_K^n = \psi(t^n,\bx_K)$ for all $K\in\T$ and $n = 0,\dots, N_T$. We suppose that the mesh size $\delta$ is sufficiently small for the inclusion $\mathrm{supp}(\psi)\subset[0,(N_T-1)\Delta t)\times\{\bx\in\Omega,\ d(\bx,\Gamma)>\delta\}$  to  hold. With this assumption sums on exterior edges disappear as one may remark in various terms hereafter. Let us define 
      \[\left\{
      \begin{aligned}
	&T_{10}(\delta)\ =\ -\iint_{\Omega_T}f_\delta(t,\bx)\,\partial_t\psi(t,\bx)\, \dd \bx \dd t - \int_\Omega f^\text{in}_\delta(\bx)\,\psi(0,\bx)\,\dd \bx,\\
	&T_{20}(\delta)\ =\ \iint_{\Omega_T}\eta(f^\infty)_\delta\,\nabla^\delta h_\delta\cdot\nabla\psi\, \, \dd \bx \dd t,\\
	&T_{30}(\delta)\ =\ -\iint_{\Omega_T}h_\delta\,\bU^\infty_\delta\cdot\nabla\psi\, \dd \bx \dd t.
      \end{aligned}\right.
      \]
      From the results of Proposition \ref{p:compactness} and assumption \eqref{cvstat}, it is clear that 
      \[
      T_{10}(\delta)+T_{20}(\delta)+T_{30}(\delta)\longrightarrow -\iint_{[0,T)\times\Omega}\left[f\,\partial_t\psi +
      \left(h\,\bU^\infty -\eta(f^\infty)\,\nabla h\right)\cdot \nabla \psi\right]
      - \int_\Omega f^\text{in}\,\psi(0,\cdot)
      \]
      as $\delta\rightarrow 0$. Now let us show that it also converges to $0$.
      By multiplying the scheme by $\Delta t\,\psi_K^n$ and summing over $n$ and $K$ we obtain 
      \[
      T_1(\delta) + T_2(\delta) + T_3(\delta) = 0,
      \]
      with
      \[\left\{
      \begin{aligned}
	&T_{1}(\delta)\ =\ \sum_{n=0}^{N_T}\sum_{K\in\T}m(K)\,(f_K^{n+1}-f^n_K)\,\psi_K^n,\\
	&T_{2}(\delta)\ =\ \sum_{n=0}^{N_T}\Delta t\sum_{K\in\T}\sum_{\substack{\sigma\in\Ei\\ \sigma= K|L}}\tau_\sigma\,\eta(\finf)_\sigma\, (h_K^{n+1}-h_L^{n+1})\,\psi_K^n,\\
	&T_{3}(\delta)\ =\ \sum_{n=0}^{N_T}\Delta t\sum_{K\in\T}\sum_{\substack{\sigma\in\Ei\\ \sigma= K|L}}m(\sigma)\left[U_{K,\sigma}^{\infty,+} g(h_K^{n+1},h_L^{n+1}) - U_{K,\sigma}^{\infty,-}g(h_L^{n+1},h_K^{n+1}) \right]\,\psi_K^n\,.
      \end{aligned}\right.
      \]
      Let us show that each $T_i(\delta)$ gets asymptotically close to $T_{i0}(\delta)$, for $i=1,\,2$ or $3$, as  $\delta$ goes to $0$.
      After a discrete integration by parts, one gets
      \[
      \begin{aligned}
	&T_{1}(\delta)&=&&& \sum_{n=0}^{N_T}\sum_{K\in\T}m(K)\,f_K^{n+1}\,(\psi_K^n-\psi_K^{n+1}) - \sum_{K\in\T}m(K)\,f_K^{0}\,\psi_K^0\,,\\
	&&=&&&-\sum_{n=0}^{N_T}\sum_{K\in\T}\int_{t^n}^{t^{n+1}}\int_K f_K^{n+1}\,\partial_t\psi(t,\bx_K)\,\dd \bx\dd t - \sum_{K\in\T}\int_K f_K^{0}\,\psi(0,\bx_K)\,\dd \bx\,,
      \end{aligned}
      \]
      which yields 
      \[
      |T_{1}(\delta)-T_{10}(\delta)|\,\leq\,\delta\ (T+1)\,m(\Omega)\,\|f_\delta\|_{L^\infty(\Omega_T)}\,\|\psi\|_{\mathcal{C}^2(\overline{\Omega_T})}\rightarrow 0\,,
      \]
      as $\delta\rightarrow 0$ by Proposition \ref{estim}. Concerning the diffusion term one has, integrating by parts
      \[
      T_{2}(\delta)\ =\ \sum_{n=0}^{N_T}\Delta t\sum_{\substack{\sigma\in\Ei\\ \sigma= K|L}}\tau_\sigma\,\eta(\finf_\sigma)\, (h_K^{n+1}-h_L^{n+1})\,(\psi_K^n-\psi_L^n)
      \]
      and
      \[
      T_{20}(\delta)\ =\ \sum_{n=0}^{N_T}\sum_{\substack{\sigma\in\Ei\\ \sigma= K|L}}m(\sigma)\,\eta(f^\infty)_\sigma\,(h_L^{n+1}-h_K^{n+1})\int_{t^n}^{t^{n+1}}\frac{1}{m(T_\sigma)}\int_{T_\sigma}\nabla\psi\cdot\bn_{K,\sigma}\, \, \dd \bx \dd t.
      \]
      Now just note that by the orthogonality hypothesis \eqref{orthog} on the mesh and the regularity of $\psi$ one has, 
      \begin{equation}
	\left|\psi_L^n -\psi_K^n - \frac{1}{\Delta t}\int_{t^n}^{t^{n+1}}\frac{1}{m(T_\sigma)}\int_{T_\sigma}\nabla\psi\cdot d_\sigma\bn_{K,\sigma}\, \, \dd \bx \dd t\right| \leq \|\psi\|_{\mathcal{C}^2(\overline{\Omega_T})}\,\delta.
	\label{taylor}
      \end{equation}
      Therefore, with a Cauchy-Schwartz inequality and thanks to  the regularity of the mesh one obtains 
      \[
      |T_{2}(\delta)-T_{20}(\delta)|\leq\delta\ \sqrt{\xi_2\,T}\,\|\eta(f^\infty)_\delta\|_{L^2(\Omega)}\,\|h_\delta\|_{1,2,\D}\,\|\psi\|_{\mathcal{C}^2(\overline{\Omega_T})}\rightarrow 0\,,
      \]
      as $\delta\rightarrow 0$, by Proposition \ref{estim} and assumption \eqref{cvstat}.
      Finally let us deal with the convection term. As it is classical we transform the upwind form of $T_{3}(\delta)$ into the sum of a numerical diffusion and a centered flux yielding $T_{3}(\delta) = T_{31}(\delta) + T_{32}(\delta)$, with
      \[
      \begin{aligned}
	&T_{31}(\delta)&=&&&\frac{1}{2}\sum_{n=0}^{N_T}\Delta t\sum_{K\in\T}\sum_{\substack{\sigma\in\Ei\\ \sigma= K|L}}m(\sigma)\left|U_{K,\sigma}^\infty\right| \left(g(h_K^{n+1},h_L^{n+1})-g(h_L^{n+1},h_K^{n+1})\right)\,\psi_K^n\,,\\
	&&=&&&\frac{1}{2}\sum_{n=0}^{N_T}\Delta t\sum_{\substack{\sigma\in\Ei\\ \sigma= K|L}}m(\sigma)\left|U_{K,\sigma}^\infty\right|\left(g(h_K^{n+1},h_L^{n+1})-g(h_L^{n+1},h_K^{n+1})\right)\,(\psi_K^n-\psi_L^n)\,,\\
      \end{aligned}
      \]
      whereas $T_{32}(\delta)$ is
      \[
      \begin{aligned}
	&T_{32}(\delta)&=&&&\frac{1}{2}\sum_{n=0}^{N_T}\Delta t\sum_{K\in\T}\sum_{\substack{\sigma\in\Ei\\ \sigma= K|L}}m(\sigma)\phantom{|}U_{K,\sigma}^\infty\phantom{|}\left(g(h_K^{n+1},h_L^{n+1})+g(h_L^{n+1},h_K^{n+1})\right) \,\psi_K^n\,,\\
	&&=&&&\sum_{n=0}^{N_T}\Delta t\sum_{\substack{\sigma\in\Ei\\ \sigma= K|L}}m(\sigma)\phantom{|}U_{K,\sigma}^\infty\phantom{|}g(h_K^{n+1},h_L^{n+1}) \,(\psi_K^n-\psi_L^n)\,,\\
	&&=&&&T_{321}(\delta)\ +\ T_{322}(\delta)\,,
      \end{aligned}
      \]
      with $T_{321}(\delta)$ and $T_{322}(\delta)$ given by
      \[
      \begin{aligned}
	&T_{321}(\delta)&=&&&\sum_{n=0}^{N_T}\Delta t\sum_{\substack{\sigma\in\Ei\\ \sigma= K|L}}m(\sigma)\phantom{|}U_{K,\sigma}^\infty\phantom{|}(g(h_K^{n+1},h_L^{n+1}) - h_K^{n+1})\,(\psi_K^n-\psi_L^n)\,,\\
	&T_{322}(\delta)&=&&&\sum_{n=0}^{N_T}\Delta t\sum_{\substack{\sigma\in\Ei\\ \sigma= K|L}}m(\sigma)\phantom{|}U_{K,\sigma}^\infty\phantom{|}h_K^{n+1}\,(\psi_K^n-\psi_L^n)\,.
      \end{aligned}
      \]
      Similarly we introduce the decomposition $T_{30}(\delta) = T_{310}(\delta) + T_{320}(\delta)$ with
      \[
      \begin{aligned}
	&T_{310}(\delta)&=&&&-\sum_{n=0}^{N_T}\sum_{\substack{\sigma\in\Ei\\ \sigma= K|L}}\int_{t^n}^{t^{n+1}} \int_{T_\sigma\cap L}(h_L^{n+1}-h_K^{n+1})\,\bU^\infty_\delta\cdot\nabla\psi\, \dd \bx \dd t\,,\\
	&T_{320}(\delta)&=&&&-\sum_{n=0}^{N_T}\sum_{\substack{\sigma\in\Ei\\ \sigma= K|L}}\int_{t^n}^{t^{n+1}} \int_{T_\sigma}h_K^{n+1}\,\bU^\infty_\delta\cdot\nabla\psi\, \dd \bx \dd t\,.
      \end{aligned}
      \]
      Let us prove that $T_{31}(\delta)$, $T_{310}(\delta)$, $T_{321}(\delta)$ and $T_{322}(\delta)-T_{320}(\delta)$ converge to zero when $\delta$ goes to zero. First, notice that using the Lipschitz continuity of $g$, the definition of $\bU_\delta^\infty$ and eventually the Cauchy-Schwartz inequality yields
      \[
      |T_{321}(\delta)|+|T_{31}(\delta)|\,\leq\,2\,\delta\,\|g\|_{W^{1,\infty}([0,\,\|h_\delta\|_{L^\infty}]^2)}\,\,\|\psi\|_{\mathcal{C}^1(\overline{\Omega_T})}\,\|h_\delta\|_{1,2,\D}\,\|\bU_\delta^\infty\|_{L^2(\Omega_T)}.
      \]
      Then, similarly one obtains
      \[
      |T_{310}(\delta)|\,\leq\,\delta\,\|\psi\|_{\mathcal{C}^1(\overline{\Omega_T})}\,\|h_\delta\|_{1,2,\D}\,\|\bU_\delta^\infty\|_{L^2(\Omega_T)},
      \]
      and by Proposition \ref{estim} and hypotheses \eqref{condg} \eqref{cvstat} both right-hand sides go to $0$ as $\delta$ goes to $0$. Finally
      \[
      T_{322}(\delta)-T_{320}(\delta)\, =\, \sum_{n=0}^{N_T}\sum_{\substack{\sigma\in\Ei\\ \sigma= K|L}}m(\sigma)\,\,\phantom{|}U_{K,\sigma}^\infty\phantom{|}h_K^{n+1}\,\int_{t^n}^{t^{n+1}}\left(\psi_K^n-\psi_L^n- \frac{1}{m(T_\sigma)}\int_{T_\sigma}\,\nabla\psi\cdot\,d_\sigma\bn_{K,\sigma}\, \dd \bx\right)\dd t,
      \]
      and by using \eqref{taylor} we obtain,
      \[
      |T_{322}(\delta)-T_{320}(\delta)|\,\leq\,\delta\,\sqrt{T\,m(\Omega)}\,\|\psi\|_{\mathcal{C}^2(\overline{\Omega_T})}\,\|h_\delta\|_{L^\infty(\Omega_T)}\,\|\bU_\delta^\infty\|_{L^2(\Omega_T)}.
      \]
      Hence, $|T_{i}(\delta) - T_{i0}(\delta)|\rightarrow0$
      as $\delta\rightarrow 0$, for any $i\in\{1,2,3\}$.
    \end{proof}

    \section{Numerical simulations}
    \label{numerics}

    \subsection{Implementation}

    Before presenting our numerical results let us state an important remark concerning the implementation of our scheme. By Theorem~\ref{longtimetheo}, we expect the solution $h_\delta$ of the scheme \eqref{schemeimp},\eqref{defflux}-\eqref{dissflux} or \eqref{schemeexp}-\eqref{dissflux} to converge to $1$ when time goes to infinity. Due to floating point numbers repartition there can be non-negligible numerical errors in the computation of difference operators leading to saturation of $\phi$-entropies. To avoid this issue, the scheme should be implemented in the following way. We introduce new unknowns defined by,
    \[
    \tilde{f}^n_K\ =\ f^n_K - \finf_K,\qquad \tilde{h}^n_K\ =\  h^n_K - 1\,,
    \]
    for $n\in\{0,\dots, N_T\}$ and $K\in\T$.
    One can readily check that in the upwind case $g(s,t)=s$, thanks to \eqref{stat}, the schemes remain unchanged by replacing $f$ by $\tilde{f}$ and $h$ by 
    $\tilde{h}$. Even for other $g$, this modified scheme is the discretization of 
    \[
     \frac{\partial}{\partial t}(f-\finf) \,+\, {\nabla}\cdot\left(\bU^\infty\,(h-1) \,-\,  \eta(f^\infty)\nabla (h-1)\right) \,=\,0\,,
    \]
    which is the same as \eqref{contprob2} since $\finf$ is steady and $\bU^\infty$ is incompressible.
    
    The new unknowns converge to $0$ when time goes to infinity and thus differences are computed with a better precision.
    Hence one should solve the scheme (and compute $\phi$-entropies, dissipations, \emph{etc...}) on the new unknowns. Moreover boundary conditions on $\tilde{h}$ become homogeneous, which actually makes the implementation easier. 
    
    \smallskip
    
    In the following test cases, the implicit scheme \eqref{schemeimp},\eqref{defflux}-\eqref{dissflux}  is used only for linear models in Section~\ref{test0} - \ref{test2}. The nonlinear model in Section~\ref{test3} is solved with the explicit scheme  \eqref{schemeexp}-\eqref{dissflux}. For the first test cases we use the linear solver SuperLU \cite{overview_2005_li}, which provides efficient results for large sparse and non-symmetric systems by performing a sparsity-preserving LU factorization. Since the advection field is steady in Equation \eqref{contprob2}, the resolution matrix can be factorized only once at the beginning of the simulation.
    \subsection{Proof of concept}\label{test0}
    In this part, we provide a numerical experiment showing the spatial accuracy of our scheme, especially in the long-time dynamics. It is performed on the following one dimensional toy model.
    The test case is the linear ($\eta(s) = s$) drift-diffusion equation  \eqref{contprob} endowed with the (scalar) advection $E(x) = 1$ and set on the  domain $\Omega = (0,1)$.  With the boundary conditions $f(t,0) = 2$ and
    $f(t,0) = 1+\exp(1)$ the function 
    \[
    f(t,x) = 1 + \exp(x) + \exp\left(\frac{x}{2}-\left(\pi^2+\frac 14\right)t\right)\sin(\pi x)
    \]
    is the exact solution of \eqref{contprob} and converges to the stationary state
    \[
    f^\infty(x) = 1 + \exp(x)\,,
    \]
    as time goes to infinity. 
    
    In order to illustrate the advantage of our approach compared to the
    one consisting in a direct approximation of (\ref{contprob}), we
    perform numerical simulations using our
    scheme \eqref{schemeimp}, \eqref{defflux}-\eqref{dissflux}  and a classical finite volume discretization of \eqref{contprob}
    with upwind flux for the convective term and two-points
    approximation of the gradient for the diffusive term. Both schemes are implicit in time and we shall call the former ``$\phi$-entropic'' and the latter ``upwind''.
    
    The domain is discretized by the regular Cartesian mesh $\T = \{K_i:=(x_i-\Delta x/2,\,x_i+\Delta x/2),\ i =0,\dots,N-1\}$ where $x_i = \Delta x/2 + i\Delta x$ and $\Delta x = 1/N$. Concerning the time discretization, the final time is $T=5$ and we choose the small time step $\Delta t = 10^{-6}$ in order to minimize the error due to time discretization. For the implementation of our scheme we explicitly compute both the discrete steady solution and the steady flux.

    We define the $L^p$ error at time $t$ between the reconstruction of the approximate solution $f_N$ (corresponding to $f_\delta$ with previous notation) and the projection of the analytic solution on the mesh by
    \[
    \epsilon_N^p(t) = \|\Pi_Nf(t,\cdot)-f_N(t,\cdot)\|_{L^p(\Omega)}
    \]
    where $\Pi_Nf(t,x) = f_i^n$ if  $(t,x)\in[t^n,t^{n+1})\times K_i$
    with $f_i^n$ a numerical approximation of the average of $f(t^{n+1},\cdot)$ on the cell $K_i$, computed with the trapezoidal rule. In Table \ref{order1global}, we measure, for both
    schemes and for different number of points $N$, the global error and experimental order of accuracy, respectively given by
    \[
    e_N^p = \sup_{t\in[0,T)}\epsilon_N^p(t),\quad k_{2N}^p = |\log(e_{2N}) - \log(e_N)|/ \log(2)\,.
    \] 
    \begin{table}[H]
      \[
      \begin{array}{|c||c|c||c|c|||| c|c|c||c|c|}
	\hline
	N&\text{Error }e_N^1&\text{Order}&\text{Error }e_N^1&\text{Order}     &\text{Error }e_N^\infty&\text{Order}&\text{Error }e_N^\infty&\text{Order}\\
	&\mathbf{\phi}\text{-entropic} &&\text{Upwind}&    &\phi\text{-entropic} &&\text{Upwind}&\\
	\hline
	20&2.07.10^{-3}&		&4.28.10^{-3}&	        &3.33.10^{-3}&		&7.38.10^{-3}&	\\
	\hline
	40&1.21.10^{-3}&	0.77&	2.36.10^{-3}&	0.86    &1.93.10^{-3}&	0.79&	4.03.10^{-3}&	0.87\\
	\hline
	80&6.45.10^{-4}&	0.91&	1.24.10^{-3}&	0.93    &1.02.10^{-3}&	0.91&	2.11.10^{-3}&	0.94\\
	\hline
	160&3.30.10^{-4}&	0.97&	6.30.10^{-4}&	0.97    &5.22.10^{-4}&	0.97&	1.07.10^{-3}&	0.98\\
	\hline
	320&1.64.10^{-4}&	1.01&	3.15.10^{-4}&	1.00    &2.59.10^{-4}&	1.01&	5.31.10^{-4}&	1.00\\
	\hline
	640&7.87.10^{-5}&	1.06&	1.55.10^{-4}&	1.03    &1.26.10^{-4}&	1.06&	2.61.10^{-4}&	1.02\\
	\hline
	1280&3.57.10^{-5}&      1.14&	7.38.10^{-5}&	1.07    &5.65.10^{-5} &1.14&	1.25.10^{-4}&	1.06\\
	\hline
      \end{array}
      \]
      \caption{{\bf Proof of concept.} Experimental spatial order of convergence in $L^1$ and $L^\infty$. \label{order1global}}
    \end{table}
    
    Both schemes are first order accurate, but we observe that our $\phi$-entropic scheme  \eqref{schemeimp}, \eqref{defflux}-\eqref{dissflux}  is performing better
    since the numerical error is smaller than the classical upwind
    scheme. Furthermore in Figure~\ref{fig:cv:1}, we observe that for
    large time the numerical error corresponding to the entropy
    preserving scheme  \eqref{schemeimp},\eqref{defflux}-\eqref{dissflux}  decays to zero and
    the error becomes negligible compared to that of the upwind scheme.
    \begin{figure}[H]
      \begin{center}
	\begin{tabular}{cc}
	  \includegraphics[width=7.5cm]{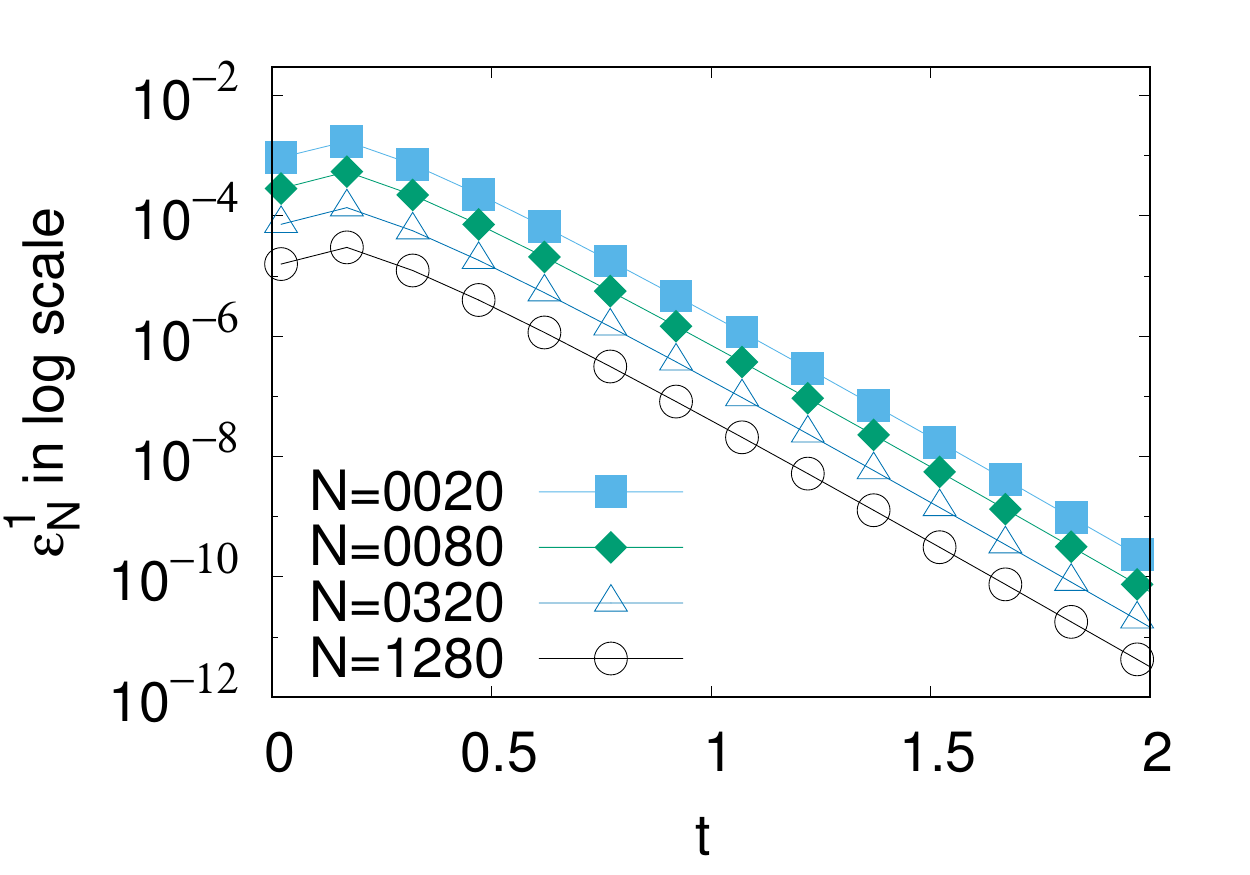} &   
	  \includegraphics[width=7.5cm]{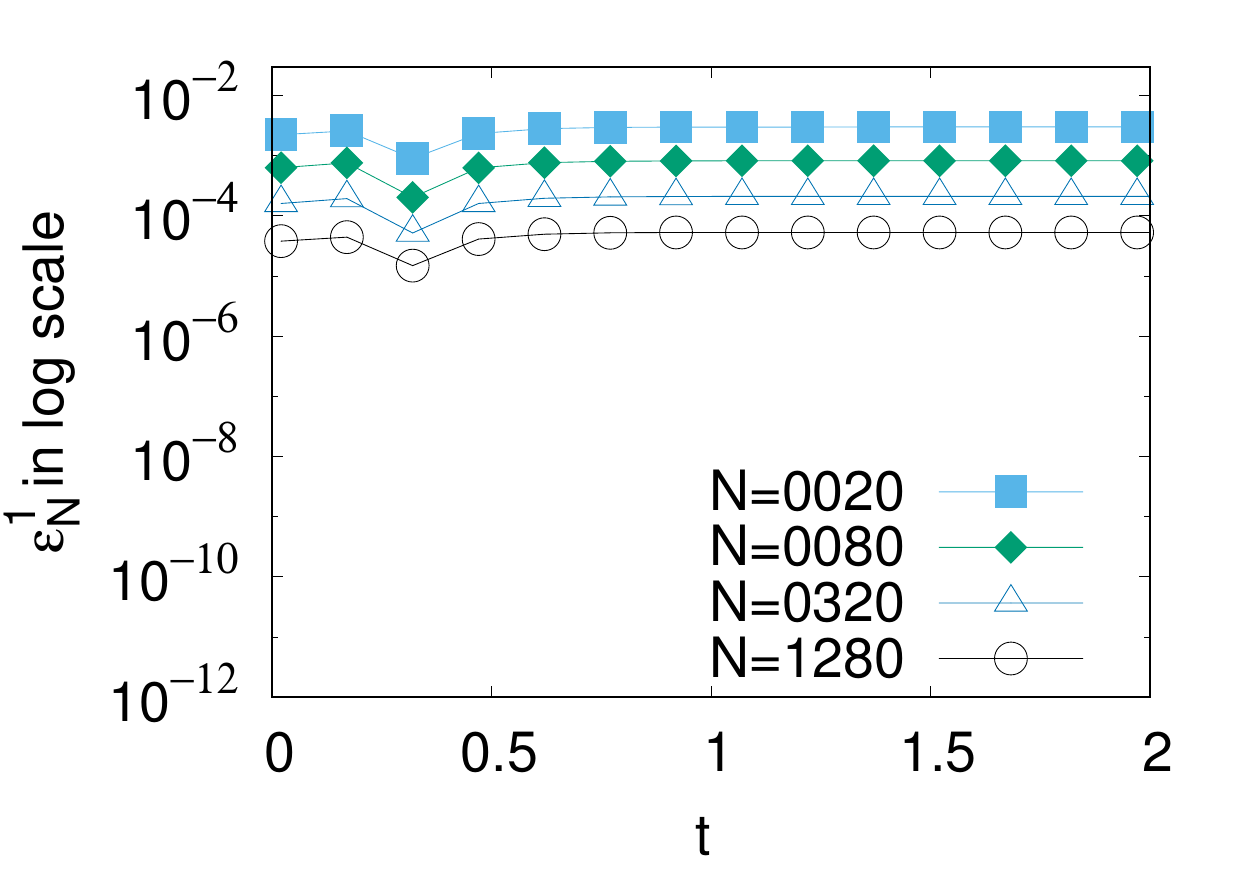}
	  \\
	  (a) & (b)
	\end{tabular}
	\caption{\label{fig:cv:1}
	  {\bf Proof of concept.} Time evolution of the $\epsilon_N^1$
	  error for (a) the $\phi$-entropic scheme and (b) the classical
	  upwind scheme.}
      \end{center}
    \end{figure}
    The accurate long-time behavior is confirmed by Figure~\ref{fig:cv:2} where the time variation of the distance to the discrete solution is represented. While the upwind scheme saturates quickly, our scheme reproduces perfectly the exponential decay to $0$ of the solution, even if $N=40$, illustrating the result of Theorem~\ref{longtimetheo}.
    \begin{figure}[H]
      \includegraphics[width=7.5cm]{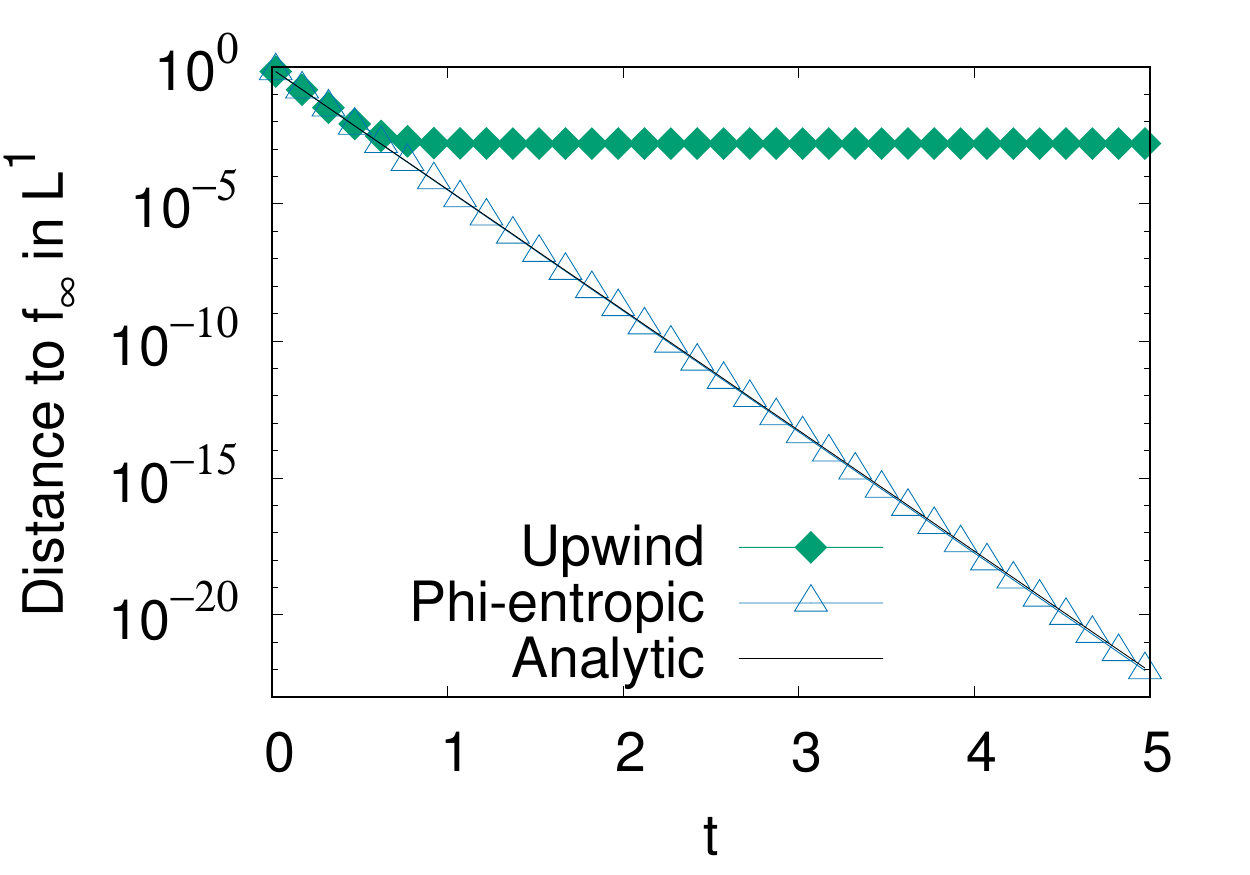}
      \caption{\label{fig:cv:2}{\bf Proof of concept.} Time evolution of the $L^1(\Omega)$ distance to the discrete steady state for $N=40$.}
    \end{figure}

    \subsection{Fokker-Planck with magnetic field }\label{test1}
    We now consider the two-dimensional version of homogeneous Fokker-Planck equation with an
    external magnetic field
    $$
    \left\{
    \begin{array}{l}
      \displaystyle \frac{\partial f}{\partial t} + b\,\bv^\perp\cdot\nabla_\bv f = \nabla_\bv \cdot(\bv f + \nabla_\bv f) \quad{\rm in}\quad\RR^+\times\RR^2\,,
      \\
      \;
      \\
      \ds f(t=0) = f_0 \quad{\rm in}\quad \RR^2\,.
    \end{array}
    \right.
    $$
    The external magnetic field is along a third direction that is orthogonal to the plane under consideration and has amplitude $b= 4$. Compared to the 3D case, the vector $\bv^\perp\, =\, (v_y, -v_x)$ replaces the cross product between $\bv$ and the direction of the magnetic field. The initial datum  $f_0$ is given by the sum of two Gaussian distributions
    $$
    f_0(\bv) = \frac{1}{2\pi}\left[ \alpha\exp\left(-\frac{|\bv-\bv_1|^2}{2}\right)  \,+\, (1-\alpha)\exp\left(-\frac{|\bv-\bv_2|^2}{2}\right) \right],
    $$
    with $\alpha=3/4$, $\bv_1=(-1,2)$ and $\bv_2=(2,-1)$.

    This equation is solved numerically in a bounded domain
    $\Omega=(-8,8)^2$ on various regular Cartesian meshes from $N=40^2$ to $N=640^2$
    points. We use our implicit scheme \eqref{schemeimp},\eqref{defflux}-\eqref{dissflux} with a time step $\Delta t=0.001$ until the final time $T=10$. We choose  non homogeneous Dirichlet boundary conditions $f^b=f^\infty$,
    where $f^\infty$ is the  steady
    state, that is,  the Maxwellian distribution 
    $$
    f^\infty(\bv) \,=\, \frac{1}{2 \pi} \exp\left(-\frac{|\bv|^2}{2}\right). 
    $$
    Here the knowledge of the steady state $f^\infty$ allows us to compute the steady flux $\F[K]^\infty$ analytically. Indeed, since $\bv \finf + \nabla_\bv \finf = 0$, one only needs to evaluate
    \[
    \F[K]^\infty = b\,\int_\sigma\,\bn_{K,\sigma}\cdot\bv^\perp\,f^\infty = b\int_\sigma\bn_{K,\sigma}^\perp\cdot\nabla_\bv f^\infty\,,
    \]
    which is, up to a multiplicative constant, the difference between $\finf$ evaluated at each endpoint of $\sigma$.
    
    We recall that the physical entropy generating function is $\phi_1:s\mapsto s\log(s)-s+1$. In Figure~\ref{fig:fp:1}, we represent the time evolution
    of the entropy $H_{\phi_1}$, the physical dissipation $D_{\phi_1}$
    and the ratio between the numerical and physical dissipation
    $C_{\phi_1}/D_{\phi_1}$ in log scale.  On one hand, $H_{\phi_1}$ and $D_{\phi_1}$ decay exponentially and are well approximated when $N$ is larger than $160^2$. Figure~\ref{fig:fp:1} (a) and (b) also illustrate the convergence to equilibrium
    at exponential rate, when time goes to infinity. On the other hand, the numerical dissipation $C_{\phi_1}$  converges to zero when the space step goes to
    zero, but since the scheme is only first order accurate, it is
    relatively slow. Besides, the  numerical dissipation  $C_{\phi_1}$
    also converges to zero at the same exponential rate than  $D_{\phi_1}$,
    hence it does not affect the accuracy on the decay
    rate for large time. Also note that for the chosen meshes the numerical dissipation is smaller than the physical dissipation. From these numerical experiments, we get some
    numerical evidence of the uniform accuracy of the scheme with respect to time.

    \begin{figure}[H]
      \begin{center}
	\begin{tabular}{cc}
	  \includegraphics[width=.49\linewidth]{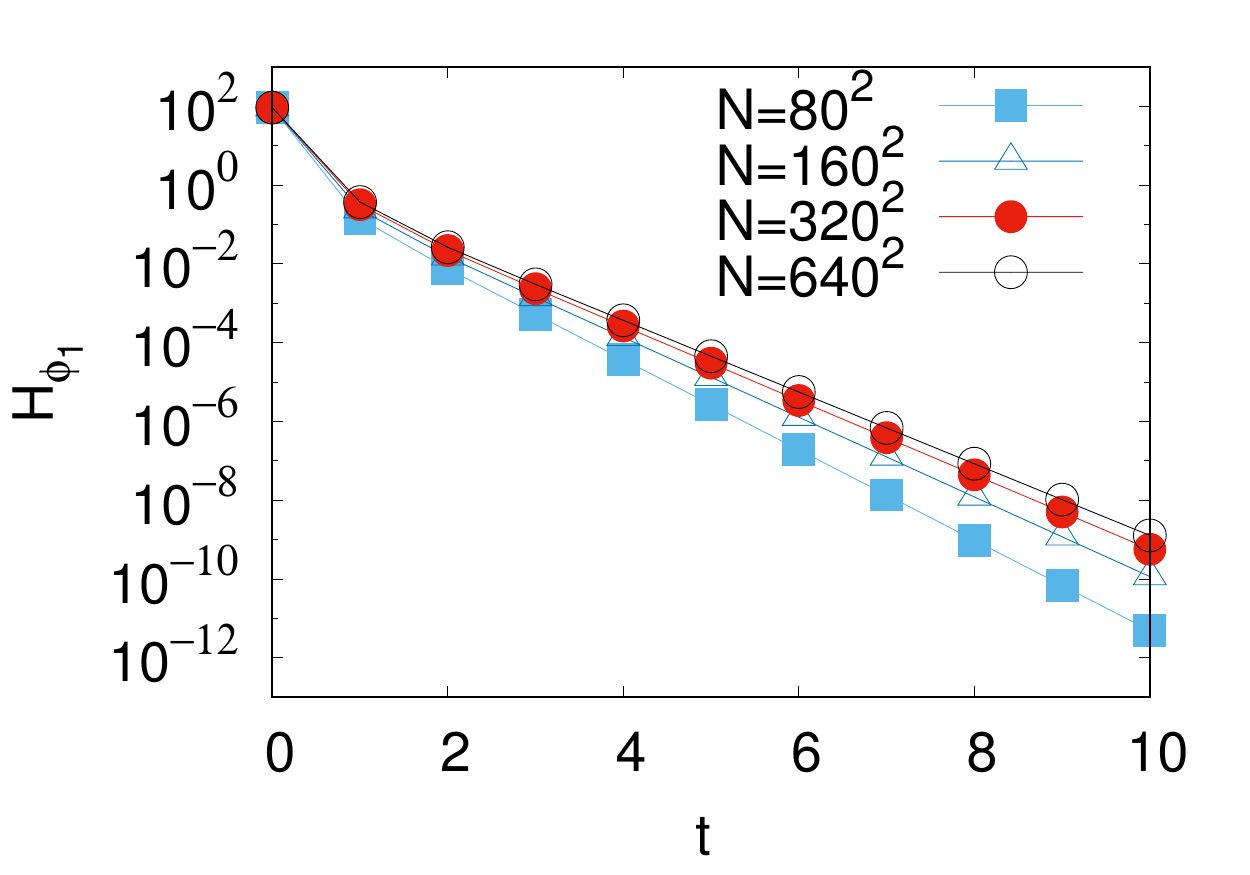} &   
	  \includegraphics[width=.49\linewidth]{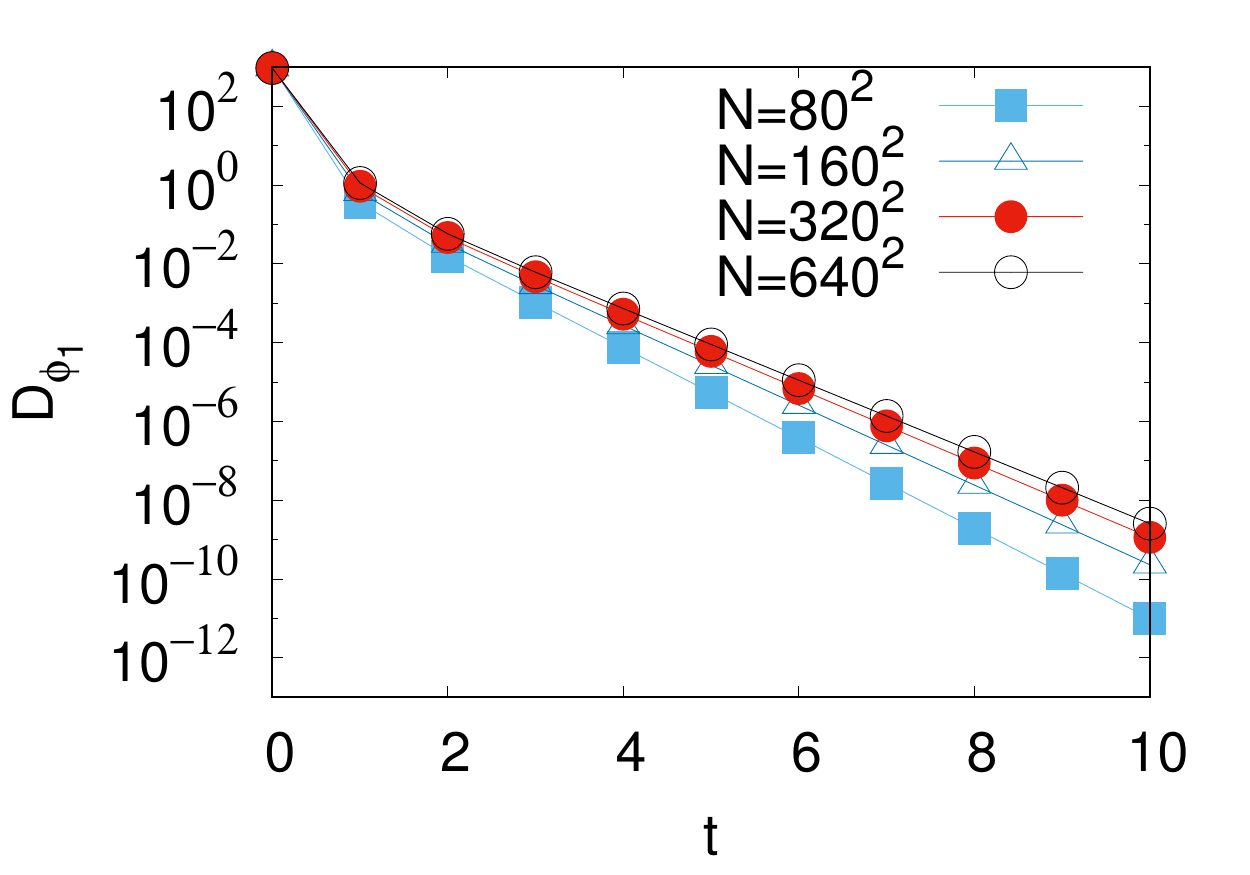}
	  \\
	  (a) & (b)
	\end{tabular}
	\begin{tabular}{c}
	  \includegraphics[width=.49\linewidth]{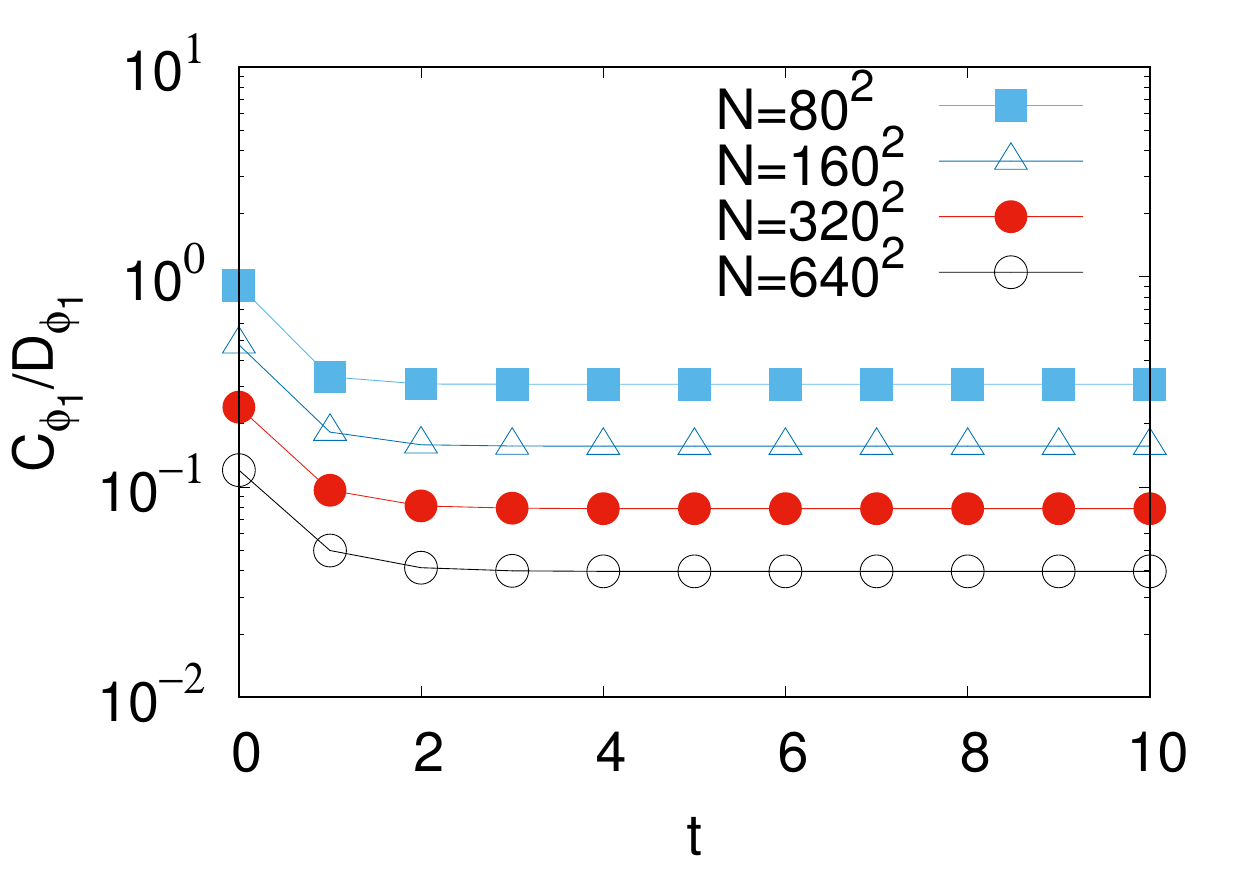}
	  \\
	  (c)
	\end{tabular}
	\caption{\label{fig:fp:1}
	  {\bf Fokker-Planck equation with magnetic field.}  Time
	  evolution of (a) the entropy $H_{\phi_1}$ (b) the physical dissipation $D_{\phi_1}$
	  and (c) the (normalized) numerical dissipation $C_{\phi_1}/H_{\phi_1}$.}
      \end{center}
    \end{figure}
    
    \begin{rema}
      For the same mesh if we take a larger magnetic field, the numerical
      dissipation can become larger than the physical one. Even if both of them
      seem to converge to zero with the same decay rate, the dissipation is
      amplified. 
    \end{rema}

    Finally, in Figures \ref{fig:fp:3}, we propose the
    time evolution of the distribution function at different times. The
    black lines represent the isovalues $f(t,\bv)\equiv 1.10^{-5}$, $3.10^{-4}$, $5.10^{-3}$, $2.5.10^{-3}$, $7.10^{-2}$, $1.2.10^{-1}$, $1.5.10^{-1}$ of the
    distribution function.  We observe the effect of the
    magnetic field by the rotation of the two bumps and under the effect
    of the Fokker-Planck operator, the solution converges to a Maxwellian
    distribution.   
    
    \begin{figure}[H]
      \begin{center}
	\begin{tabular}{cc}
	  \includegraphics[width=.49\linewidth]{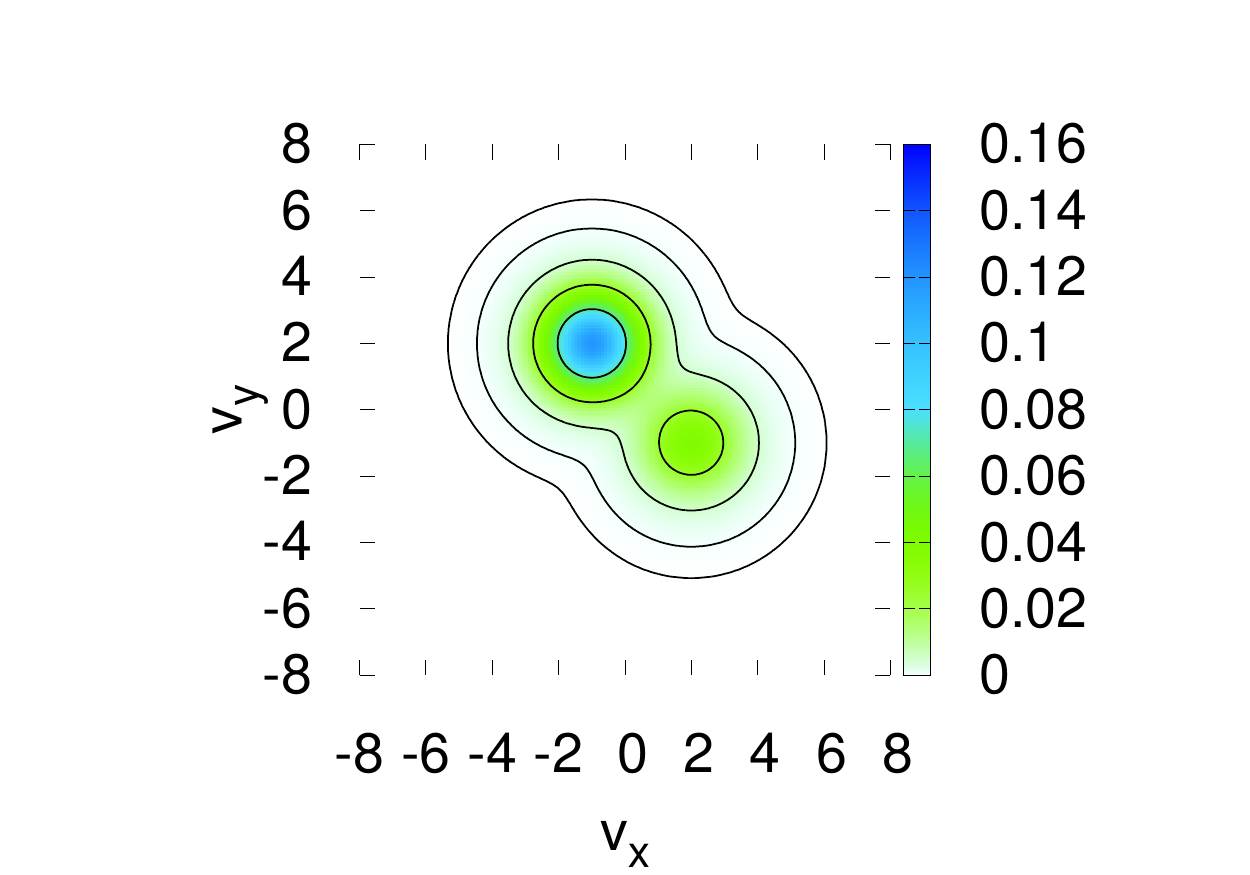} &   
	  \includegraphics[width=.49\linewidth]{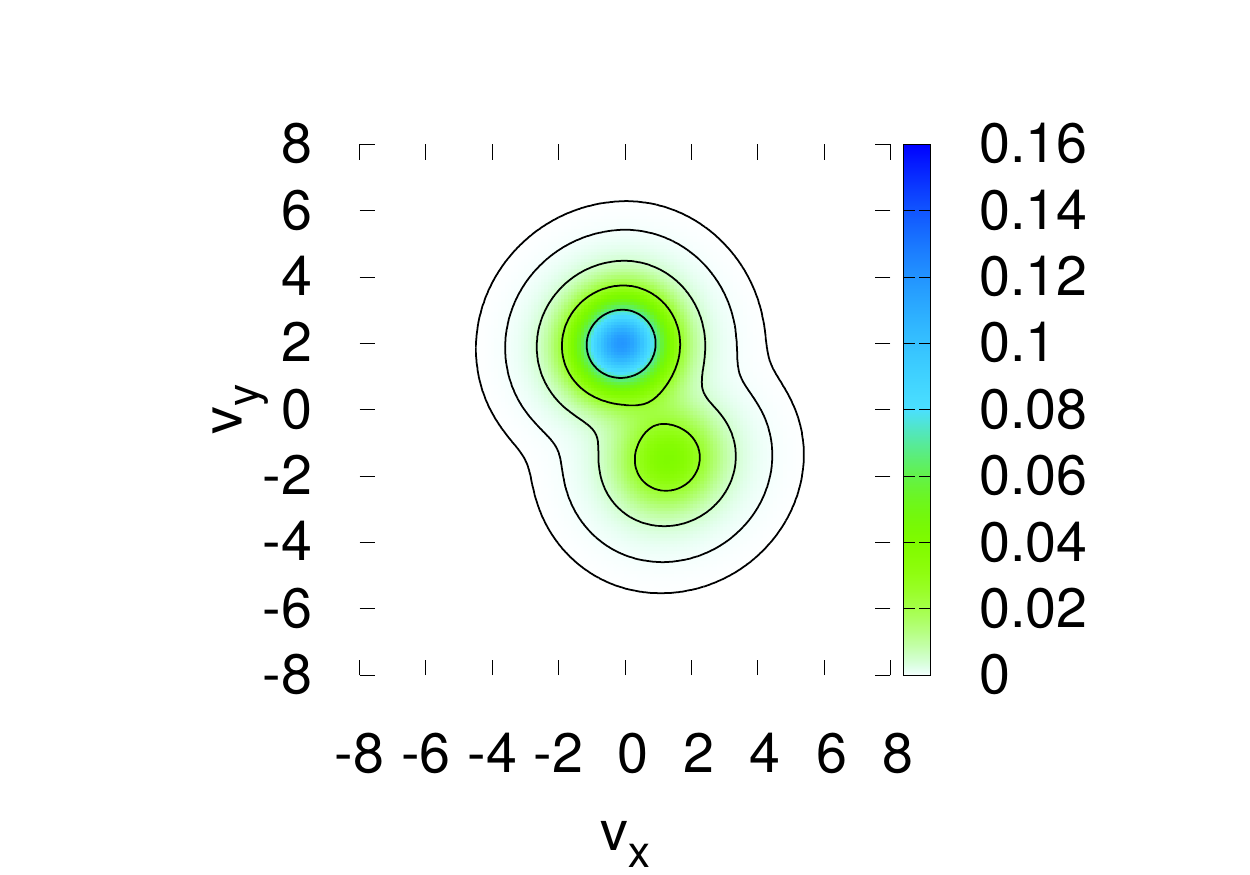}    
	  \\
	  $t=0$ & $t=0.1$
	  \\
	  \includegraphics[width=.49\linewidth]{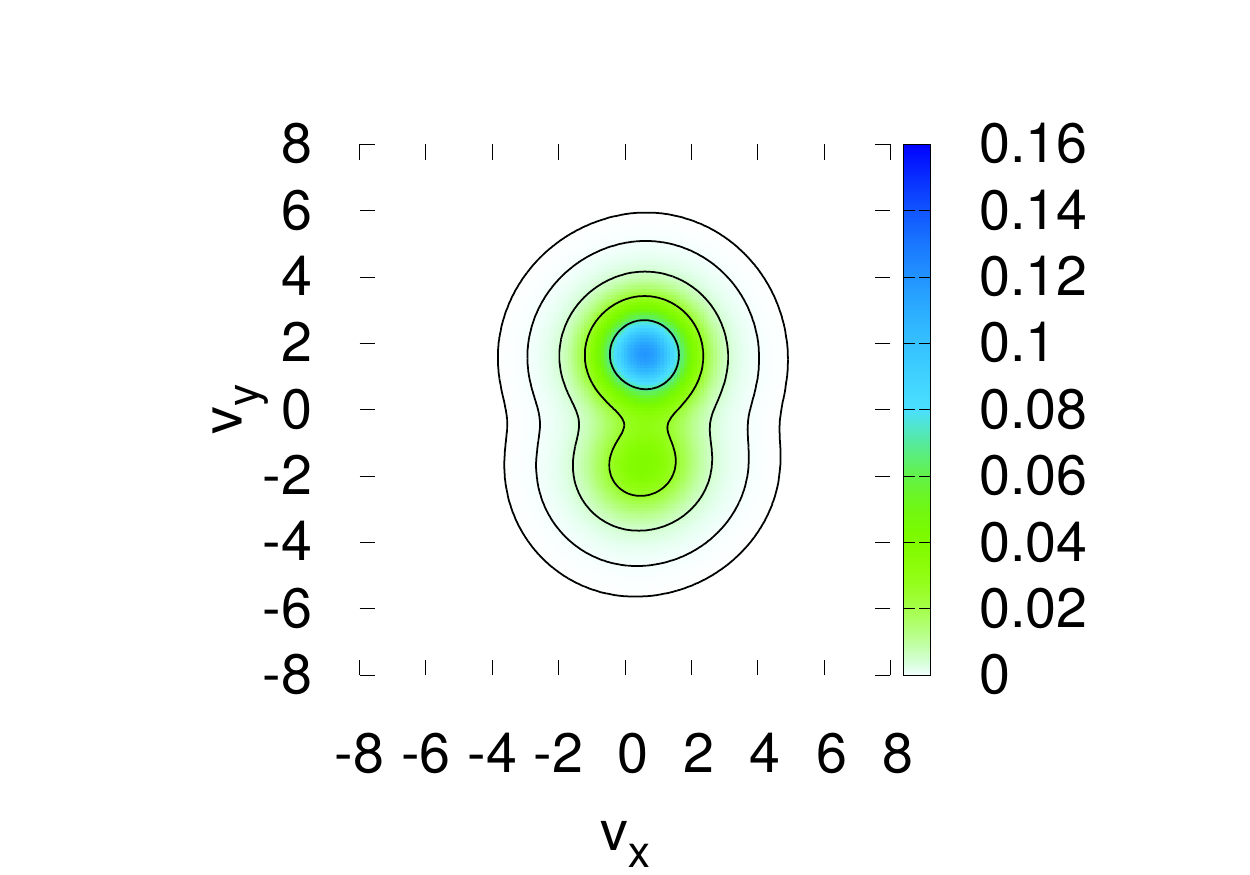} &
	  \includegraphics[width=.49\linewidth]{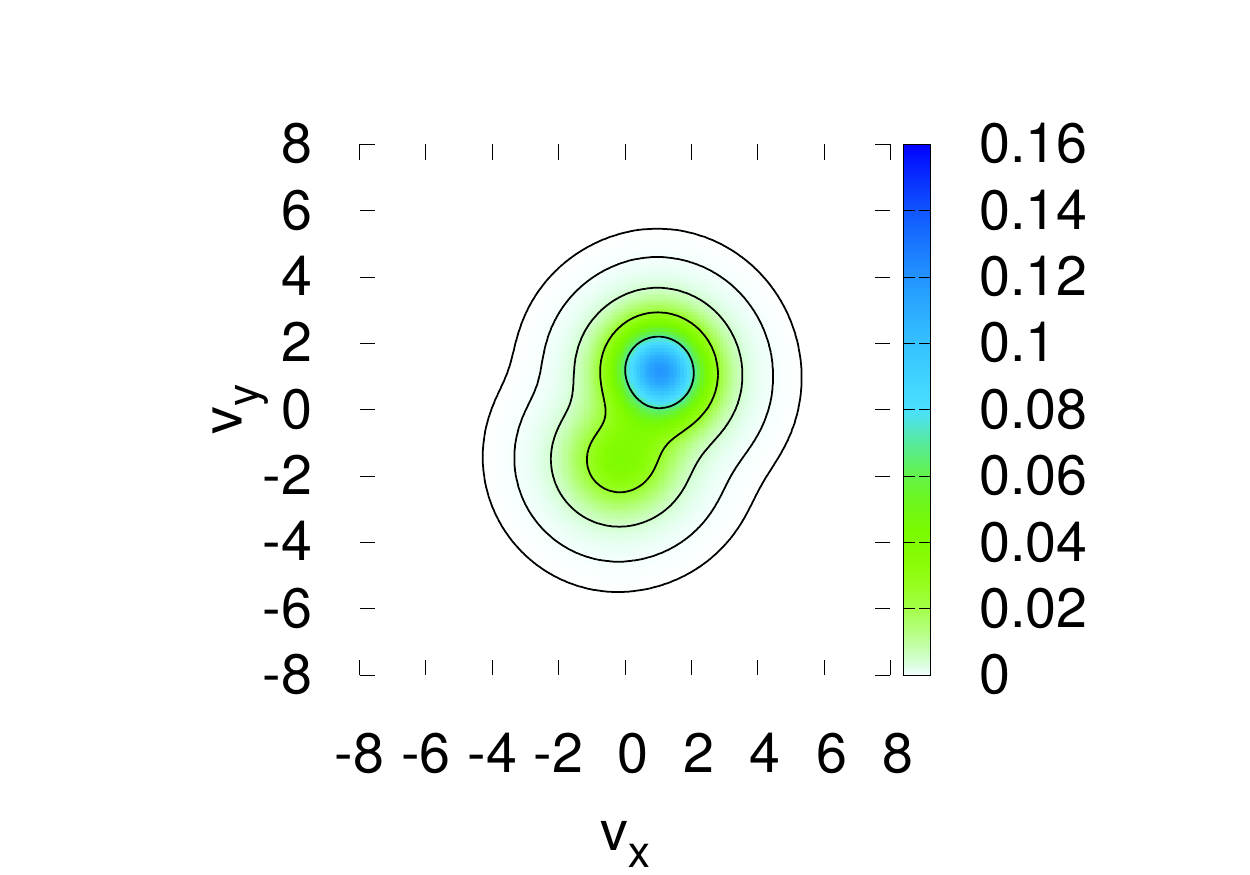}
	  \\
	  $t=0.2$ & $t=0.3$
	  \\
	  \includegraphics[width=.49\linewidth]{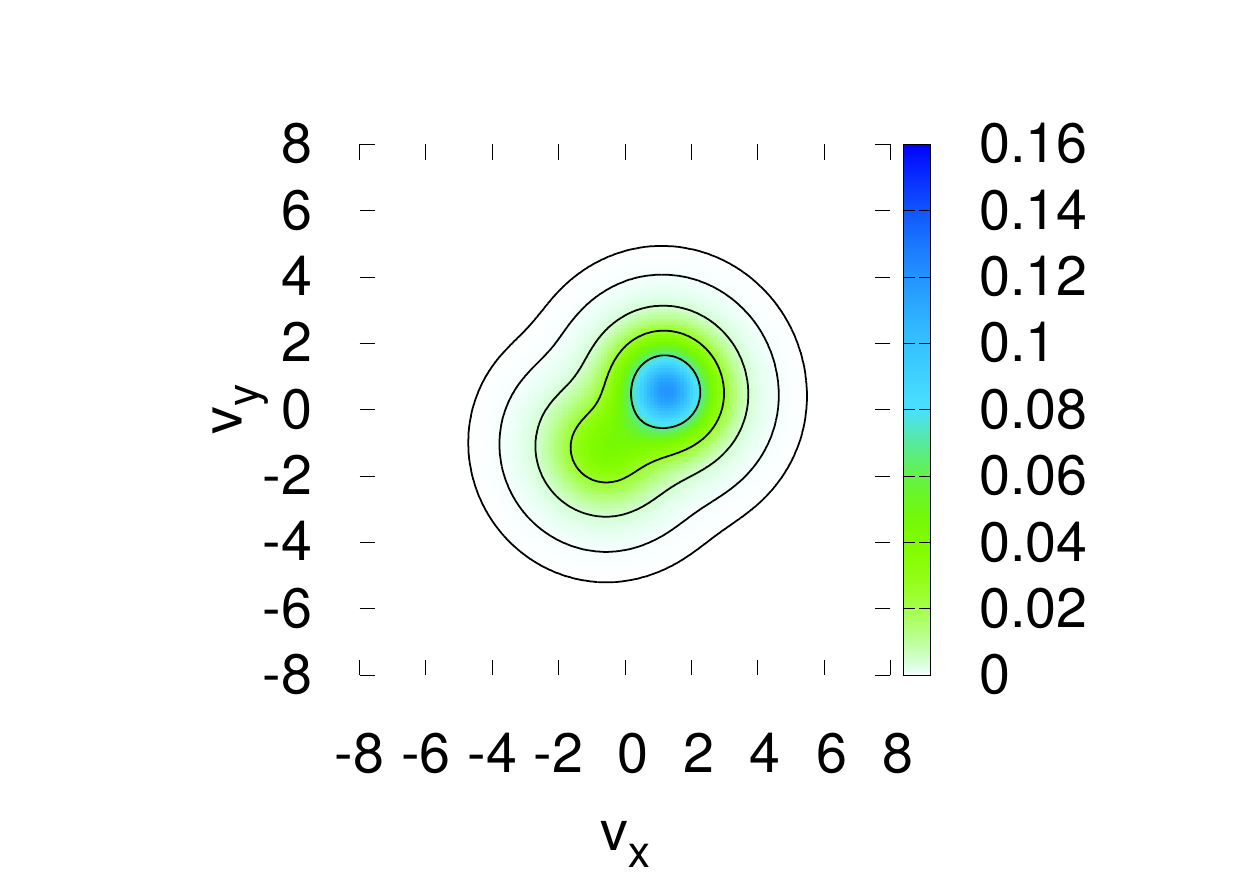} &   
	  \includegraphics[width=.49\linewidth]{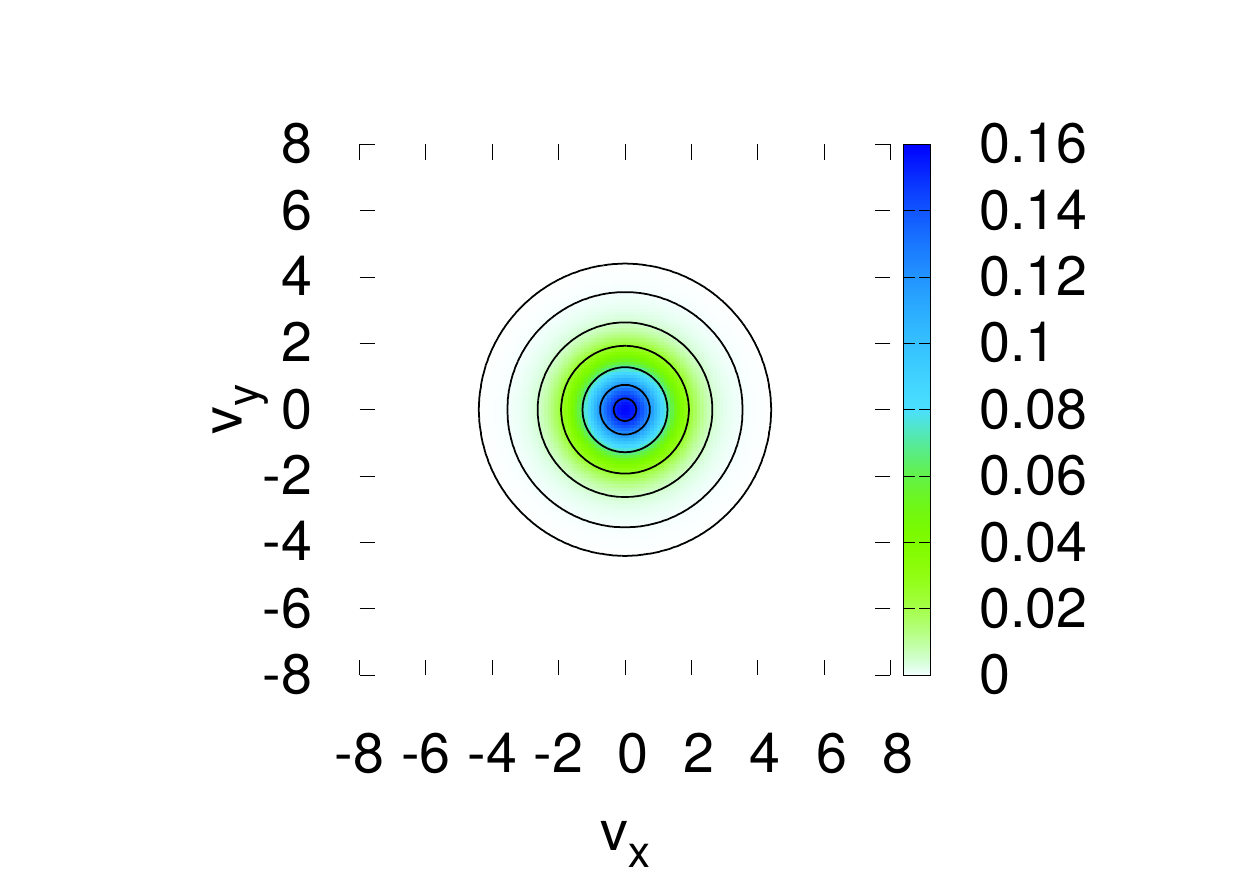} 
	  \\
	  $t=0.4$ & $t=\infty$
	\end{tabular}
	\caption{\label{fig:fp:3}{\bf Fokker-Planck equation with magnetic field.}  Time evolution of the distribution on the fine mesh $N=160^2$.}
      \end{center}
      
    \end{figure}
    
    \subsection{Polymer flow in a dilute solution}\label{test2}
    We investigate the numerical approximation of the Fokker-Planck part of the kinetic
    Fokker-Planck equation for polymers \cite{masmoudi_2008_well} 
    $$
    \left\{
    \begin{array}{l}
      \displaystyle \frac{\partial F}{\partial t} =
      -\nabla_\bk\cdot\left[\left(\bA\,\bk -
      \frac{1}{2}\nabla_\bk\Pi(\bk)\right) F -  \frac{1}{2}\nabla_\bk F\right].
      \\
      \;
      \\
      \ds F(t=0) = F_0 \quad{\rm in}\quad \Omega\subset\RR^3,
    \end{array}
    \right.
    $$
    where the matrix $\bA$ represents the gradient of an external velocity
    field and  is given by
    $$
    \bA = \left( \begin{array}{lll} 1/4 & -1/2 & 0 \\ 1/2 & -1/4 & 0 \\ 0
      & 0 & 0\end{array}\right)\,.
    $$
    The domain is $\Omega=(-4,4)^3$ and we choose the Hookean model $\Pi(\bk)= |\bk|^2/2$. The initial datum  $F_0$ is given by the sum of two Gaussian distributions
    $$
    F_0(\bk) = \frac{1}{2\, (2\pi)^{3/2}}\left[ \exp\left(-\frac{|\bk-\bk_1|^2}{2}\right)  \,+\, \exp\left(-\frac{|\bk-\bk_2|^2}{2}\right) \right],
    $$
    with $\bk_1=(-3/2,1,0)$ and $\bk_2=(1,-3/2,0)$. This
    equation is supplemented
    with  homogeneous Neumann boundary conditions such that global mass is
    conserved. For numerical simulations we choose various meshes from
    $N=24^3$ to $64^3$ points with $\Delta t = 0.01$ using a time implicit scheme until $T=5$. In this case, the steady state is not known, hence the
    steady equation is first solved numerically to compute a
    consistent approximation of the equilibrium $(f_K^\infty)_{K\in\T}$  and  the stationary flux $\F[K]^\infty$.
    If the matrix $\bA$ were equal to zero we would expect the steady state to be close to a Gaussian
    $
    G(\bk) = \exp(-\Pi(\bk))\,,
    $
    at least far from edges since $G$ do not satisfy the boundary conditions. Thus we can expect $\finf/G$ to be close to some constant in the domain, which seems easier to approximate numerically. Hence for all $K\in\T$ and $\sigma\in\Ei$, we define the quantities $G_K$ and $G_\sigma$ as the evaluation of $G$ at the center of the respective cell or edge, as well as
    \[
    h^\infty_K = \frac{\finf_K}{G_K}\,.
    \]
    The scheme is solved on this new unknown $h^\infty$. More precisely, it is given by \eqref{stat} with the flux
    \[
    \F[K]^\infty = 
    m(\sigma)\,G_\sigma\,\left[A_{K,\sigma}^{+}\,h^\infty_K - A_{K,\sigma}^{-}\,h^\infty_L-\frac{1}{2\,d_\sigma}\left(h^\infty_L-h^\infty_K\right)\right]\,,
    \]
    if $\sigma = K|L$ and $\F[K]^\infty = 0$ if $\sigma\in\Ee$. The quantity $A_{K,\sigma}$ is the evaluation of $(\bA\bk)\cdot\bn_{K,\sigma}$ at the center of the edge $\sigma$.  Besides, because of the conservative boundary conditions, we need to specify the mass of the steady state to be that of the initial data in order to get a unique solution to the scheme. Then we define the steady state on interior edges by $\finf_\sigma = (\finf_K+\finf_L)/2$ if $\sigma=K|L$. Observe that this scheme is consistent with the following equation
    \[
    {\nabla}\cdot\left[G(\bk)((\bA\bk)\,h^\infty -
    \frac 12{\nabla} h^\infty)\right] = 0\,,
    \]
    which is only a reformulation of the original steady equation.

    In Figure~ \ref{fig:poly:1}, we  represent the time evolution
    of the entropy $H_{\phi_2}$, its dissipation  $D_{\phi_2}$
    and the numerical dissipation due to the convective term
    $C_{\phi_2}$ in log scale.  First when $N\geq 32^3$ points, the
    entropy and the physical dissipation are well approximated compared to
    the solution computed with a fine mesh $N\geq 64^3$. Once again both of
    them are decreasing function of time and converge to zero with an
    exponential decay rate.  The numerical dissipation of
    the convective term $C_{\phi_2}$ is much smaller than the physical
    one and also converges to zero when times goes to infinity
    exponentially fast, hence it does not affect the accuracy on the decay
    rate for large time.

    \begin{figure}[H]
      \begin{center}
	\begin{tabular}{cc}
	  \includegraphics[width=.49\linewidth]{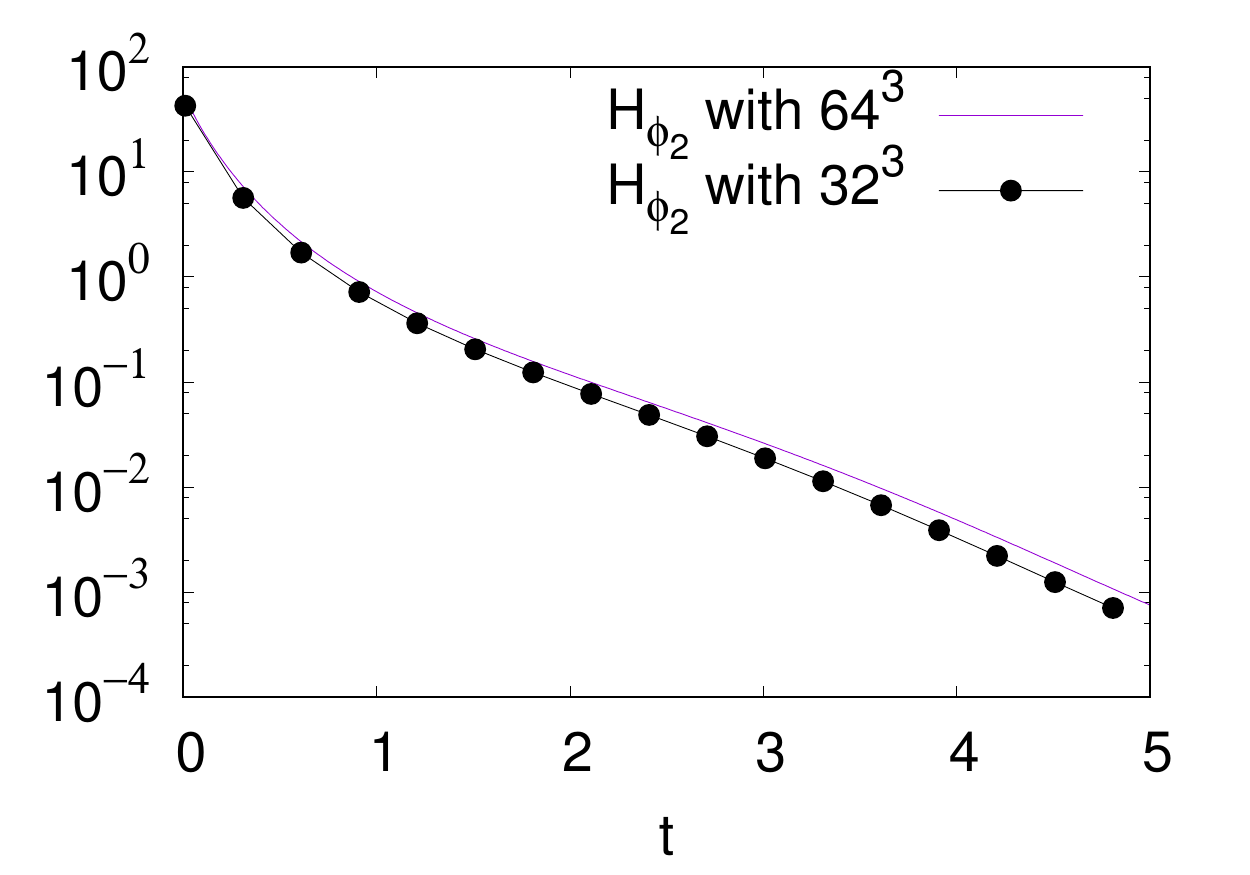} &   
	  \includegraphics[width=.49\linewidth]{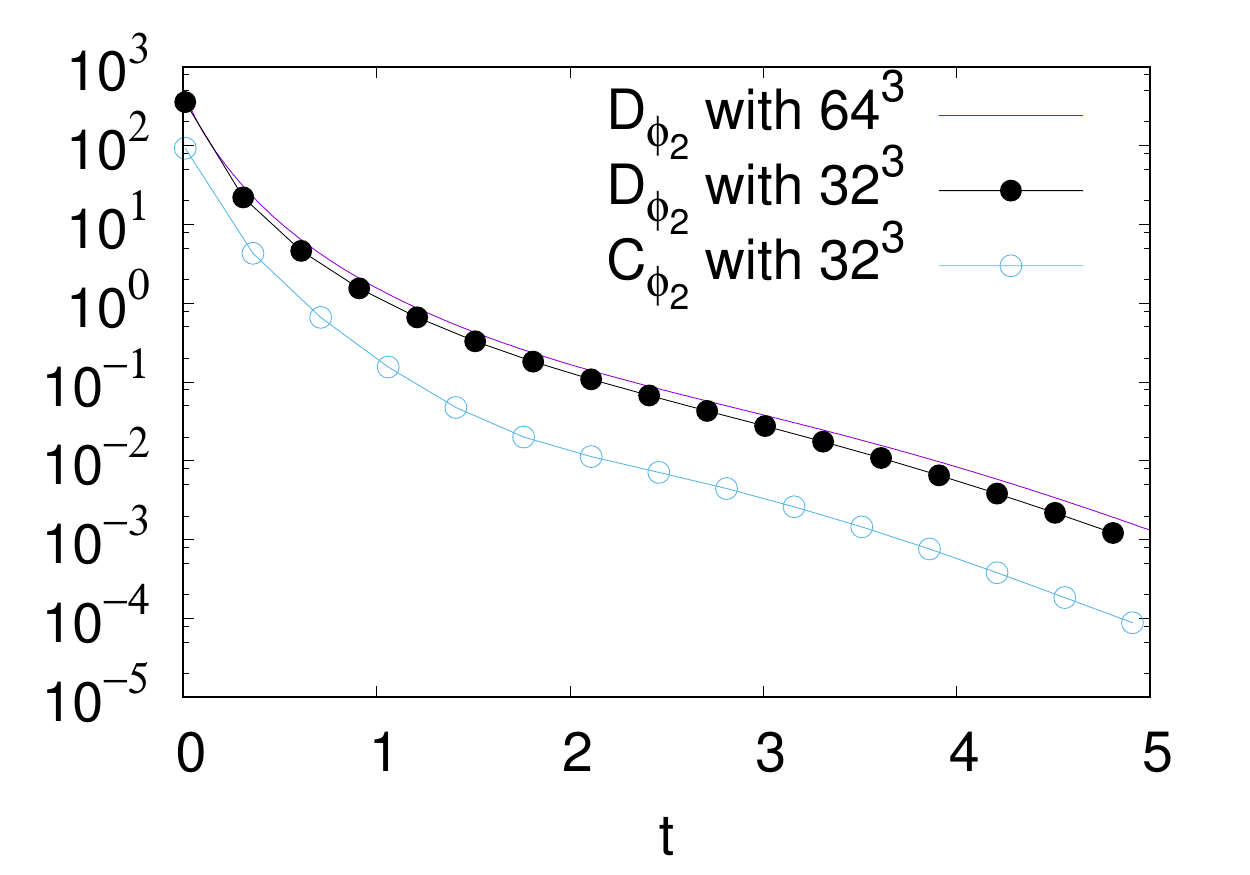}
	  \\
	  (a) & (b) 
	\end{tabular}
	\caption{\label{fig:poly:1}{\bf Polymer flow in a dilute solution.}  Time
	  evolution of the $2$-entropy $H_{\phi_2}$ and the corresponding physical dissipation 
	  and  numerical dissipation $(D_{\phi_2},C_{\phi_2})$  with $N=32^3$ mesh points.}
      \end{center}
    \end{figure}
    
    Finally, in Figure \ref{fig:poly:2}, we set forth the
    time evolution of the distribution function at different time. The
    first column represents an isovalue $f(t,\bk)\equiv 0.02$ of the
    distribution function whereas the second column is a two dimensional
    projection in the plane $k_x-k_y$, the solution converges to the
    discrete steady state, which is consistent with the
    equilibrium.

    \begin{figure}[H]
      \begin{center}
	\begin{tabular}{cc}
	  \includegraphics[width=.49\linewidth]{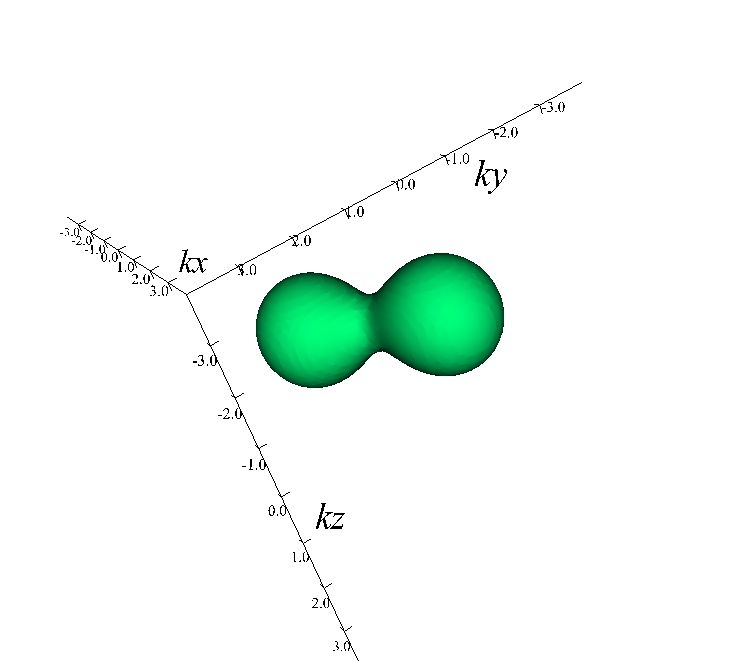} &
	  \includegraphics[width=.49\linewidth]{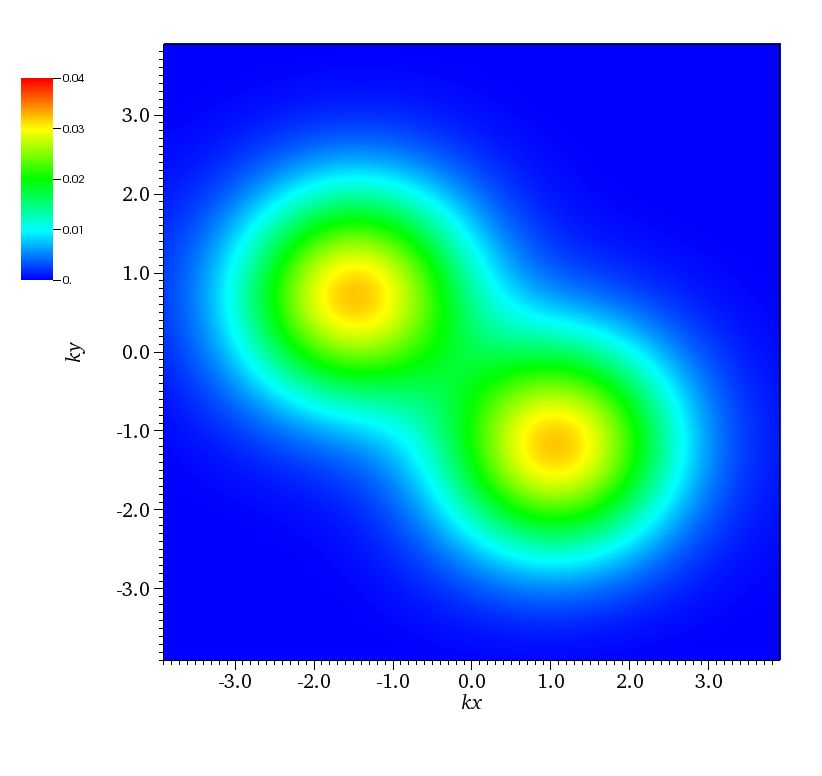}
	  \\    
	  \includegraphics[width=.49\linewidth]{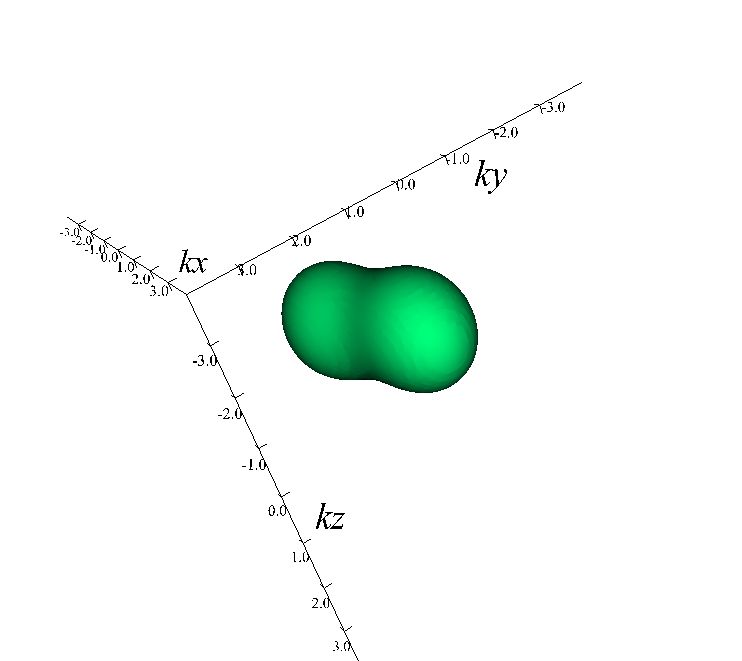} &
	  \includegraphics[width=.49\linewidth]{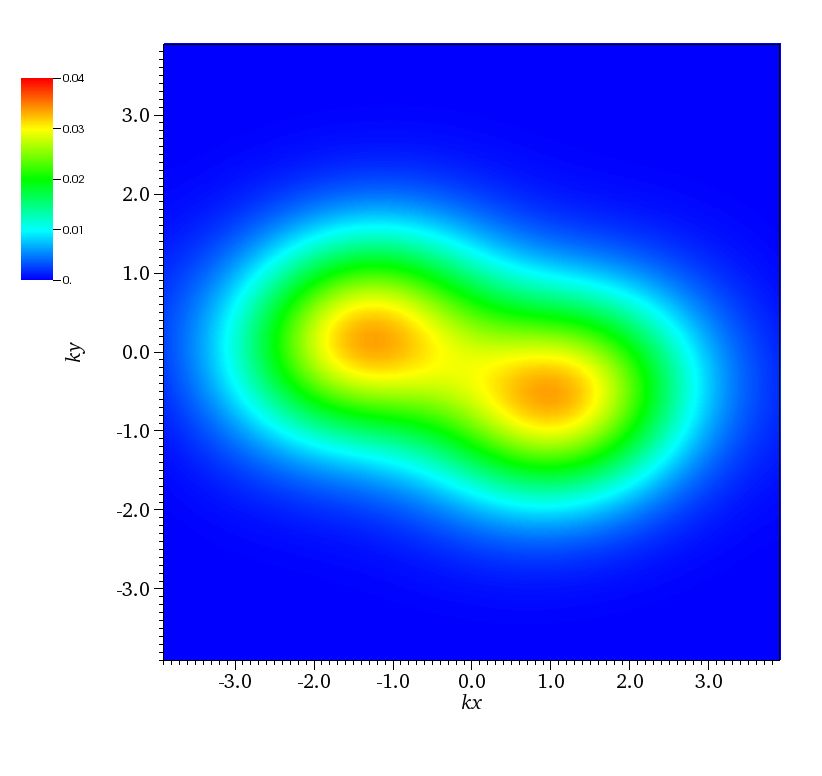} 
	  \\
	  \includegraphics[width=.49\linewidth]{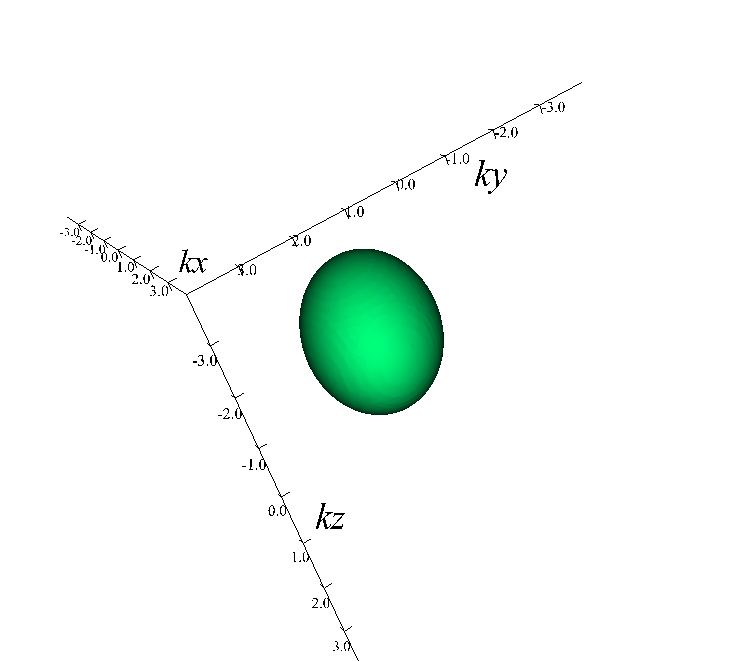} &
	  \includegraphics[width=.49\linewidth]{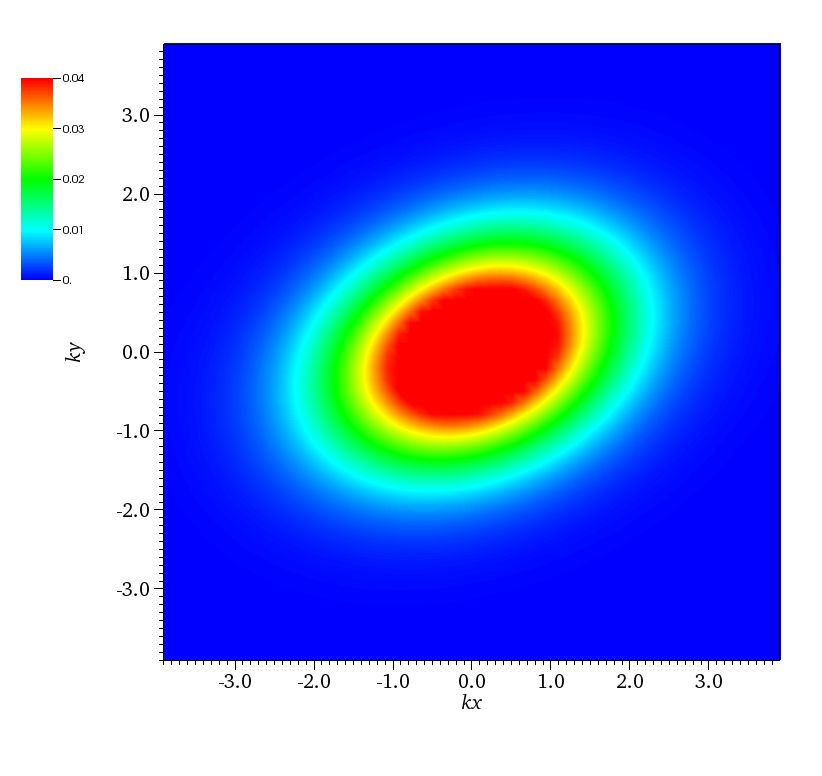}
	\end{tabular}
	\caption{\label{fig:poly:2}
	  {\bf Polymer flow in a dilute solution.} (a) one
	  isovalue $F(t,\bk)=0.02$ (b) $k_x-k_y$ projection of the distribution
	  in the $\bk$ space at time $t=0.2$,
	  $t=0.7$ and $t=5$. }
      \end{center}
    \end{figure}
    
    \subsection{Porous medium equation}\label{test3}
    We finally study the numerical approximation of the porous medium equation
    $$
    \left\{
    \begin{array}{l}
      \displaystyle \frac{\partial f}{\partial t} = \Delta f^m\,,
      \\
      \;
      \\
      \ds f(t=0) = f_0 \quad{\rm in}\, \Omega=(0,1)\times (-1,1)^2,
    \end{array}
    \right.
    $$
    with $m=2$ and $f_0\equiv 0$ together with the non-homogeneous Dirichlet boundary
    conditions
    $$
    f^b = 
    \left \{
    \begin{array}{ll} 
      2.5\,, & \textrm{ if } \,  x=1 \textrm{ and }    y^2 + z^2 \leq 1/8\,,
      \\
      \,
      \\
      1\,, &\textrm{ else}.
    \end{array}\right.
    $$
    This model is nonlinear, with $\eta(s) = s^m$, and without convective terms. Moreover, since $\eta'(0)=0$ the equation is degenerate. As a consequence, with our choice of initial data and boundary conditions, we expect the solution to be equal to zero on some subset of $\omega\subset\Omega$, having non-zero measure, for some  positive time. 
    
    Since the steady state is not known, we first compute a numerical approximation $(f_K^\infty)_{K\in\T}$
    and the corresponding  stationary flux $\F[K]^\infty$ using the scheme given in Example~\ref{exampFlux}, which amounts to solving \eqref{stat} with
    \[
    \F[K]^\infty\ =\ -\tau_\sigma\,D_{K,\sigma}\eta(\finf)\,.
    \]
    Then from the knowledge of $\eta(\finf)\in\X[\eta(f^b)]^{N_T+1}$, we set $\finf_K = \eta^{-1}(\eta(\finf)_K)$ on each cell $K\in\T$ and $\eta(\finf)_\sigma = (\eta(\finf)_K + \eta(\finf)_L)/2$ for interior edges $\sigma = K|L$.
    
    For the numerical simulations we choose two Cartesian meshes with
    $N=30^3$ and $60^3$ points using the explicit scheme \eqref{schemeexp}-\eqref{dissflux}. Hence
    the time step now satisfies a CFL condition $\Delta t = O(\Delta
    x^2)$ (precisely defined in Theorem \ref{mainexp}) which is satisfied for the two meshes with
    $\Delta t = 10^{-5}$. In Figure \ref{fig:porous:1}, we represent the
    time evolution of the relative entropy and the physical and numerical
    dissipation until  final time $T=0.5$. These results are in good agreement with
    those obtained using a finer mesh and the numerical dissipation is
    several orders of magnitude smaller than the physical
    dissipation. 
    \begin{figure}[H]
      \begin{center} 
	\begin{tabular}{cc}
	  \includegraphics[width=.49\linewidth]{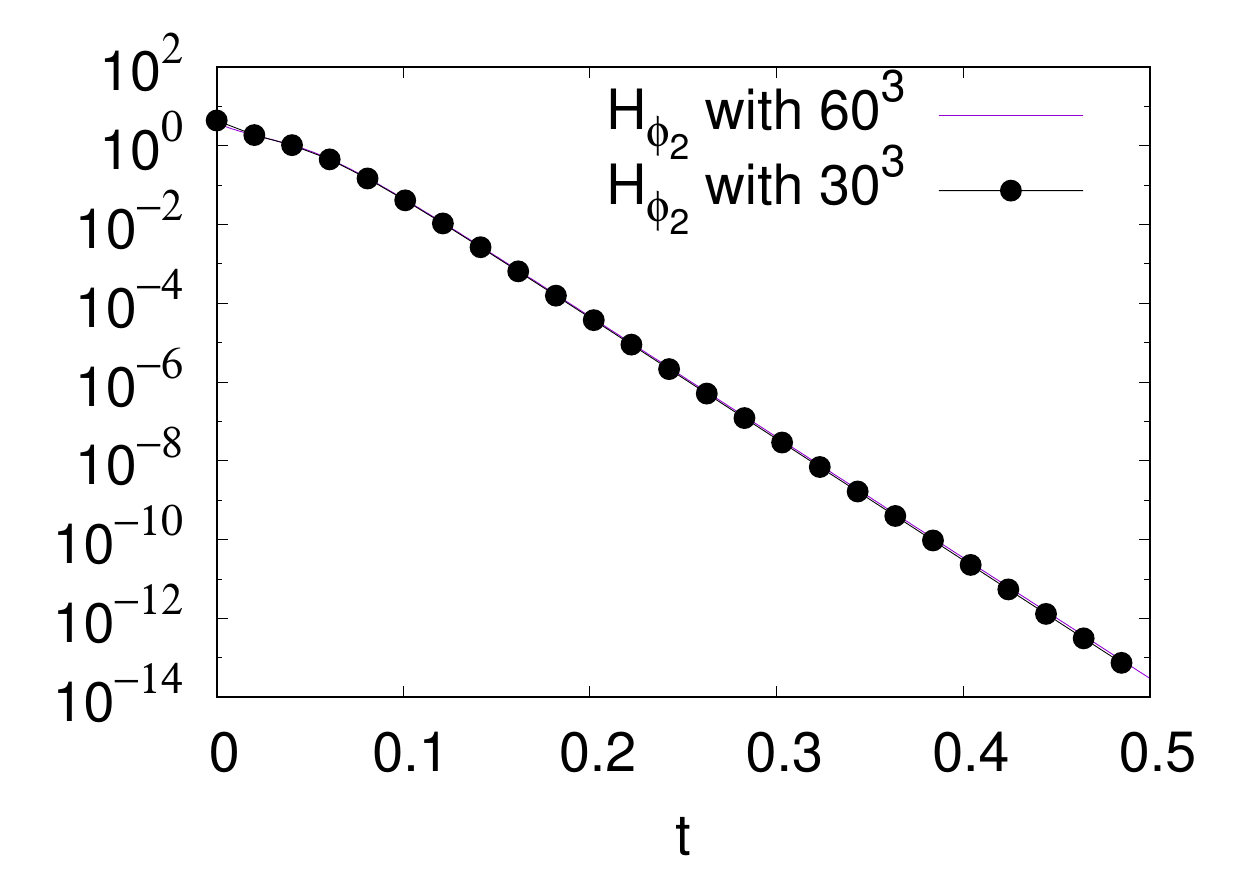} &   
	  \includegraphics[width=.49\linewidth]{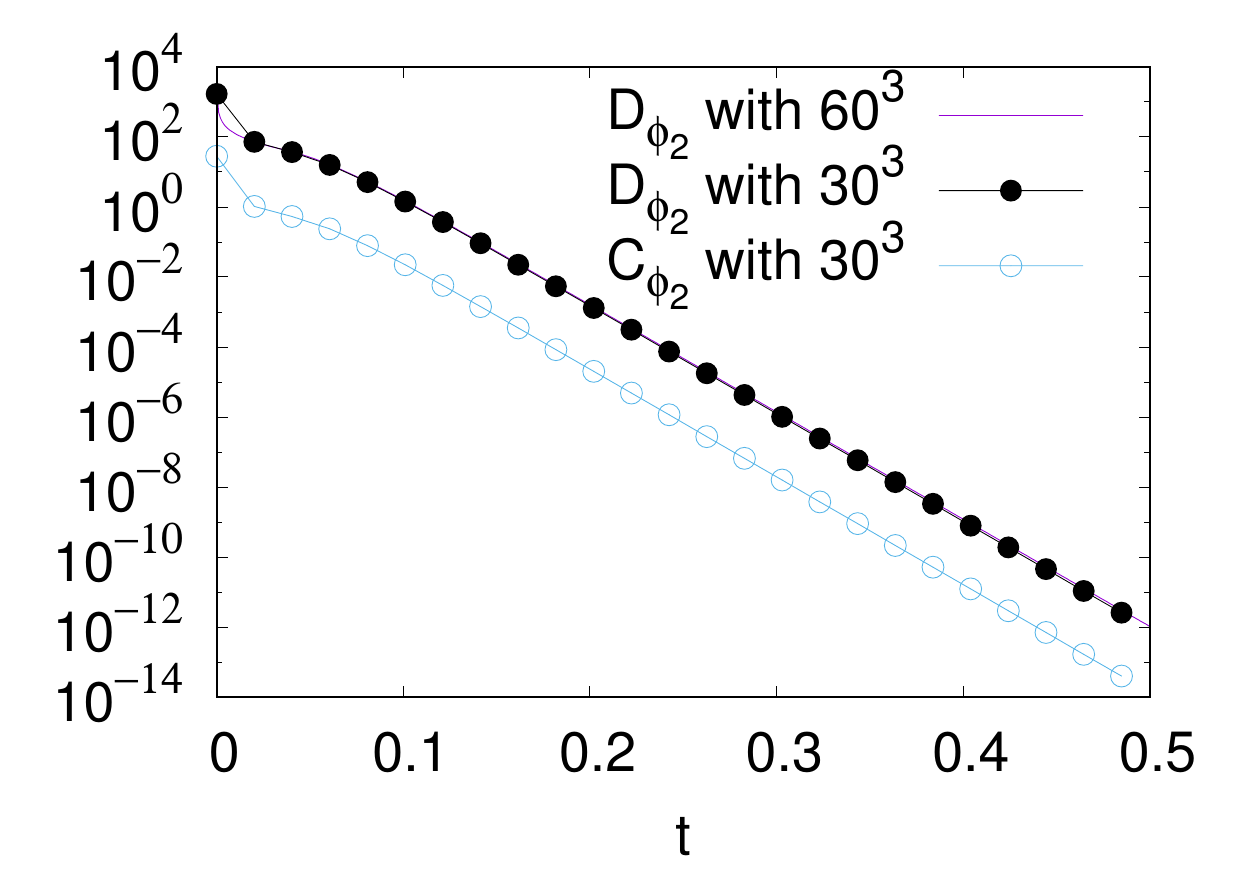}
	  \\
	  (a) & (b)
	\end{tabular}
	\caption{\label{fig:porous:1}
	  {\bf Porous medium equation.}  Time
	  evolution of the $2$-entropy $H_{\phi_2}$ and the corresponding physical dissipation 
	  and  numerical dissipation $(D_{\phi_2},C_{\phi_2})$  with   $N=30^3$ mesh points.}
      \end{center}
    \end{figure}
    
    Finally in Figure \ref{fig:porous:2} we represent the intersection of the graph of the  discrete solution with the plane $z=0$ at different time. The black line represents the isovalue $10^{-16}$ that surrounds the zone where the diffusion has not yet happened, namely where the distribution is null. This illustrates the good behavior of our scheme with respect to the degeneracy of the porous medium equation.
    
    \begin{figure}[H]
      \begin{center}
	\begin{tabular}{ccc}
	  \includegraphics[width=.33\linewidth]{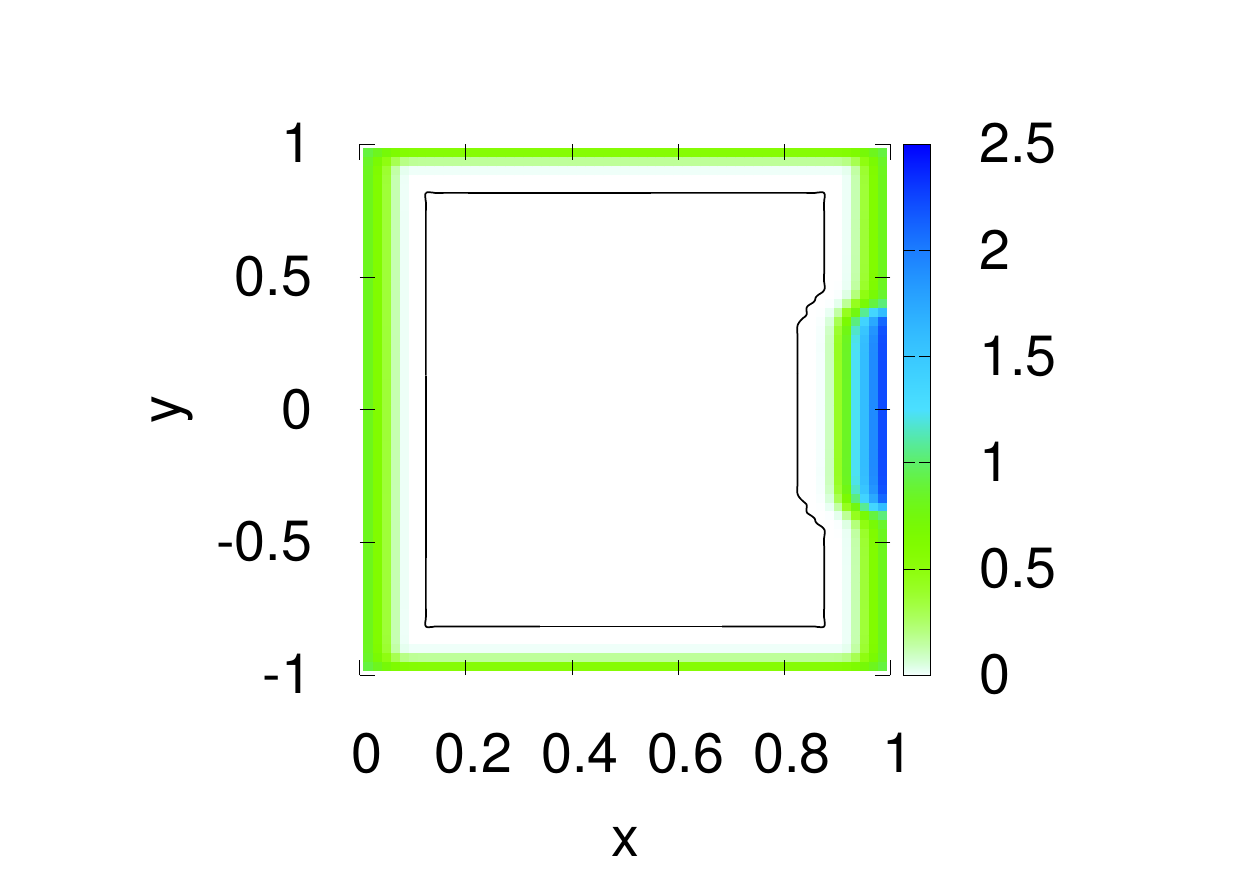} &   
	  \includegraphics[width=.33\linewidth]{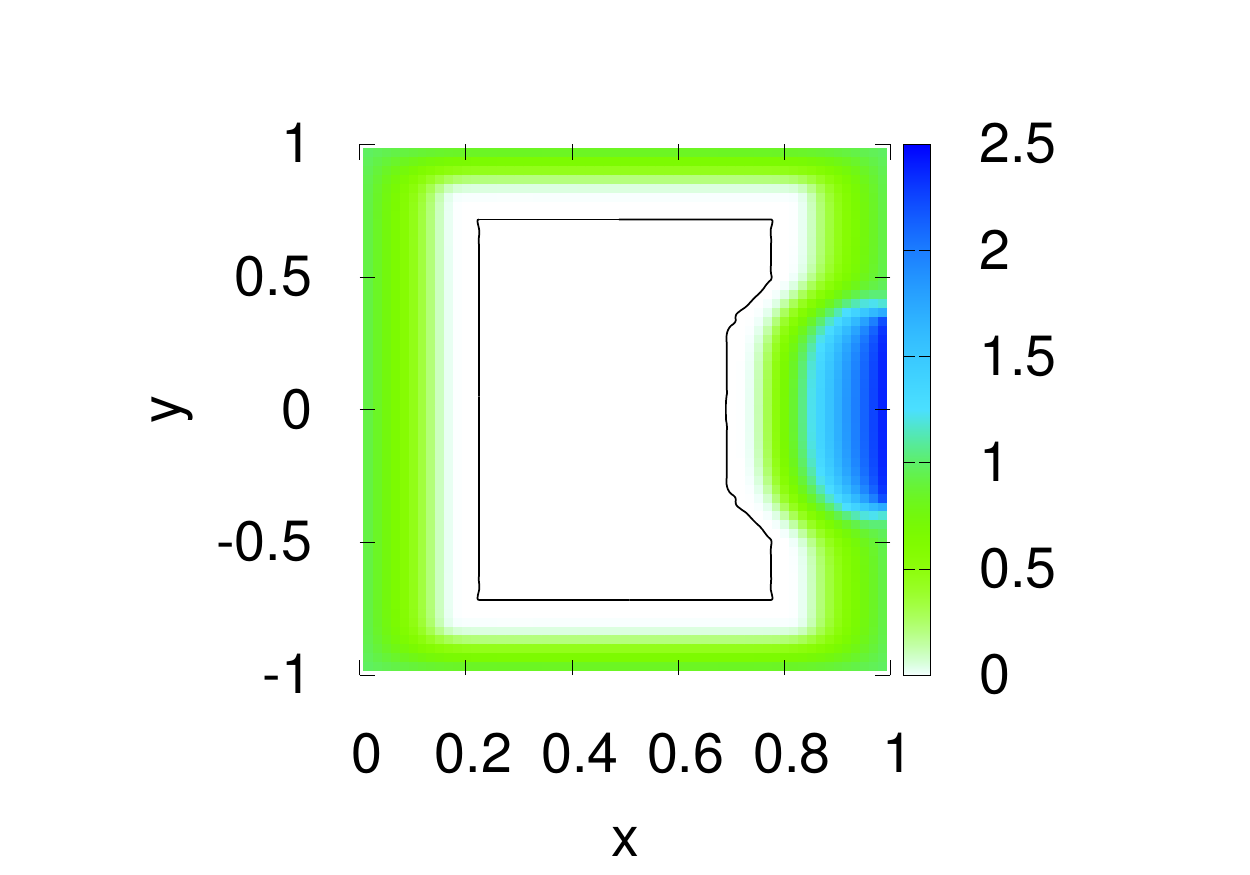} &
	  \includegraphics[width=.33\linewidth]{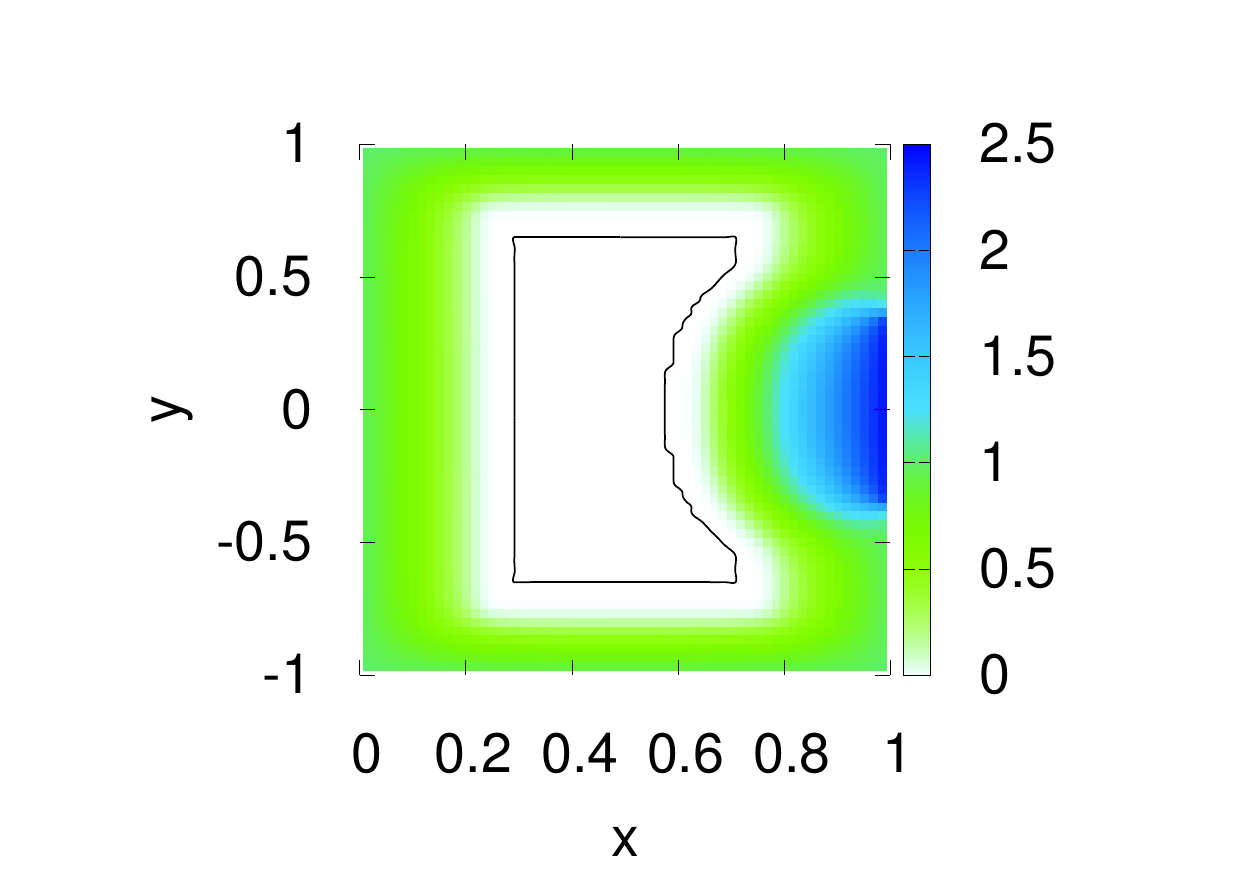} \\
	  $t=0.001$ & $t=0.005$&$t=0.01$ \\
	  \includegraphics[width=.33\linewidth]{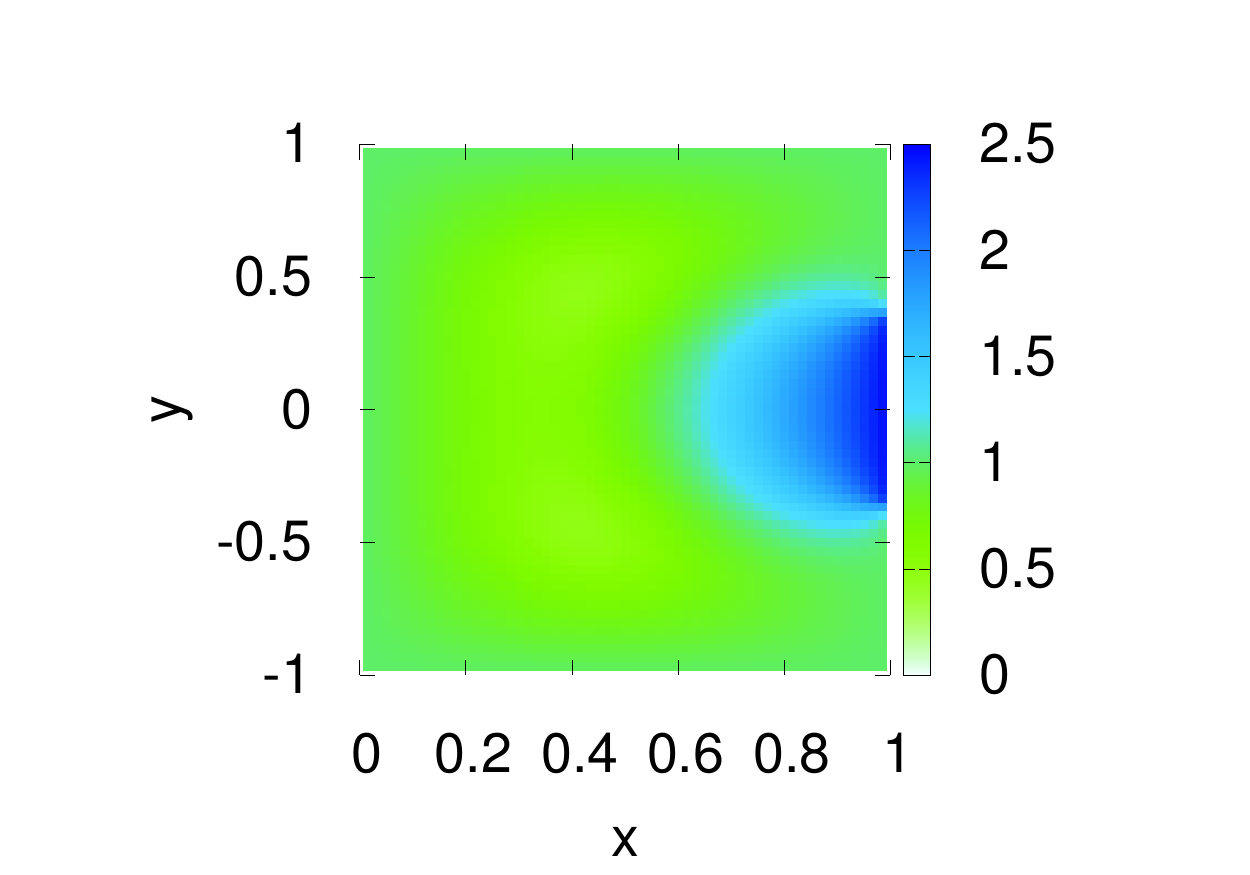} &   
	  \includegraphics[width=.33\linewidth]{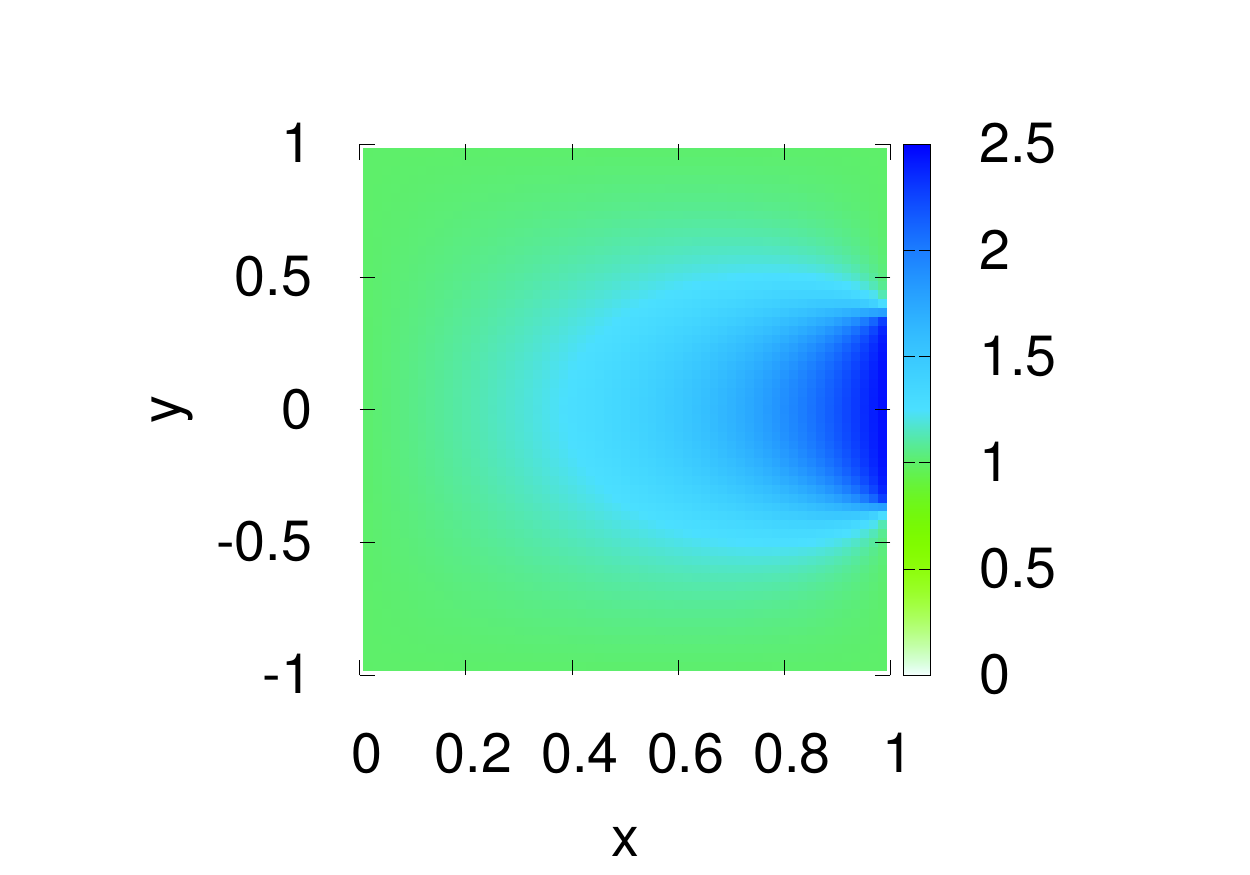} &
	  \includegraphics[width=.33\linewidth]{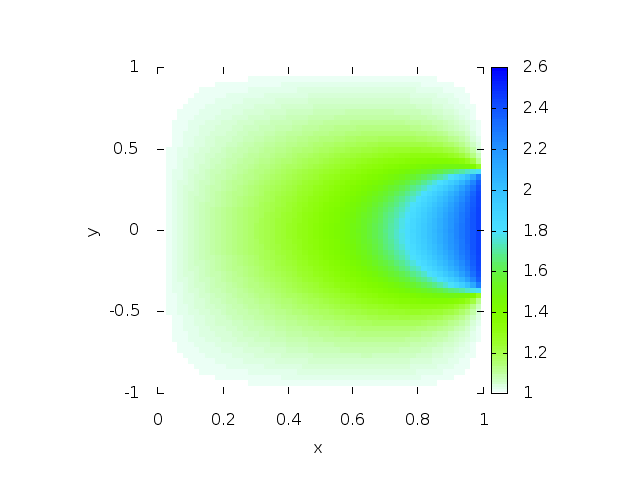} \\
	  $t=0.05$ & $t=0.5$ & $t=\infty$
	\end{tabular}
	\caption{\label{fig:porous:2}{\bf Porous medium equation.}  Time evolution of the distribution at $z=0$ for $N=60^3$. The black line is the isovalue $f(t,x,y,0)\equiv 10^{-16}$.}
      \end{center}
    \end{figure}
    
    \section{Comments and conclusion}
    
    In this paper, we have built a scheme for boundary-driven convection-diffusion equations
    that preserves the relative $\phi$-entropy structure of the model. After proving well-posedness, stability, exponential return to equilibrium and convergence,
    we gave several test cases 
    that confirm the satisfying long time behavior of the numerical scheme in different settings 
    (non-homogeneous Dirichlet/generalized Neumann boundary conditions, explicit and implicit time discretizations, linear and nonlinear models). 
    
    There are several directions that may be investigated for future work. The first objective is to generalize this scheme to anisotropic diffusions on possibly non-orthogonal meshes.
    It requires an adapted discretization of the gradient operator in every direction. 
    There are several papers \cite{eymard_2004_finite, eymard_2006_cell, eymard_2010_discretization, eymard_2012_small, cances_2016_convergence} of Eymard, Herbin, Gallouët, and Guichard and Cancès that are dealing with this type of problem. Their techniques are based on hybrid finite volume schemes for which the
    discrete gradient relies on the use of auxiliary unknowns located on edges between
    control volumes. However it is still unclear for the authors whether one can get the equivalent of the monotony properties of the 2-point flux used here, and that allows us to get the whole class of $\phi$-entropy inequalities.
    
    The spirit of our scheme is to start from a consistent discretization of the steady state and build the transient scheme upon the latter to ensure a satisfying behavior in the long-time asymptotic. We would like to adapt this kind of strategy to other types of numerical scheme (Discontinuous Galerkin, Finite elements, \emph{etc.}) and to different models. 
    
    Finally, as we saw in the introduction, many kinetic models (depending on space and velocity variables) such 
    the Vlasov-Fokker-Planck equation or the full dumbbell model for polymers write as the sum of a transport in space and a (convection)-diffusion in velocity. The second part of these models is treated in the present paper.
    While the diffusion operator is not coercive in all the variables, thanks to the phase space mixing properties of the transport operator, this still leads to an 
    entropy-diminishing behavior and a trend to a global equilibrium. This property is called hypocoercivity \cite{Villani_hypo} and its preservation by numerical schemes has never been studied to our knowledge and this would be another interesting extension of this work.

    \medskip
    \noindent \textsc{Acknowledgements.} The second author would like to thank Thierry Dumont for his kind help on sparse matrix routines. Both authors would like to acknowledge the anonymous reviewer for many constructive remarks that helped improving the quality of this paper.
    
    
    \bibliographystyle{plain}
    \bibliography{bibli}
  \end{document}